\documentclass[12pt]{amsart}

\usepackage[margin=1in]{geometry}
\usepackage{amsmath, amsthm, amssymb, amsfonts}
\usepackage{hyperref}
\usepackage{graphicx}
\graphicspath{{Images/}}
\usepackage{float}
\usepackage{tikz}
\usepackage{caption}
\usepackage{subcaption}
\usepackage{tikz-cd}
\usepackage{mathtools}
\usepackage{stmaryrd,wasysym,bm}
\usepackage{accents}

\title{Winding homology of knotoids}
\author{Deniz Kutluay}
\address{Department of Mathematics, Indiana University, Bloomington, IN 47405}
\email{dkutluay@indiana.edu}

\newtheorem{thm}{Theorem}[section]
\newtheorem{cor}[thm]{Corollary}
\newtheorem{lem}[thm]{Lemma}
\newtheorem{prop}[thm]{Proposition}
\newtheorem{ex}[thm]{Example}

\theoremstyle{definition}
\newtheorem{defn}[thm]{Definition}
\theoremstyle{remark}
\newtheorem{rem}[thm]{Remark}
\newtheorem{notation}[thm]{Notation}

\newtheorem{convention}[thm]{Convention}

\newcommand{\Q}{{\mathbb Q}}
\newcommand{\Z}{{\mathbb Z}}
\newcommand{\R}{{\mathbb R}}
\newcommand{\CP}{{\mathbb C}}
\newcommand{\A}{{\mathcal A}}
\newcommand{\C}{{\mathcal C}}
\newcommand{\V}{{\mathcal V}}
\newcommand{\NegCros}{{\includegraphics[scale=0.12]{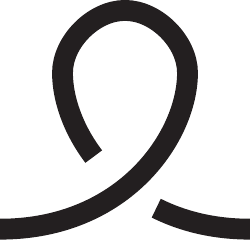}}}
\newcommand{\Arc}{{\includegraphics[scale=0.12]{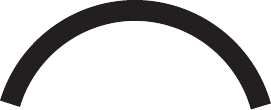}}}
\newcommand{\NoCirc}{{\includegraphics[scale=0.12]{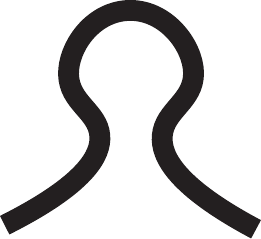}}}
\newcommand{\WithCirc}{{\includegraphics[scale=0.12]{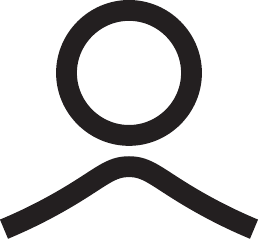}}}
\newcommand{\ArcSeg}{{\includegraphics[scale=0.12]{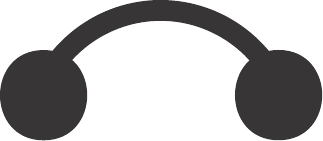}}}
\newcommand{\NoCircSeg}{{\includegraphics[scale=0.12]{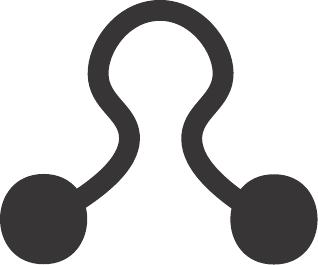}}}
\newcommand{\WithCircSeg}{{\includegraphics[scale=0.12]{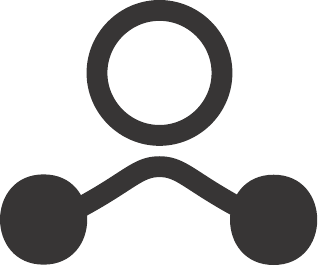}}}
\newcommand{\myfrac}[2]{\genfrac{}{}{0pt}{2}{#1}{#2}}
\newcommand{\RIIcros}{\vcenter{\hbox{\includegraphics[scale=0.18]{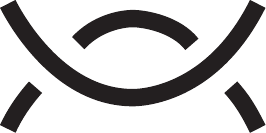}}}}
\newcommand{\RIInocros}{\vcenter{\hbox{\includegraphics[scale=0.17]{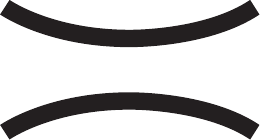}}}}
\newcommand{\LeftBump}{\vcenter{\hbox{\includegraphics[scale=0.15]{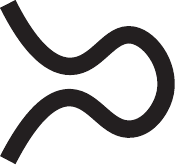}}}}
\newcommand{\RightBump}{\reflectbox{$\LeftBump$}}
\newcommand{\RightSoft}{\vcenter{\hbox{\includegraphics[scale=0.15]{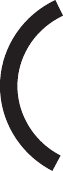}}}}
\newcommand{\LeftSoft}{\reflectbox{$\RightSoft$}}
\newcommand{\LeftBRightS}{\vcenter{\hbox{\includegraphics[scale=0.25]{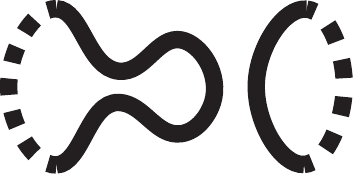}}}}
\newcommand{\SimpCirc}{\vcenter{\hbox{\includegraphics[scale=0.15]{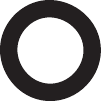}}}}
\newcommand{\TwoWaves}{{\vcenter{\hbox{\includegraphics[scale=0.15]{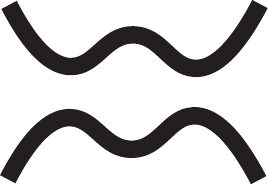}}}}}
\newcommand{\LeftSCircRightS}{\vcenter{\hbox{\includegraphics[scale=0.25]{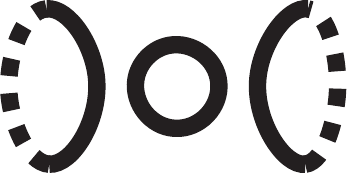}}}}

\newcommand{\WavesClosed}{\vcenter{\hbox{\includegraphics[scale=0.25]{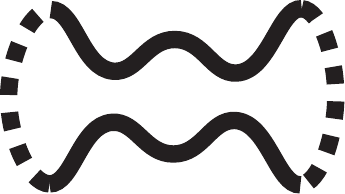}}}}
\newcommand{\CaseIILBRS}{\vcenter{\hbox{\includegraphics[scale=0.25]{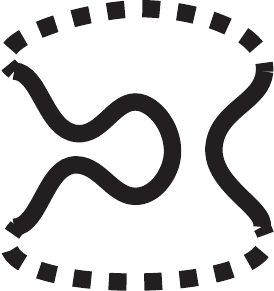}}}}
\newcommand{\CaseIILSCRS}{\vcenter{\hbox{\includegraphics[scale=0.25]{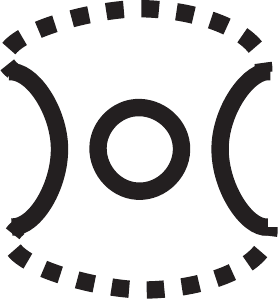}}}}
\newcommand{\CaseIITwoWaves}{\vcenter{\hbox{\includegraphics[scale=0.25]{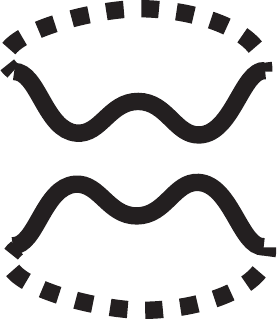}}}}
\newcommand{\CaseIIIoo}{\vcenter{\hbox{\includegraphics[scale=0.25]{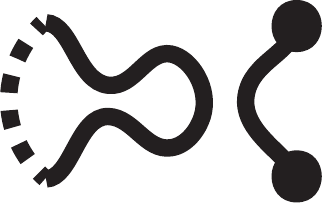}}}}
\newcommand{\CaseIIIio}{\vcenter{\hbox{\includegraphics[scale=0.25]{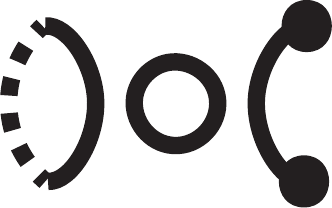}}}}
\newcommand{\CaseIIIoi}{\vcenter{\hbox{\includegraphics[scale=0.25]{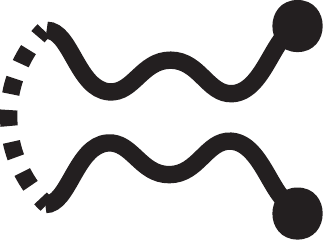}}}}
\newcommand{\CaseIIIii}{\vcenter{\hbox{\includegraphics[scale=0.25]{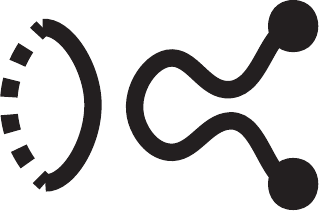}}}}
\newcommand{\RightSoftDot}{\vcenter{\hbox{\includegraphics[scale=0.15]{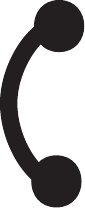}}}}
\newcommand{\RightBumpDot}{\vcenter{\hbox{\includegraphics[scale=0.15]{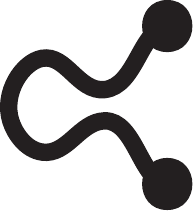}}}}
\newcommand{\TwoWavesDot}{\vcenter{\hbox{\includegraphics[scale=0.15]{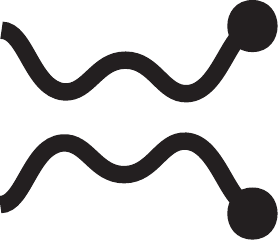}}}}
\newcommand{\CaseIVoo}{\vcenter{\hbox{\includegraphics[scale=0.25]{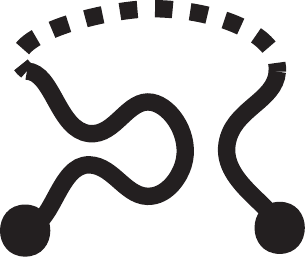}}}}
\newcommand{\CaseIVio}{\vcenter{\hbox{\includegraphics[scale=0.25]{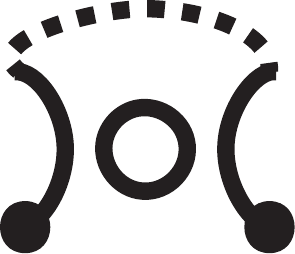}}}}
\newcommand{\CaseIVoi}{\vcenter{\hbox{\includegraphics[scale=0.25]{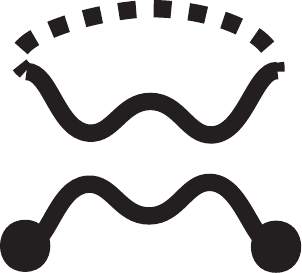}}}}
\newcommand{\CaseIVii}{\vcenter{\hbox{\includegraphics[scale=0.25]{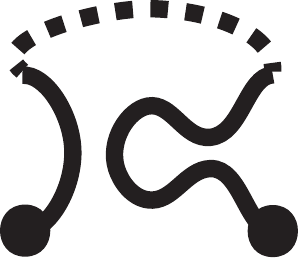}}}}
\newcommand{\LeftBumpSubDot}{\vcenter{\hbox{\includegraphics[scale=0.15]{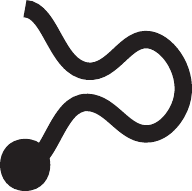}}}}
\newcommand{\LeftSoftSubDot}{\vcenter{\hbox{\includegraphics[scale=0.15]{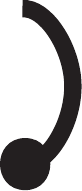}}}}
\newcommand{\TwoWavesSubDot}{\vcenter{\hbox{\includegraphics[scale=0.15]{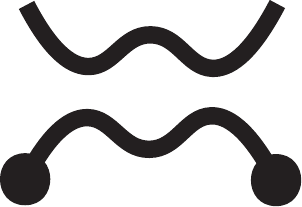}}}}
\newcommand{\RIInocrosDot}{\vcenter{\hbox{\includegraphics[scale=0.17]{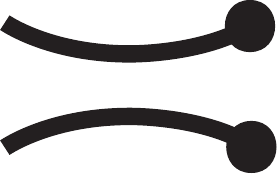}}}}
\newcommand{\RIInocrosSubDot}{\vcenter{\hbox{\includegraphics[scale=0.17]{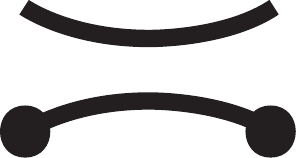}}}}
\newcommand{\CaseVoo}{\vcenter{\hbox{\includegraphics[scale=0.25]{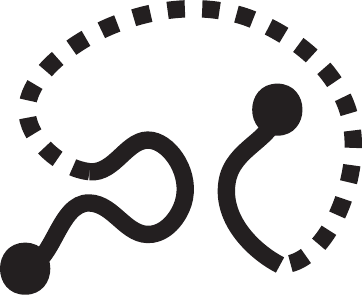}}}}
\newcommand{\CaseVio}{\vcenter{\hbox{\includegraphics[scale=0.25]{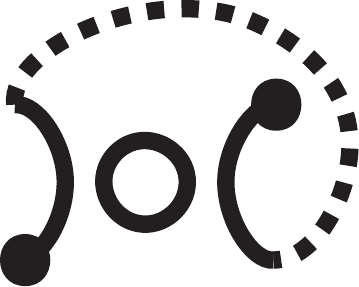}}}}
\newcommand{\CaseVoi}{\vcenter{\hbox{\includegraphics[scale=0.25]{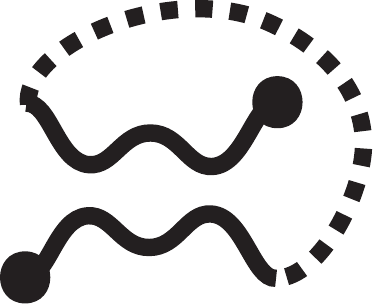}}}}
\newcommand{\CaseVii}{\vcenter{\hbox{\includegraphics[scale=0.25]{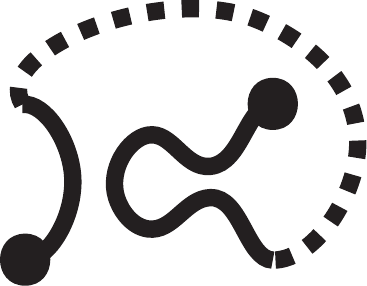}}}}
\newcommand{\TwoWavesDiagDot}{\vcenter{\hbox{\includegraphics[scale=0.15]{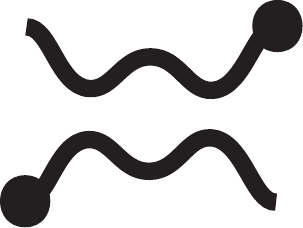}}}}
\newcommand{\RIInocrosDiagDot}{\vcenter{\hbox{\includegraphics[scale=0.17]{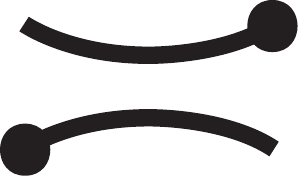}}}}

\newcommand{\Dooo}{\vcenter{\hbox{\includegraphics[scale=0.05]{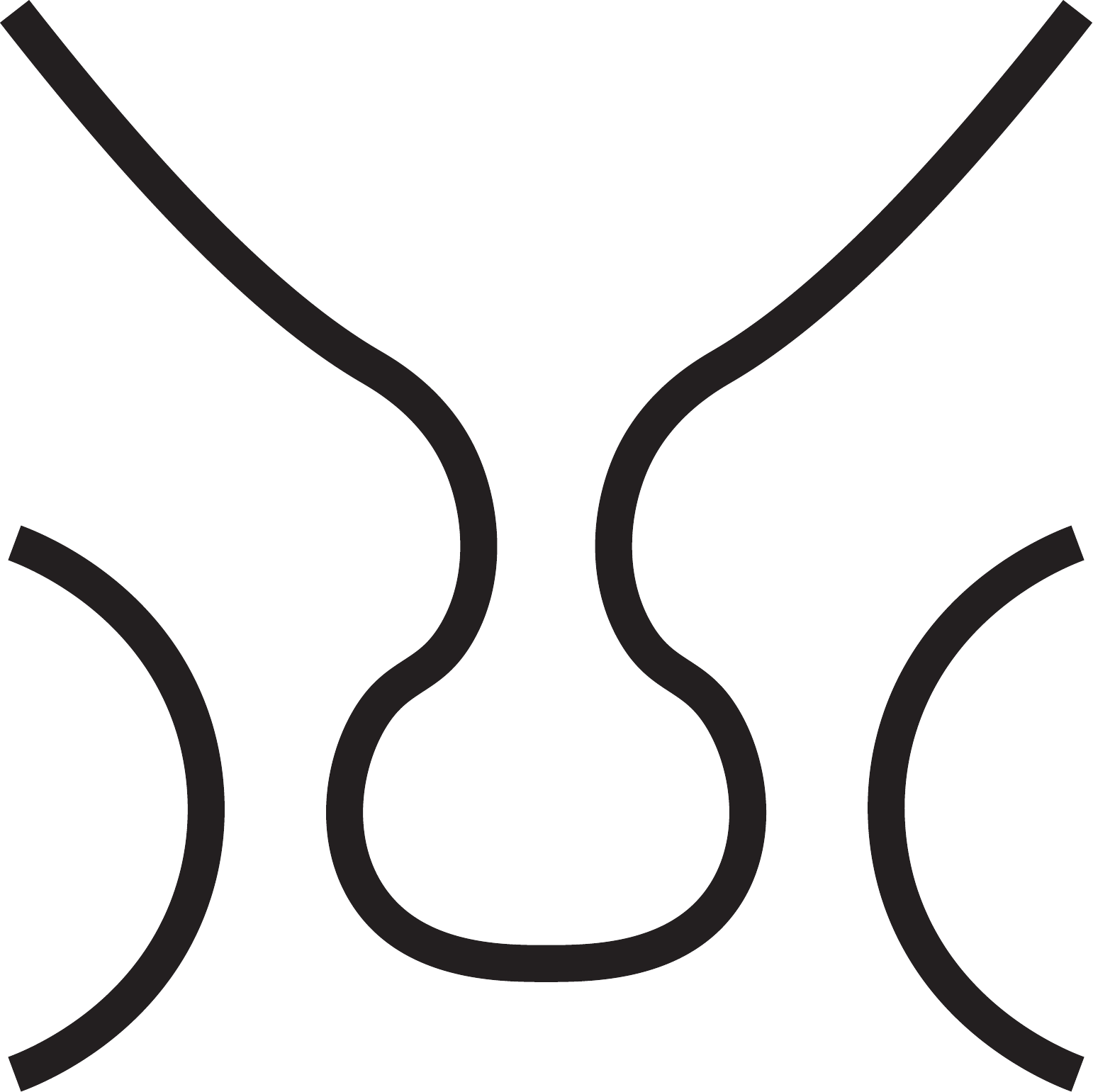}}}}
\newcommand{\Dioo}{\vcenter{\hbox{\includegraphics[scale=0.05]{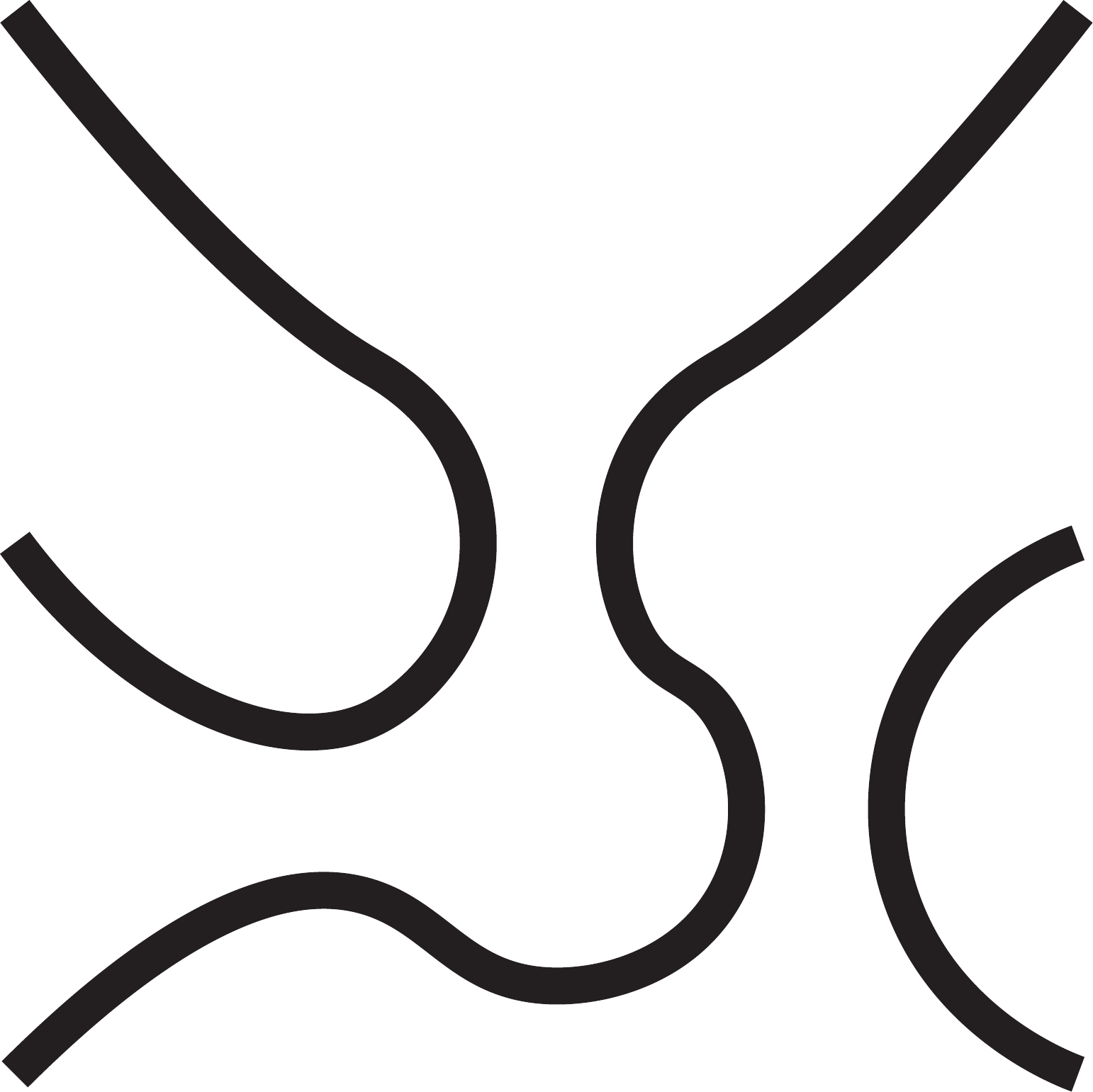}}}}
\newcommand{\Doio}{\vcenter{\hbox{\includegraphics[scale=0.05]{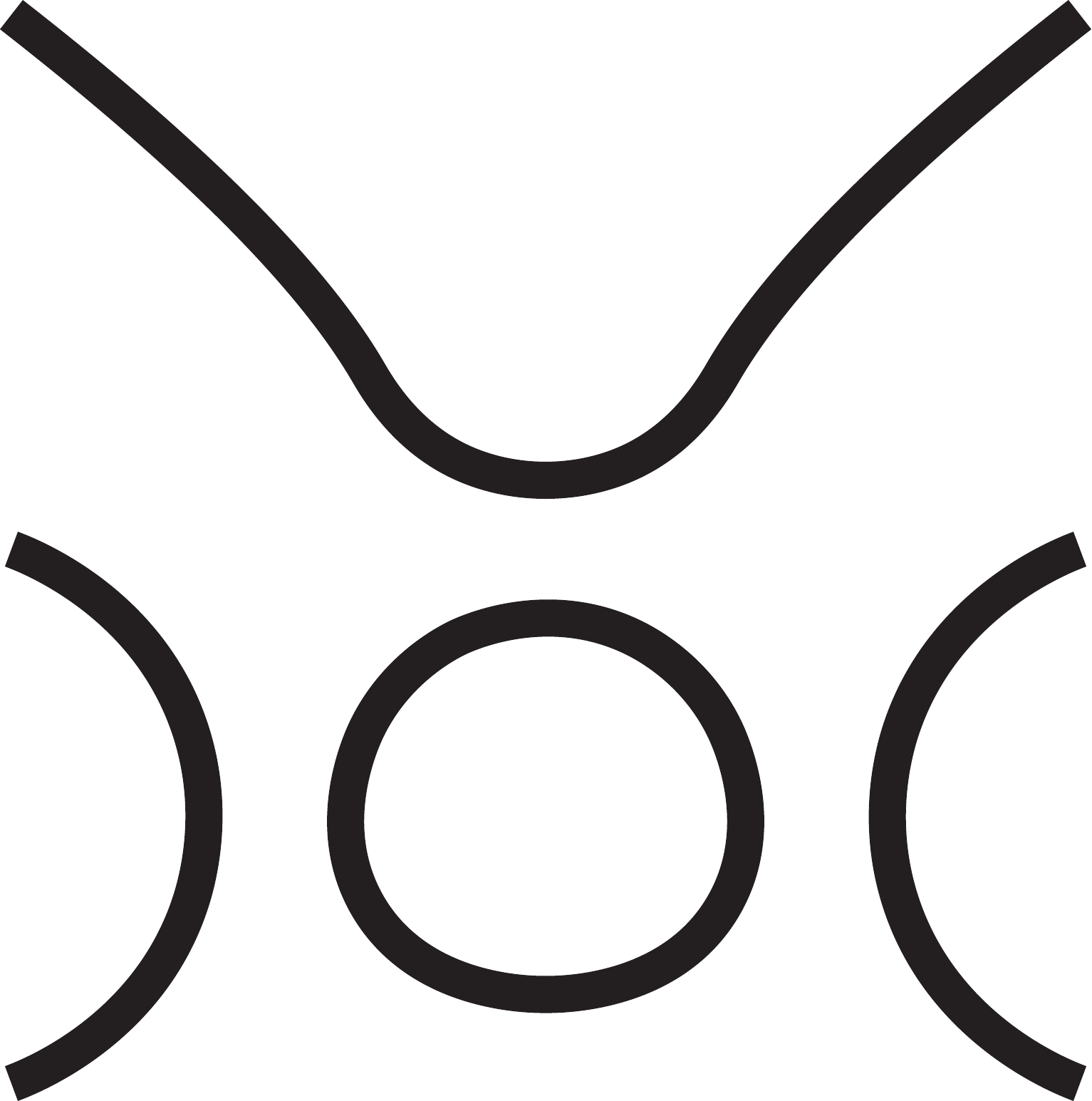}}}}
\newcommand{\Dooi}{\vcenter{\hbox{\includegraphics[scale=0.05]{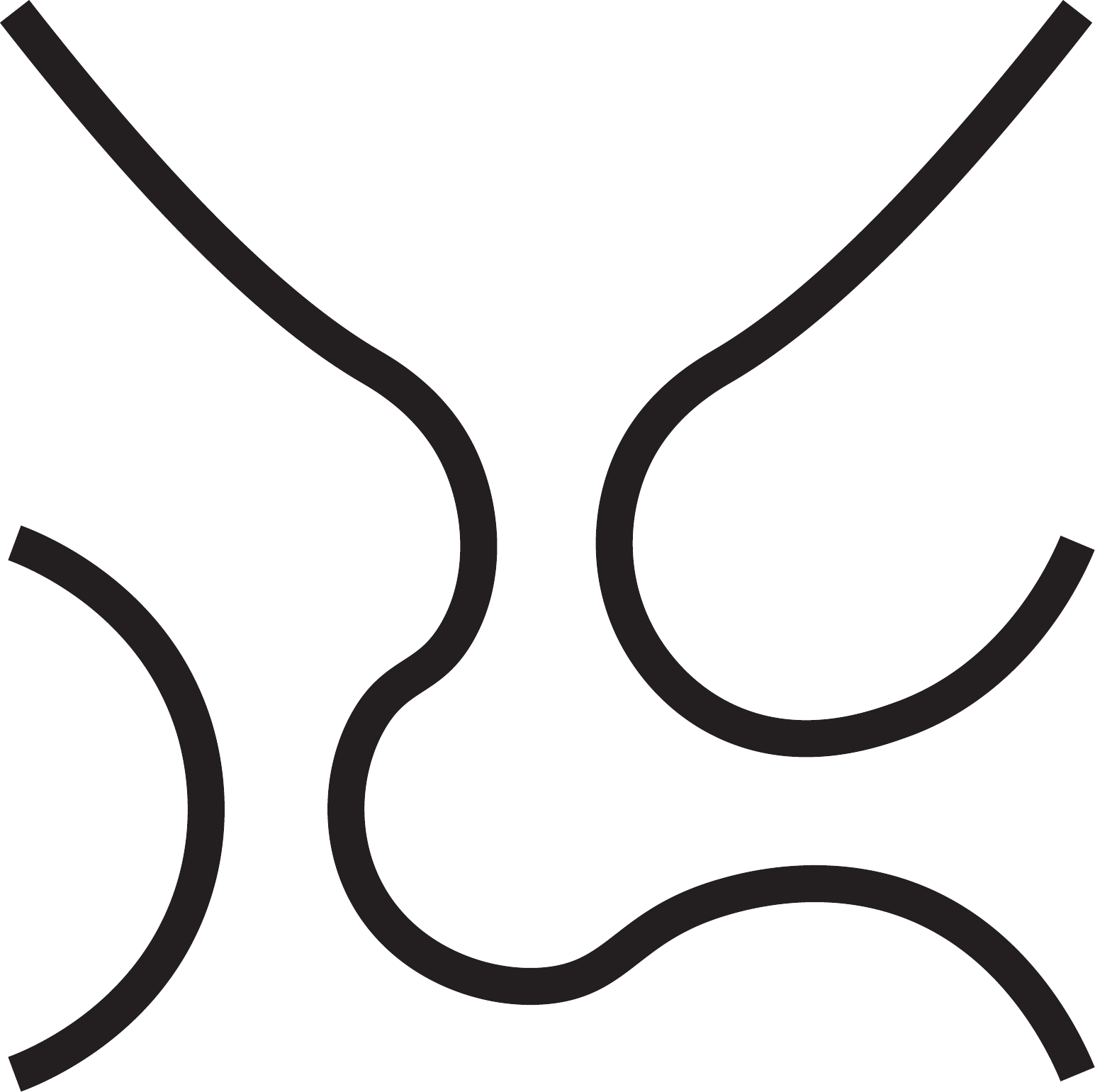}}}}
\newcommand{\Diio}{\vcenter{\hbox{\includegraphics[scale=0.05]{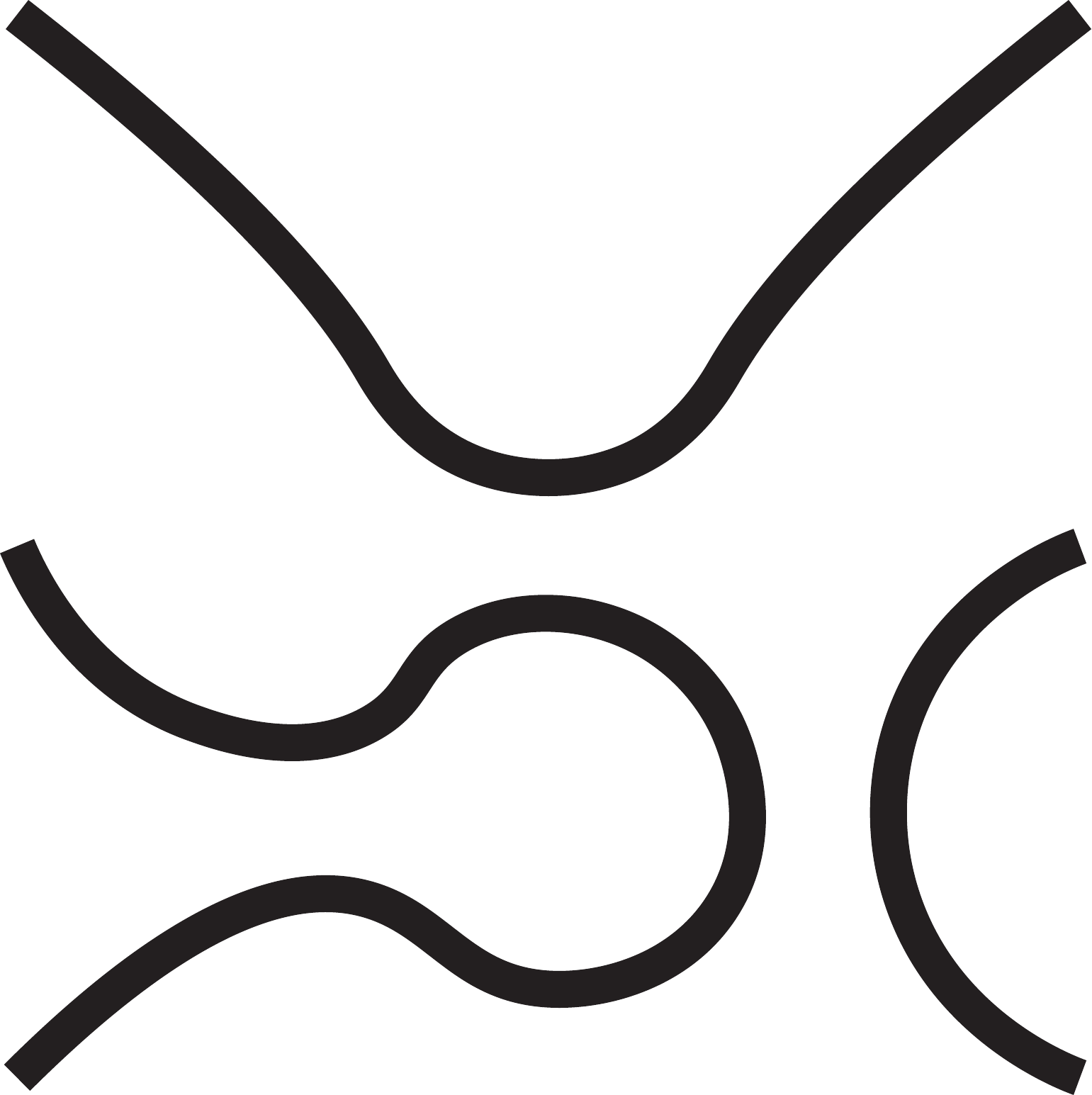}}}}
\newcommand{\Dioi}{\vcenter{\hbox{\includegraphics[scale=0.05]{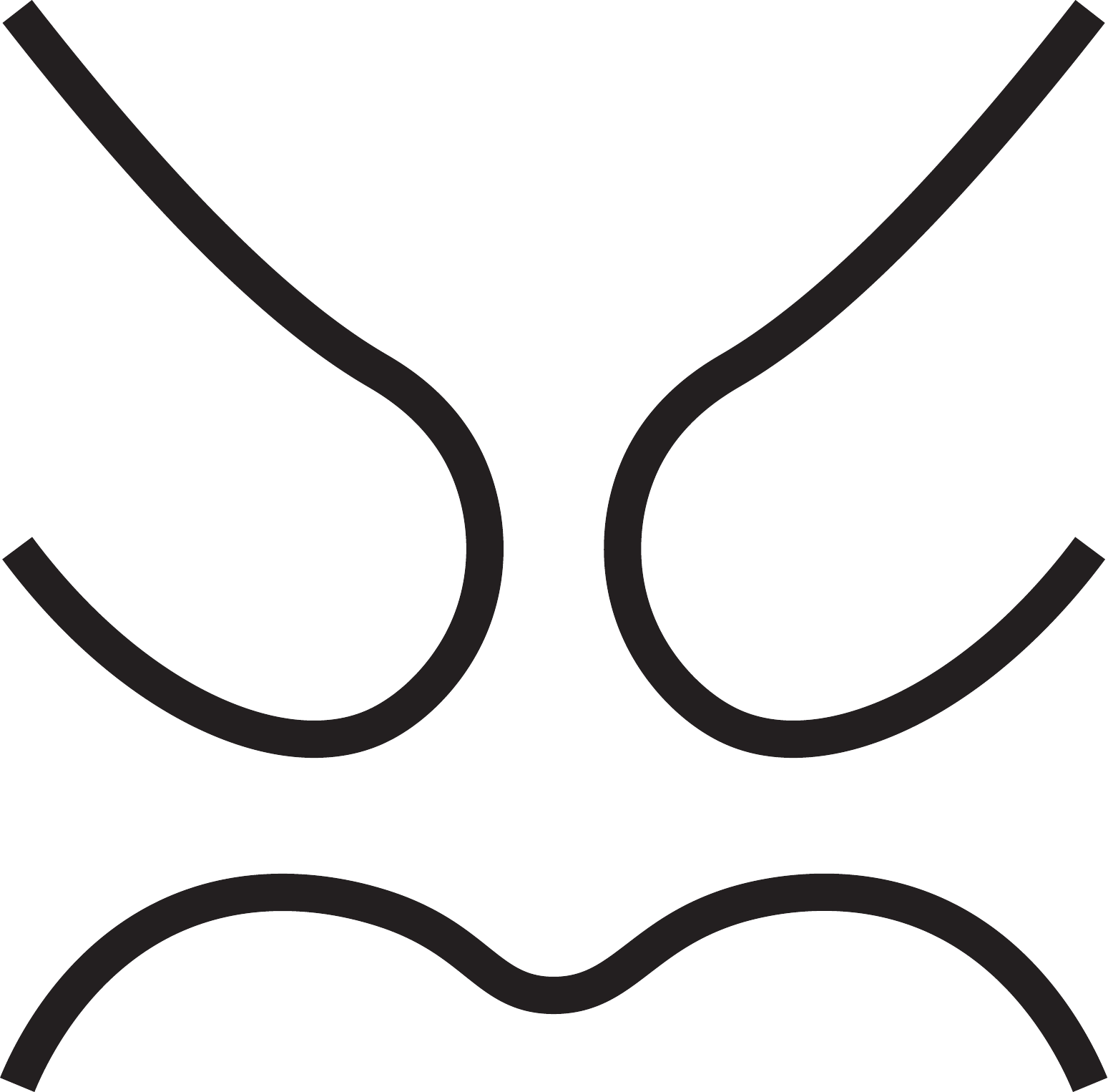}}}}
\newcommand{\Doii}{\vcenter{\hbox{\includegraphics[scale=0.05]{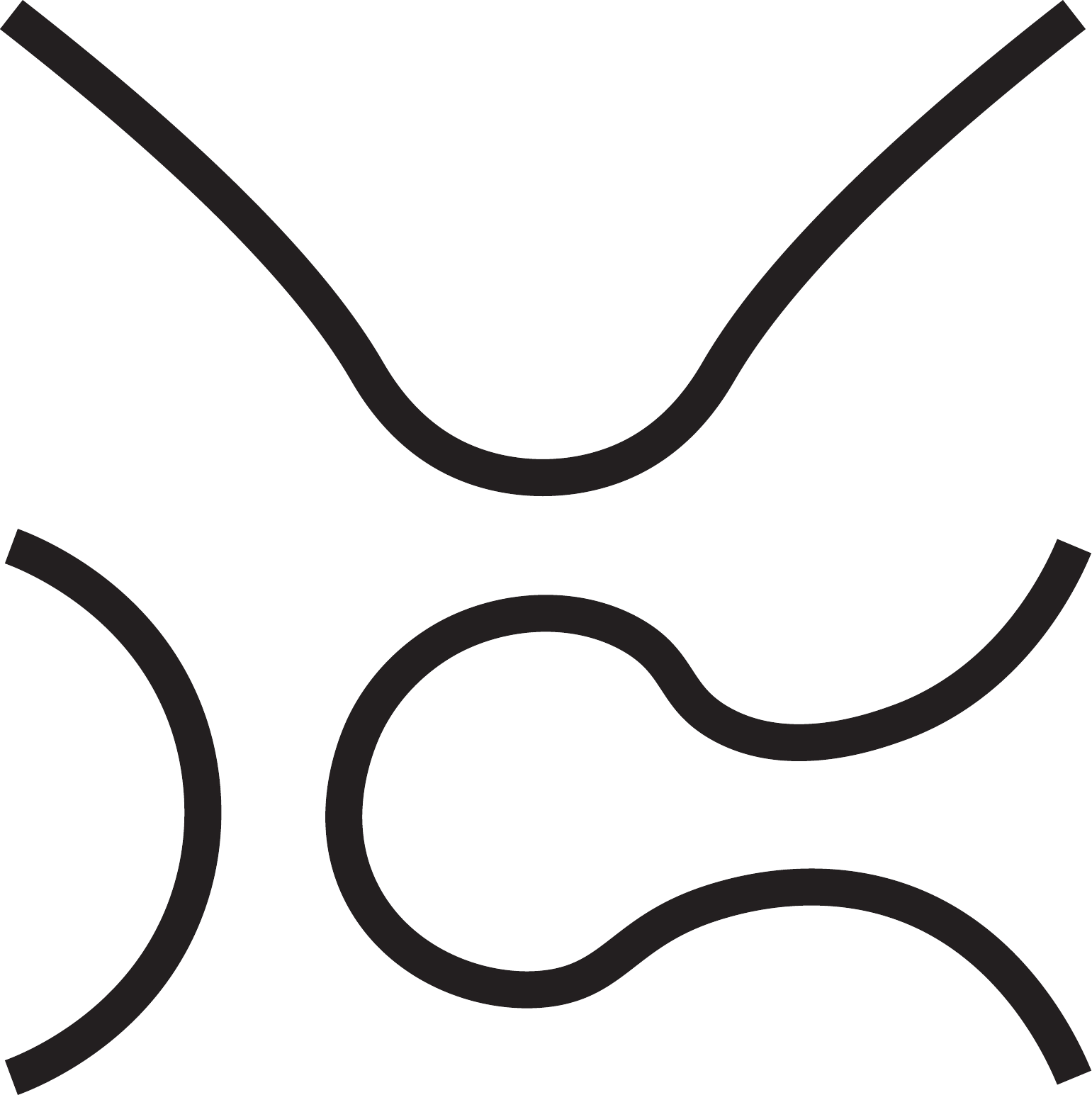}}}}
\newcommand{\Diii}{\vcenter{\hbox{\includegraphics[scale=0.05]{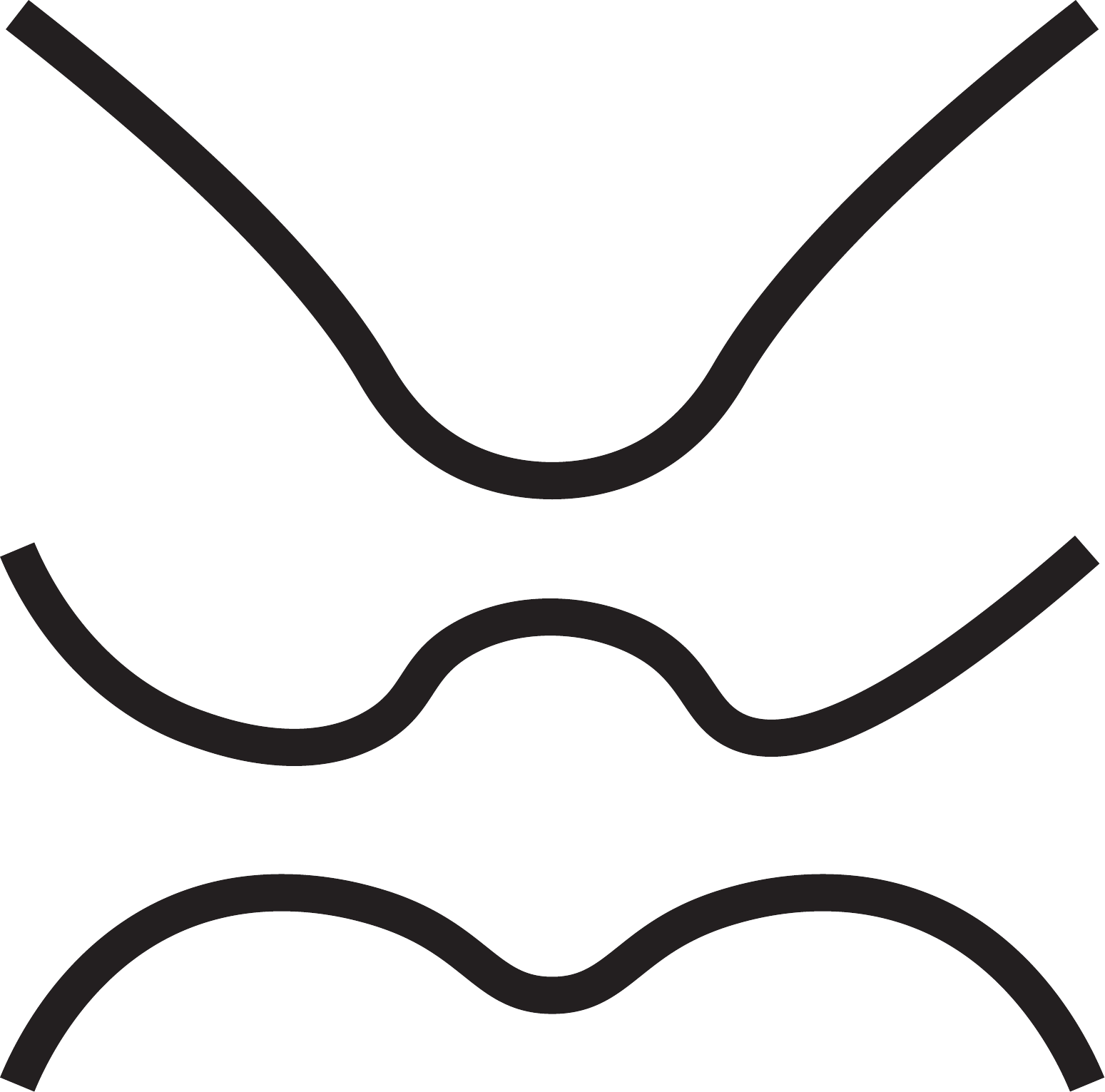}}}}

\newcommand{\Dpooo}{\raisebox{0.2cm}{\rotatebox{180}{$\Dooo$}}}
\newcommand{\Dpioo}{\raisebox{0.2cm}{\rotatebox{180}{$\Dioo$}}}
\newcommand{\Dpoio}{\raisebox{0.2cm}{\rotatebox{180}{$\Doio$}}}
\newcommand{\Dpooi}{\raisebox{0.2cm}{\rotatebox{180}{$\Dooi$}}}
\newcommand{\Dpiio}{\raisebox{0.2cm}{\rotatebox{180}{$\Diio$}}}
\newcommand{\Dpioi}{\raisebox{0.2cm}{\rotatebox{180}{$\Dioi$}}}
\newcommand{\Dpoii}{\raisebox{0.2cm}{\rotatebox{180}{$\Doii$}}}
\newcommand{\Dpiii}{\raisebox{0.2cm}{\rotatebox{180}{$\Diii$}}}

\newcommand{\DiooWithLetters}[3]{
\raisebox{-.2cm}{\begin{picture}(20,20)
\put(0,0){\includegraphics[scale=.04]{D100}}
\put(-5,12){\tiny{$#1$}}
\put(0,-5){\tiny{$#2$}}
\put(15,-5){\tiny{$#3$}}
\end{picture}}
}

\newcommand{\DoioWithLetters}[4]{
\raisebox{-.2cm}{\begin{picture}(20,20)
\put(0,0){\includegraphics[scale=.04]{D010}}
\put(-7,-5){\tiny{$#1$}}
\put(2,18){\tiny{$#2$}}
\put(7.1,2.6){\scalebox{.55}{$#3$}}
\put(19,2){\tiny{$#4$}}
\end{picture}}
}

\newcommand{\DpiooWithLetters}[3]{
\raisebox{-.2cm}{\begin{picture}(20,20)
\put(0,0){\raisebox{.67cm}{\rotatebox{180}{\includegraphics[scale=.04]{D100}}}}
\put(-5,13){\tiny{$#1$}}
\put(0,-4){\tiny{$#2$}}
\put(15,-4){\tiny{$#3$}}
\end{picture}}
}

\newcommand{\DpoioWithLetters}[4]{
\raisebox{-.2cm}{\begin{picture}(20,20)
\put(0,0){\raisebox{.65cm}{\rotatebox{180}{\includegraphics[scale=.04]{D010}}}}
\put(2,-4){\tiny{$#1$}}
\put(19,12){\tiny{$#2$}}
\put(7.1,11.63){\scalebox{.55}{$#3$}}
\put(-5,12){\tiny{$#4$}}
\end{picture}}
}

\newcommand{\DioiWithLetters}[3]{
\raisebox{-.2cm}{\begin{picture}(20,20)
\put(0,0){\includegraphics[scale=.04]{D101}}
\put(-5,11){\tiny{$#1$}}
\put(2,-5){\tiny{$#2$}}
\put(19,11){\tiny{$#3$}}
\end{picture}}
}

\newcommand{\DoiiWithLetters}[3]{
\raisebox{-.2cm}{\begin{picture}(20,20)
\put(0,0){\includegraphics[scale=.04]{D011}}
\put(-8,-5){\tiny{$#1$}}
\put(2,20){\tiny{$#2$}}
\put(19,2){\tiny{$#3$}}
\end{picture}}
}

\newcommand{\DpioiWithLetters}[3]{
\raisebox{-.2cm}{\begin{picture}(20,20)
\put(0,0){\raisebox{.65cm}{\rotatebox{180}{\includegraphics[scale=.04]{D101}}}}
\put(-8,-4){\tiny{$#1$}}
\put(2,20){\tiny{$#2$}}
\put(19,-4){\tiny{$#3$}}
\end{picture}}
}

\newcommand{\DpoiiWithLetters}[3]{
\raisebox{-.2cm}{\begin{picture}(20,20)
\put(0,0){\raisebox{.65cm}{\rotatebox{180}{\includegraphics[scale=.04]{D011}}}}
\put(-6,12){\tiny{$#1$}}
\put(1.5,-4){\tiny{$#2$}}
\put(19,12){\tiny{$#3$}}
\end{picture}}
}

\numberwithin{equation}{section}
\numberwithin{table}{section}

\begin{document}

\begin{abstract}
Knotoids were introduced by V. Turaev as open-ended knot-type diagrams that generalize knots. Turaev defined a two-variable polynomial invariant of knotoids which encompasses a generalization of the Jones knot polynomial to knotoids. We define a triply-graded homological invariant of knotoids categorifying the Turaev polynomial, called winding homology. Forgetting one of the three gradings gives a generalization of the Khovanov knot homology to knotoids.
\end{abstract}

\maketitle

%------------------------------------------------------------------------

\section{Introduction}
\subsection{Summary}
Turaev~\cite{Turaev} introduced the theory of knotoids in 2010. Knotoids are presented by knot-like diagrams that are generic immersions of the unit interval into a surface, together with the under/overpassing information at double points. Knotoids are defined as the equivalence classes of knotoid diagrams under isotopy and the Reidemeister moves; see ~\cite{GKL} for a survey and ~\cite{GDS} for comprehensive tables of knotoids. Intuitively, knotoids can be considered as open-ended knot-type pictures up to an appropriate equivalence. It is shown in ~\cite{Turaev} that knotoids in $S^2$ generalize knots in $S^3$.

Turaev generalized the Jones knot polynomial to knotoids in $S^2$. Moreover, Turaev introduced a two-variable polynomial invariant of knotoids extending the Jones polynomial.

On the other hand, for an oriented link diagram, Khovanov~\cite{Khovanov} defined a bigraded chain complex whose homology is an invariant of the link. This invariant is a categorification of the Jones polynomial in the sense that the graded Euler characteristic of the Khovanov homology is the Jones polynomial. The Khovanov homology is a stronger invariant of knots than the Jones polynomial.

In this paper, we generalize the Khovanov knot homology to knotoids. We show that the resulting Khovanov homology of knotoids is stronger than the Jones polynomial of knotoids. Next, we categorify the two-variable Turaev polynomial to a triply-graded homological invariant of knotoids. We call this homological invariant \emph{winding homology} as its definition substantially uses winding numbers of closed curves in $\R^2$. In particular, forgetting one of the gradings in the winding homology yields our Khovanov homology of knotoids.\begin{figure}[hbt!]
    \centering
    \begin{tikzcd}[row sep=2cm, column sep=2cm]
        \text{Jones knotoid polynomial} \arrow[r] \arrow[d] & \text{Turaev polynomial} \arrow[d] \\
        \text{Khovanov knotoid homology} \arrow[r] \arrow[ru,no head, loosely dotted, thick] & \text{Winding homology}
    \end{tikzcd}
    \caption{Solid arrows stand for \lq\lq stronger" invariant of knotoids in $S^2$. The diagonal dotted line means that neither is stronger than the other one.}
    \label{fig:allinvariantscompared}
\end{figure}
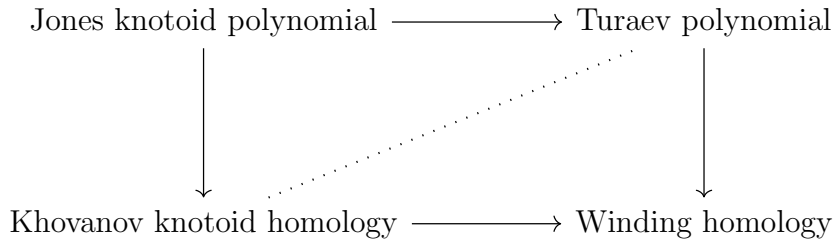 We provide examples of knotoid pairs illustrating the strength of the invariants, as claimed in Figure ~\ref{fig:allinvariantscompared}. We also show that the winding homology is a stronger invariant than the Turaev polynomial and the Khovanov knotoid homology combined. For knots, the winding homology is equivalent to the usual Khovanov homology.

Lastly, we introduce the winding potential function of smooth closed curves in $\R^2$, and use this function to obtain refined polynomial invariants of knotoids in $\R^2$. It is not surprising to have refined invariants for knotoids in $\R^2$ as there are more knotoids in $\R^2$.

\subsection{Organization}

In Section ~\ref{sec.knotoids}, we give a brief introduction to knotoids, and review the Turaev polynomial. In Section ~\ref{sec.categorification}, we define a triply-graded chain complex for a knotoid diagram in $S^2$. Invariance of the homology of this chain complex under the Reidemeister moves, and independence from the choices made in the definition are proved in Section ~\ref{sec.invariance}. In Section ~\ref{sec.computations}, we provide computational results and examples. In Section ~\ref{sec:r2knotoids}, we give refinements of Turaev's polynomials for knotoids in $\R^2$.

\subsection{Acknowledgements} I am indebted to my advisor at Indiana University, Vladimir Turaev, for encouraging me to pursue this problem, and for all his help. I would like to thank Matthew Hogancamp, Paul Kirk, Charles Livingston, and Dylan Thurston for many valuable conversations and comments. The program computing the winding homology of knotoids is built on top of the code of Bar-Natan's program computing the Khovanov homology of knots. The computer search for the examples was run on the supercomputer Carbonate of Indiana University.

%------------------------------------------------------------------------

\section{Knotoids and the Turaev polynomial}\label{sec.knotoids}
\subsection{Knotoids}\label{subsec:knotoids}
We review the essentials of the theory of knotoids; see \cite{Turaev}, \cite{GKL}, \cite{GDS} for details. A \emph{knotoid diagram} in a surface $\Sigma$ is an immersion $K:[0,1]\rightarrow\Sigma$ having only double transversal points and over/under information for each crossing. The images of 0 and 1 under the immersion are called the \emph{leg} and \emph{head} of $K$, respectively. A \emph{multi-knotoid diagram} is defined in the same way except possibly with extra closed components, but still with only one segment component. Two (multi-)knotoid diagrams are \emph{isotopic} if there is an ambient isotopy of $\Sigma$ that transforms one knotoid into the other preserving their orientations. Two (multi-)knotoid diagrams are \emph{equivalent} if they are isotopic or can be transformed one into the other by Reidemeister moves. Each equivalence class under this relation is called a \emph{knotoid}. Note that passing an arc of $K$ over the head or leg of $K$, as shown in Figure~\ref{fig:omegamoves} and denoted by $\Omega_{\pm}$, is not allowed under equivalence. Similarly, passing an arc under the head/leg, denoted by $\overline{\Omega}_{\pm}$, is also not allowed. In this paper, we will mostly consider knotoids in $S^2$, and when specified, knotoids in $\R^2$.
\begin{figure}[hbt!]
    \centering
    \includegraphics[scale=0.65]{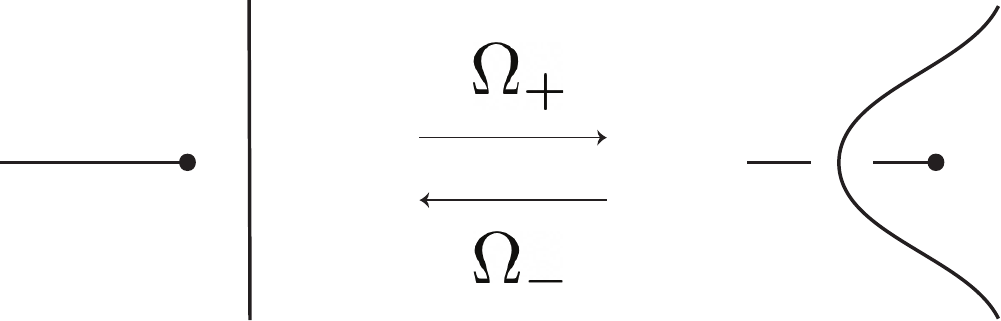}
    \caption{Forbidden $\Omega_{\pm}$ moves for passing an arc over the head/leg.}
    \label{fig:omegamoves}
\end{figure}
For a knotoid diagram $K$ in $S^2$, an immersed, oriented arc $\alpha$ (the red dashed arc in Figure~\ref{fig:Kplusminus}) from the leg to the head of $K$ is called a \emph{shortcut}. We write $\alpha^r$ for $\alpha$ with reversed orientation. When $\alpha$ is embedded, $K\cup \alpha^r$ specifies a knot in $S^3$ if $\alpha$ is assumed to pass under $K$ at each crossing. This knot is written as $K_-$; see Figure~\ref{fig:Kplusminus}. If $\alpha$ passes over $K$ at each crossing, then the resulting knot is written as $K_+$.
\begin{figure}[hbt!]
    \centering
    \includegraphics[scale=.35]{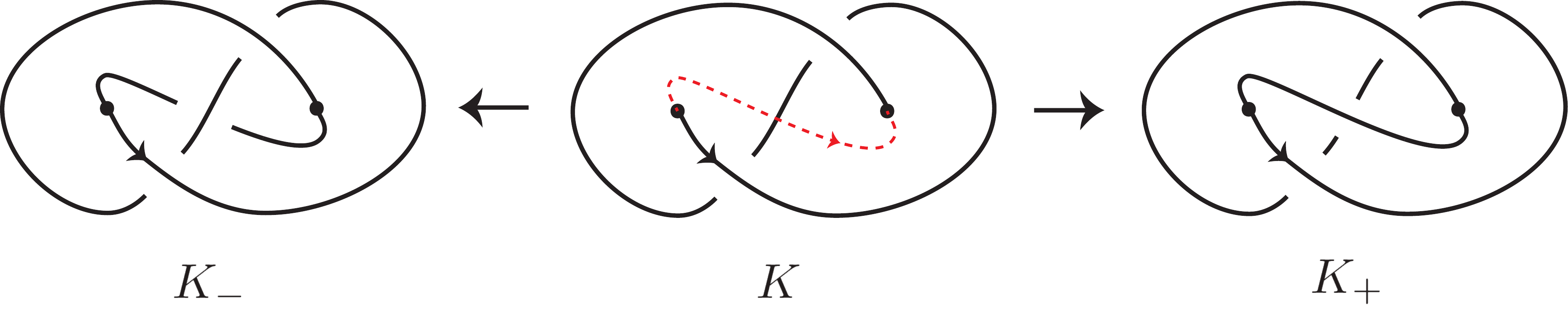}
    \caption{Maps from knotoids in $S^2$ to knots in $S^3$.}
    \label{fig:Kplusminus}
\end{figure}

In the other direction, given an oriented knot diagram $\kappa$ in $S^2$, deleting a small arc yields a knotoid diagram, denoted $\kappa^\bullet$. It turns out (see ~\cite{Turaev}) that different choices of the deleted arc give equivalent knotoid diagrams. Then, it follows that the knotoid $\kappa^\bullet$ is well-defined in $S^2$. Since $(\kappa^\bullet)_\pm=\kappa$, the set of knots injects into knotoids, and thus, can be considered a subset. Knotoids of the form $\kappa^\bullet$ will be referred as knots. A knotoid, that is not a knot, is called a \emph{pure knotoid}.

The set of (multi-)knotoids has a well-defined multiplication (see ~\cite{Turaev}) as follows: For two (multi-)knotoids $K_1$ and $K_2$, pick small disk neighborhood $N_1$ of the head of $K_1$, and a small disk neighborhood $N_2$ of the leg of $K_2$. Then gluing $S^2-N_1$ to $S^2-N_2$ along the boundary consistent with the knotoids and their orientations gives another (multi-)knotoid in $S^2$, denoted by $K_1\cdot K_2$.

\subsection{The Turaev polynomial}
Let $K$ be a knotoid on a surface $\Sigma$. A \emph{state} of $K$ is obtained by resolving the crossings of $K$ by 0 and 1-smoothings; see Figure ~\ref{fig:smoothings}.
\begin{figure}[hbt!]
    \centering
    \includegraphics[scale=.55]{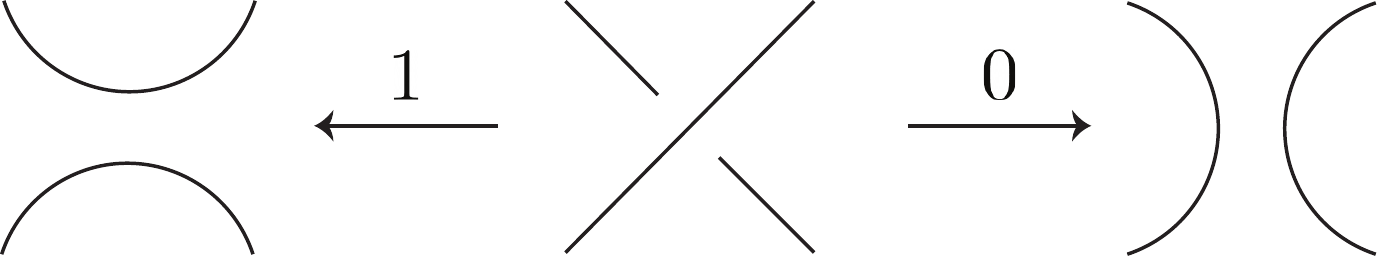}
    \caption{0 and 1-smoothings.}
    \label{fig:smoothings}
\end{figure}
In analogy with the Kauffman bracket polynomial of knots (see ~\cite{Kauffman}) the \emph{Kauffman bracket polynomial} of the knotoid $K$ is defined as
\begin{equation}
    \langle K\rangle =\sum_{s\in S(K)}A^{\sigma_s}(-A^2-A^{-2})^{|s|-1}\in\Z[A^{\pm 1}],
\end{equation}
where $S(K)$ is the set of all states of $K$, $|s|$ is the number components of $s$, and $\sigma_s$ is the number of 0-smoothings minus the number of 1-smoothings. Normalizing the Kauffman bracket, the \emph{Jones polynomial} of the knotoid $K$ is written as 
\begin{equation}
    J_K(A)=\langle K \rangle_\circ=(-A^3)^{-wr(K)}\langle K \rangle,
\end{equation}
where $wr(K)=n_+-n_-$, and $n_\pm$ is the number of positive/negative crossings.

For a knotoid $K\subset S^2$, let $\alpha$ be a shortcut oriented from the leg of $K$ to the head of $K$. Then the \emph{Turaev polynomial} (see ~\cite{Turaev}) is given by
\begin{equation}\label{equ:extendedbracket}
    T_K(A,u)=(-A^3)^{-wr(K)}u^{-K\cdot \alpha}\sum_{s\in S(K)}A^{\sigma_s}u^{k_s\cdot \alpha}(-A^2-A^{-2})^{|s|-1}\in\Z[A^{\pm 1},u^{\pm 1}].
\end{equation}
The term $k_s$ denotes the segment component of the state $s$, oriented from the leg to the head. Then $K\cdot \alpha$ (and $k_s\cdot \alpha$) denote the number of times $K$ (and $k_s$) crosses $\alpha$ from right to left minus the number of times from left to right. The Laurent polynomial $T_K$ is independent of the choice of $\alpha$, and invariant under Reidemeister moves. The substitution $u=1$ in the Turaev polynomial recovers the Jones polynomial of the knotoid $K$:
\begin{equation}\label{equ:jones4knotoids}
    T_K(A,u=1)=J_K(A).
\end{equation}
Using the substitution $q=-A^{-2}$, one can rewrite ~\eqref{equ:extendedbracket} as
\begin{equation}\label{equ:extendedJones}
    T_K(q,u)=(-1)^{n_-}q^{n_+-2n_-}\sum_{s\in S(K)} u^{\mu(s)}(-q)^{|\!|s|\!|}(q+q^{-1})^{|s|-1},
\end{equation}
where $|\!|s|\!|$ is the number of 1-smoothings of $s$, and $\mu(s)=k_s\cdot \alpha-K\cdot \alpha$. Note that $k_s\cdot \alpha$ is equal to $(k_s\cup_i c_s^i)\cdot \alpha$ for any arbitrary assignment of orientations on closed components $\{c_s^i\}_i$ of $s$. We may assume that all smoothings take place away from the intersection points between $K$ and $\alpha$. Thus, the contribution of each intersection to $(k_s\cup_i c_s^i)\cdot \alpha$ and $K\cdot \alpha$ is either the same or differ by $\pm2$, so $\mu(s)$ is an even integer. Since $wr(K)$ and $\sigma(s)$ have the same parity, it is, in fact, the case that $T_K(q,u) \in \Z[q^{\pm 1},u^{\pm 2}]$.

The following substitutions recover the Jones polynomials of the knots $K_-$, $K_+$ from the Turaev polynomial of the knotoid $K$ (see ~\cite{Turaev}):
\begin{align}
    T_K(q, u=-i\cdot q^{-\frac{3}{2}})&=J_{K_-}(q),\\
    T_K(q, u=i\cdot q^{\frac{3}{2}})&=J_{K_+}(q).
\end{align}

%------------------------------------------------------------------------

\section{Winding Homology}\label{sec.categorification}
\noindent In this section, we give a categorification of the Turaev polynomial $T_K$. For simplicity, we work over the field $\Q$, rather than over $\Z$ or the polynomial ring $\Z[c]$ as in ~\cite{Khovanov}. With a little more work, the constructions of this section can be carried out over $\Z$.

\subsection{Chain groups} 
For an oriented (multi-)knotoid diagram $K$ with $n$ crossings together with a choice of shortcut $\alpha$, let  $\C=\{c_1,c_2,\dots,c_n\}$ be an ordered set of crossings with the ordering $c_1<c_2<\dots<c_n$, and $\V$ be the $\Q$-vector space with the ordered basis $\C$. A complete resolution diagram for $K$ is specified by a \emph{smoothing} function $s:\C\rightarrow \{0,1\}$ which will sometimes be referred as a \emph{state}. Let $\V_s$ be the subspace of $\V$ with basis $s^{-1}(1)$ endowed with an ordering inherited from $\C$, so that $|\!|s|\!|=|s^{-1}(1)|$. The exterior power $\bigwedge^{|\!|s|\!|}\V_s$ is a one-dimensional $\Q$-vector space.

Let $\A=\Q[X]/(X^2)$ be the graded $\Q$-vector space with $\deg(1)=1$ and $\deg(X)=-1$. For any tensor powers of $\A$, the degree is extended additively, that is, $\deg(v\otimes w)=\deg(v)+\deg(w)$. We define the $i$-th chain group as follows
\begin{equation}\label{equ:windingithchaingroup}
    C_i(K)=\bigoplus_{\underset{|\!|s|\!|=i+n_-}{s\in S(K)}} \A^{\otimes |s|-1}\otimes X\A \otimes \bigwedge\nolimits^{|\!|s|\!|}\V_s,
\end{equation}
where $-n_-\leq i\leq n_+$. Intuitively, each closed component of a state is assigned a copy of $\A$, the segment component is assigned the one-dimensional $\Q$-vector space $X\A$, and the wedge power is included to determine signs in differentials later.

The chain groups are equipped with three gradings. For a generator $v$ in a summand of $C_i(K)$ labeled by state $s$, 
\begin{itemize}
    \item[i.] the homological grading is given by $i(v)=|\!|v|\!|-n_-=|\!|s|\!|-n_-$,
    \item[ii.] the $q$-grading is given by $q(v)=\deg(v)+i(v)+n_+-n_-+1$, and
    \item[iii.] the $u$-grading is given by $\mu(v)=k_s\cdot \alpha-K\cdot \alpha$.
\end{itemize}
\begin{notation}
The total complex is written as $C(K)=\bigoplus_i C_i(K)=\underset{i,j,k}{\bigoplus}C_{i,j}^k(K)$ where subscripts $i,j$ refer to the $i$- and $q$-gradings, respectively, whereas $k$ refers to the $u$-grading.
\end{notation}
\begin{figure}[hbt!]
    \centering
    \includegraphics[scale=0.4]{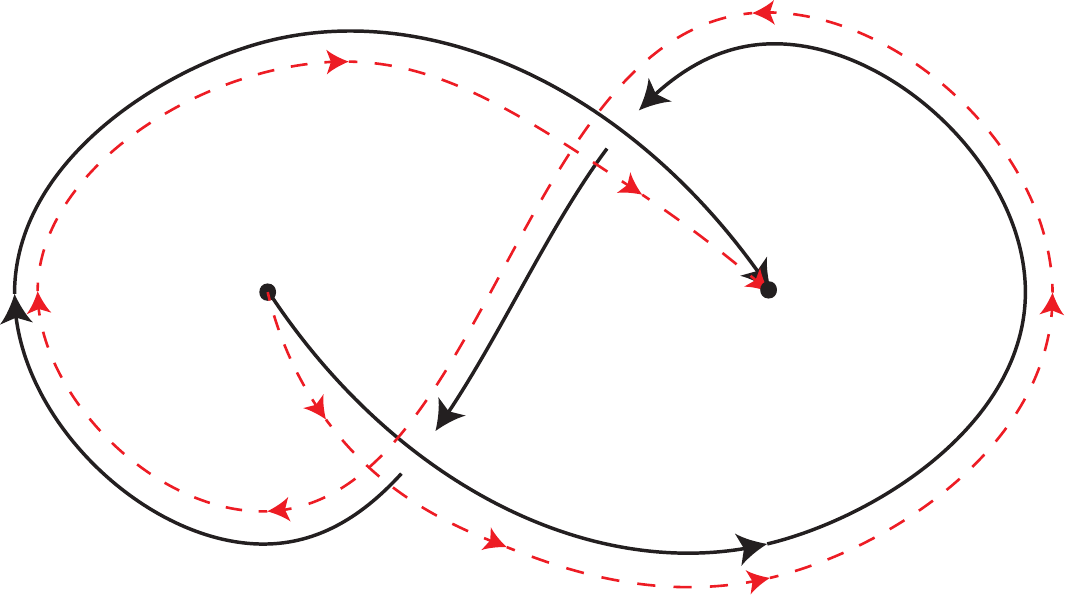}
       \caption{Canonical shortcut.}
    \label{fig:canonicalshortcut}
\end{figure}

The $u$-grading $\mu(v)$ of a generator $v$ depends only on the index state of the summand space, to which $v$ belongs. In particular, $\mu(v)$ is independent of the shortcut $\alpha$, which is dropped from the notation. However, to calculate the gradings of the chain groups, one has to make a choice of $\alpha$. To eliminate this intermediate step, and calculate $\mu(v)$ directly from the state $s$, we make a special choice for the shortcut $\alpha$ as follows. 
\begin{defn}
For a knotoid $K$, the shortcut obtained by pushing $K$ slightly to its right, while having the same orientation as $K$ (see Figure~\ref{fig:canonicalshortcut}) is referred as the \emph{canonical shortcut} of $K$.
\end{defn}

When $\alpha$ is chosen as the canonical shortcut, both positive and negative crossings have no contribution to $K\cdot \alpha$, so it is sufficient to consider $\mu(v)=k_s\cdot \alpha$. Notice that a positive (resp. negative) crossing does not contribute to $k_s\cdot \alpha$ when it is assigned a $0$-smoothing (resp. $1$-smoothing); see Figure ~\ref{fig:contributions to ks dot a}.
\begin{figure}[hbt!]
    \centering
    \begin{subfigure}[b]{0.4\textwidth}
    \centering
    \includegraphics[width=\textwidth]{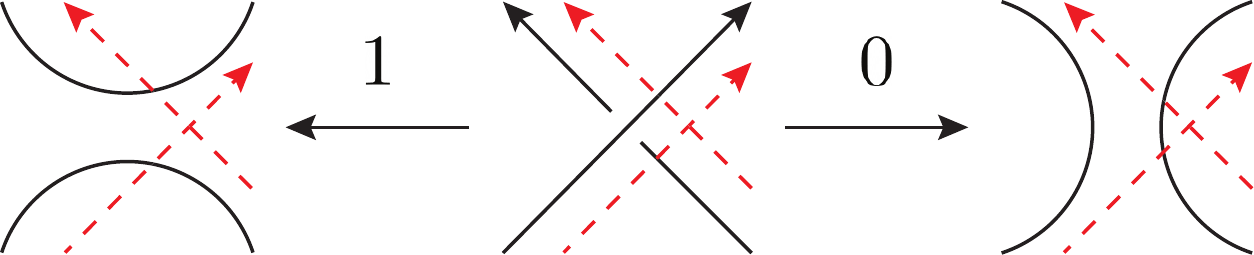}
    \caption{Positive crossing}
    \label{fig:positive4shortcut}
    \end{subfigure}
    \hspace{.5in}
    \begin{subfigure}[b]{0.4\textwidth}
    \centering
    \includegraphics[width=\textwidth]{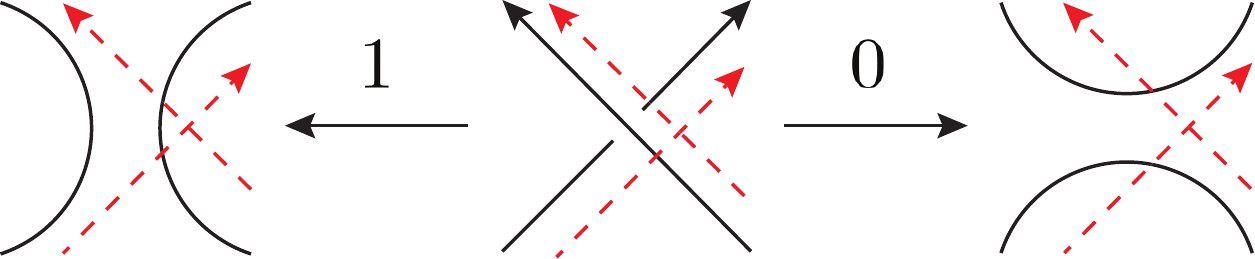}
    \caption{Negative crossing}
    \label{fig:negative4shortcut}
    \end{subfigure}
    \caption{Contributions of crossings to $k_s\cdot \alpha$.}
    \label{fig:contributions to ks dot a}
\end{figure}
For the two other remaining cases, the sign of each crossing of $k_s$ and $\alpha$ depends whether the incoming string of $k_s$ to the crossing of $K$ is an overpass or underpass, and whether it has the same or opposite orientation compared to the orientation of $K$. All possible cases are listed in Table ~\ref{tab:ks signs}.
\begin{table}[hbt!]
    \centering
    \begin{subtable}{0.45\textwidth}
        \begin{tabular}{|c|c|c|} \hline
            & Overpass & Underpass \\ \hline
            Same & - & + \\ \hline
            Opposite & + & -\\ \hline
        \end{tabular}
        \caption{Positive crossing with 1-smoothing}
        \label{tab:pos1smooth}
    \end{subtable}
    \hfill
    \begin{subtable}{0.45\textwidth}
        \begin{tabular}{|c|c|c|} \hline
            & Overpass & Underpass \\ \hline
            Same & + & - \\ \hline
            Opposite & - & +\\ \hline
        \end{tabular}
        \caption{Negative crossing with 0-smoothing}
        \label{tab:neg0smooth}
    \end{subtable}
    \caption{Signs for the calculation of $k_s\cdot \alpha$}\vspace{-.4cm} 
    \label{tab:ks signs}
\end{table}
To summarize this information in a compact formula that is independent of the shortcut $\alpha$, we introduce the following functions.
\begin{defn}
For a crossing $c$ and a marked incoming arc of $\gamma$ of $c$, the \emph{flow} function is defined as
\[ \phi_c(\gamma)=
    \begin{cases}
      +1 & \gamma \text{ has the same orientation as }$K$. \\
      -1 & \text{Otherwise.}
   \end{cases}
\]and the \emph{level} function as
\[ \lambda_c(\gamma)=
    \begin{cases}
      +1 & \gamma \text{ approaches to an overpass}. \\
      -1 & \text{Otherwise.}
   \end{cases}
\]
\end{defn}
\begin{lem}\label{lem:shortcutformula}
For a knotoid diagram $K$, letting $sign(c)$ denote the sign of the crossing $c$ and $s(c)$ be the value of $c$ under the smoothing function, we have
\begin{equation}\label{equ:upsilon}
    \mu(v)=\sum_{\underset{\gamma \subset k_s, \gamma\rightarrow c}{c \subset K}}
    \delta_{sign(c),(-1)^{s(c)+1}}(-1)^{s(c)}\phi_c(\gamma)\lambda_c(\gamma),
\end{equation}
where $v$ is a generator of the summand $\A^{\otimes |s|-1}\otimes X\A \otimes \bigwedge\nolimits^{|\!|s|\!|}\V_s$ of $C_i(K)$, and the notation $\gamma\rightarrow c$ means that $\gamma$ is an approaching arc to $c$.
\end{lem}
\begin{rem}
Note that the summation is not necessarily over all crossings of $K$ since the number of times $k_s$ visits a crossings can be 0, 1 or 2 depending on the state $s$. In either case, one does not need to make a choice of $\alpha$ to compute the formula ~\eqref{equ:upsilon}.
\end{rem}
\begin{rem}
It is required to work with a knotoid diagram $K$, as opposed to a multi-knotoid, to use Lemma ~\ref{lem:shortcutformula}, because it is assumed above that the canonical shortcut $\alpha$ follows $K$ through all crossings twice. For a multi-knotoid diagram, it is also true that $K\cdot \alpha =0$, but $k_s\cdot \alpha$ is not necessarily equal to $\mu(v)$ in the formula~\eqref{equ:upsilon}.
\end{rem}

\subsection{Differentials}
A complete resolution of a knotoid diagram yields a single segment component $k_s$ and zero or more closed components. For any two states that differ by only one crossing, the number of components either increase/decrease by 1 or stay the same. The corresponding maps can be viewed as merging of two components (closed or segment), division of a component (closed or segment) or a self map of the segment component, called \emph{anticurl} or \emph{turbulence} map. All these (local) cobordisms and their induced maps on associated vector spaces are listed diagrammatically in Figure ~\ref{fig:localcoborisms and induced maps}.
\begin{figure}[hbt!]
    \centering
    \begin{subfigure}[b]{0.3\textwidth}
        \includegraphics[width=\textwidth]{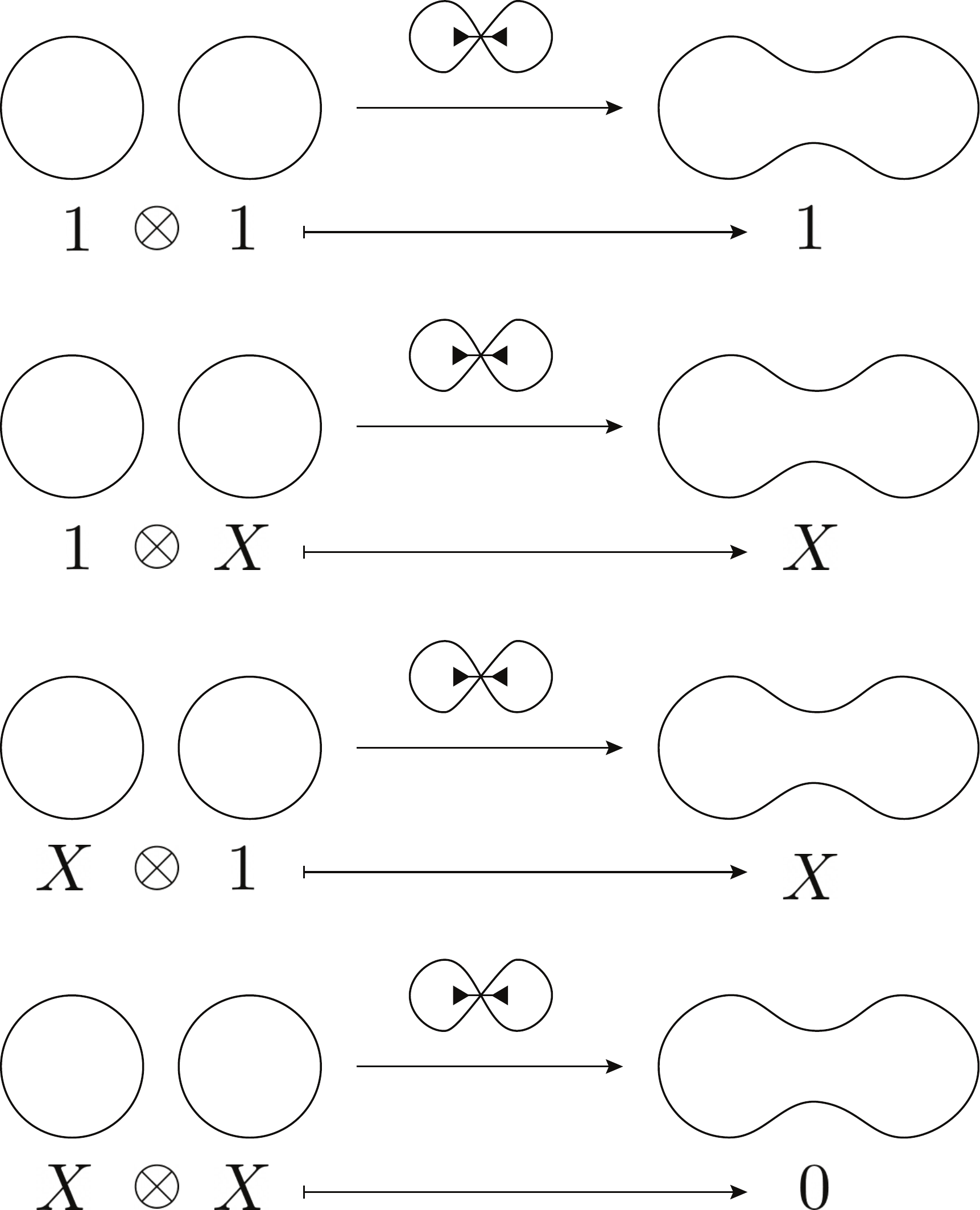}
        \caption{$m_1:\A\otimes \A\rightarrow \A$. Merging of two closed components.}
        \label{fig:merge}    
    \end{subfigure}
    \hfill
    \begin{subfigure}[b]{0.3\textwidth}
        \includegraphics[width=\textwidth]{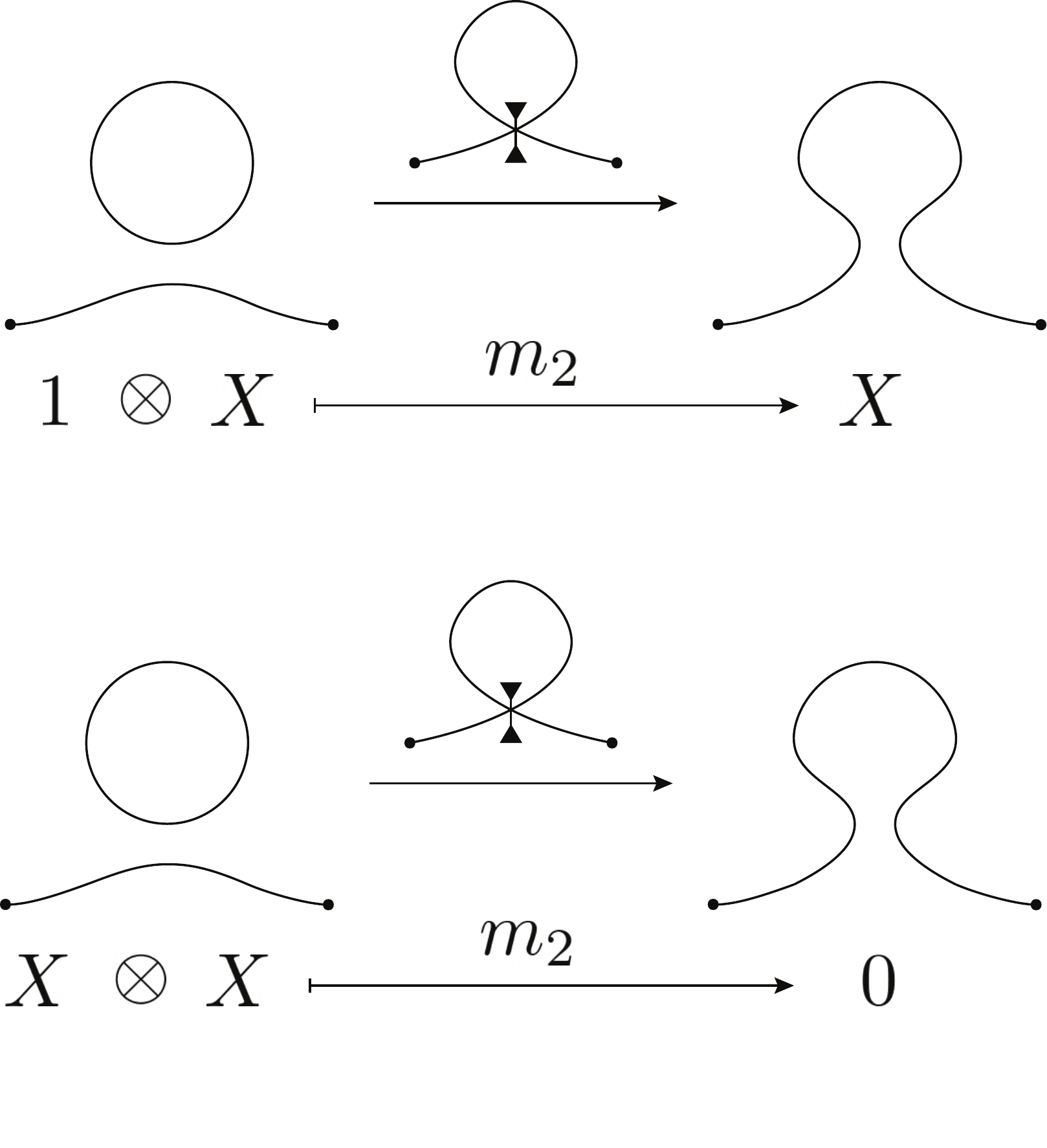}
        \caption{$m_2:\A\otimes X\A\rightarrow X\A$.\\ Merging into the segment component.}
        \label{fig:mergesegment}    
    \end{subfigure}
    \hfill
    \begin{subfigure}[b]{0.3\textwidth}
        \includegraphics[width=\textwidth]{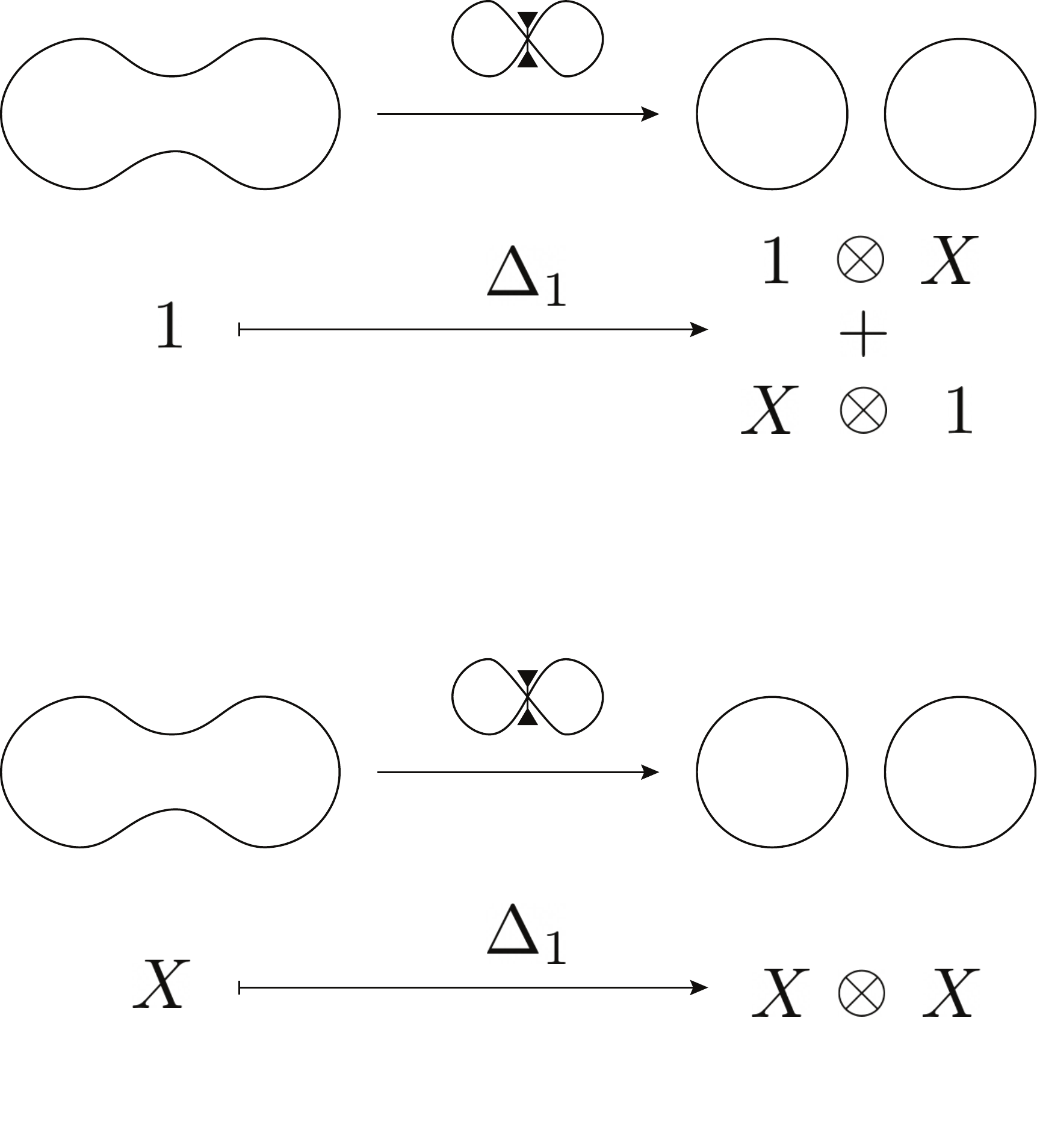}
        \caption{$\Delta_1:\A \rightarrow \A\otimes \A$.\\ Division of a closed component into two.}
        \label{fig:divide}    
    \end{subfigure}\\
    \begin{subfigure}[b]{0.3\textwidth}
        \includegraphics[width=\textwidth]{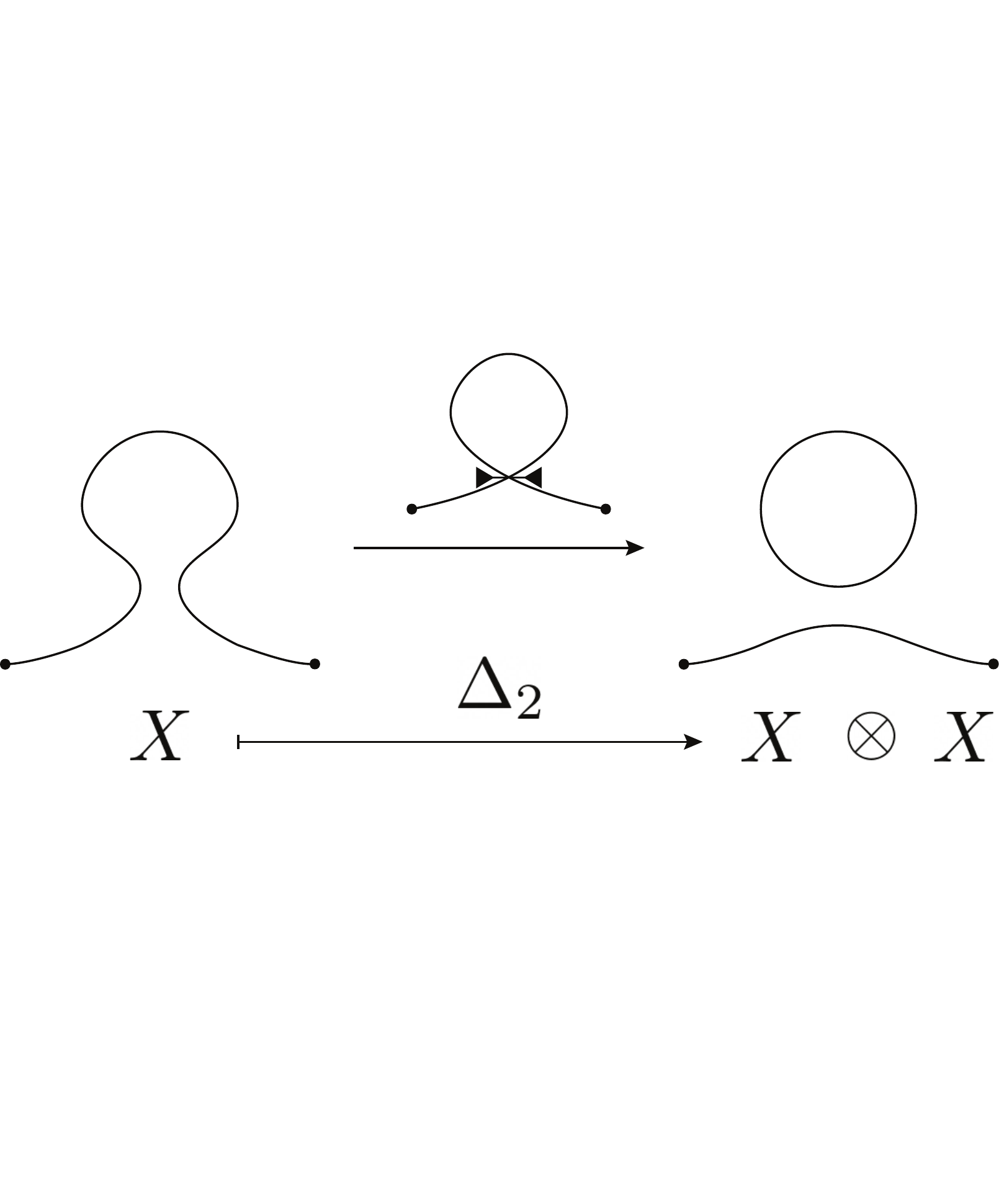}
        \caption{$\Delta_2:X\A \rightarrow \A\otimes X\A$. Division from the segment component.}
        \label{fig:dividesegment}    
    \end{subfigure}
    \hfill
    \begin{subfigure}[b]{0.3\textwidth}
        \includegraphics[width=\textwidth]{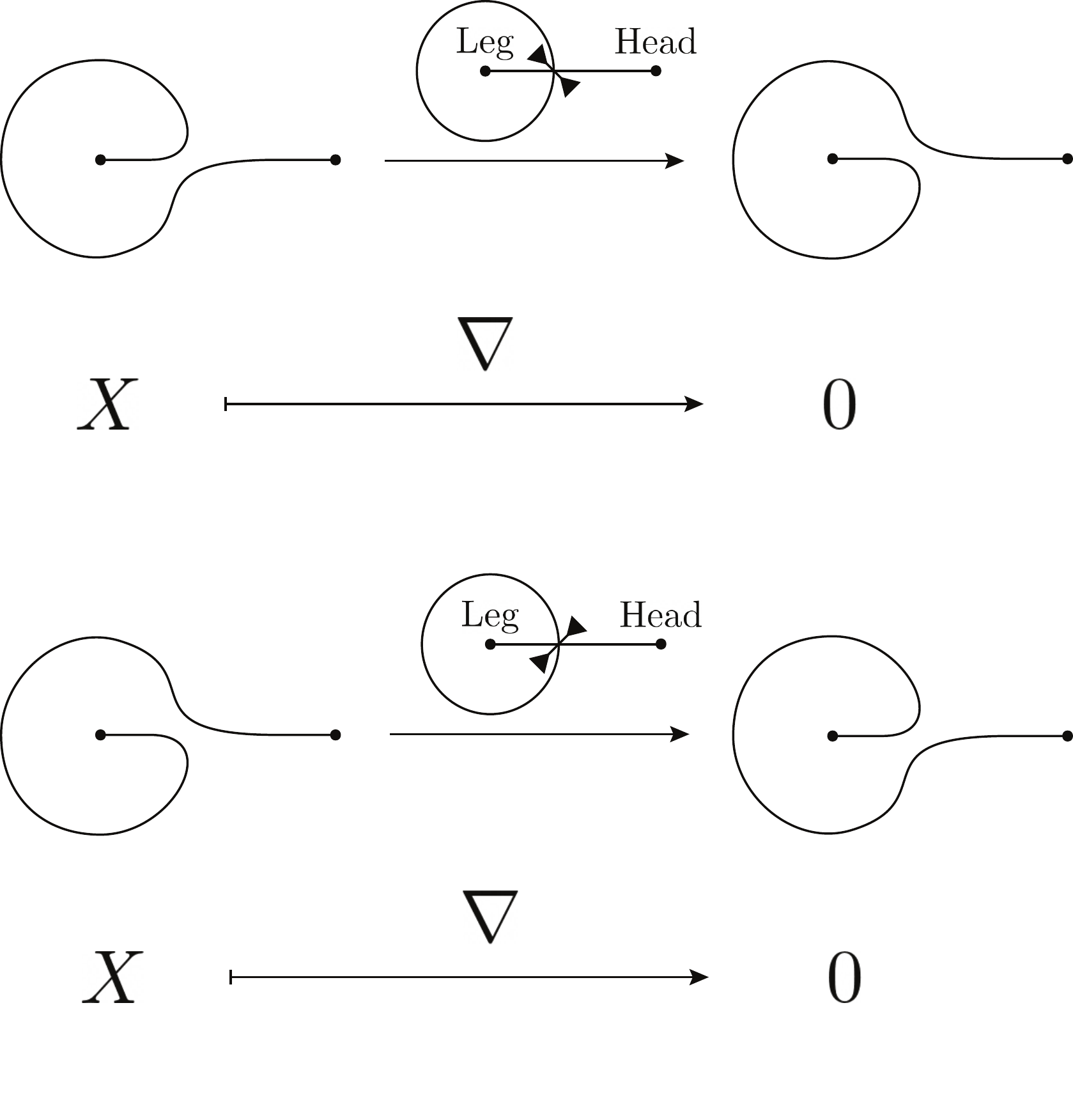}
        \caption{$\nabla: X\A\rightarrow X\A$. Anti-curl (or turbulence) move as self morphism of the segment component.}
        \label{fig:anticurl}    
    \end{subfigure}
    \hfill
    \begin{subfigure}[b]{0.3\textwidth}
        \includegraphics[width=\textwidth]{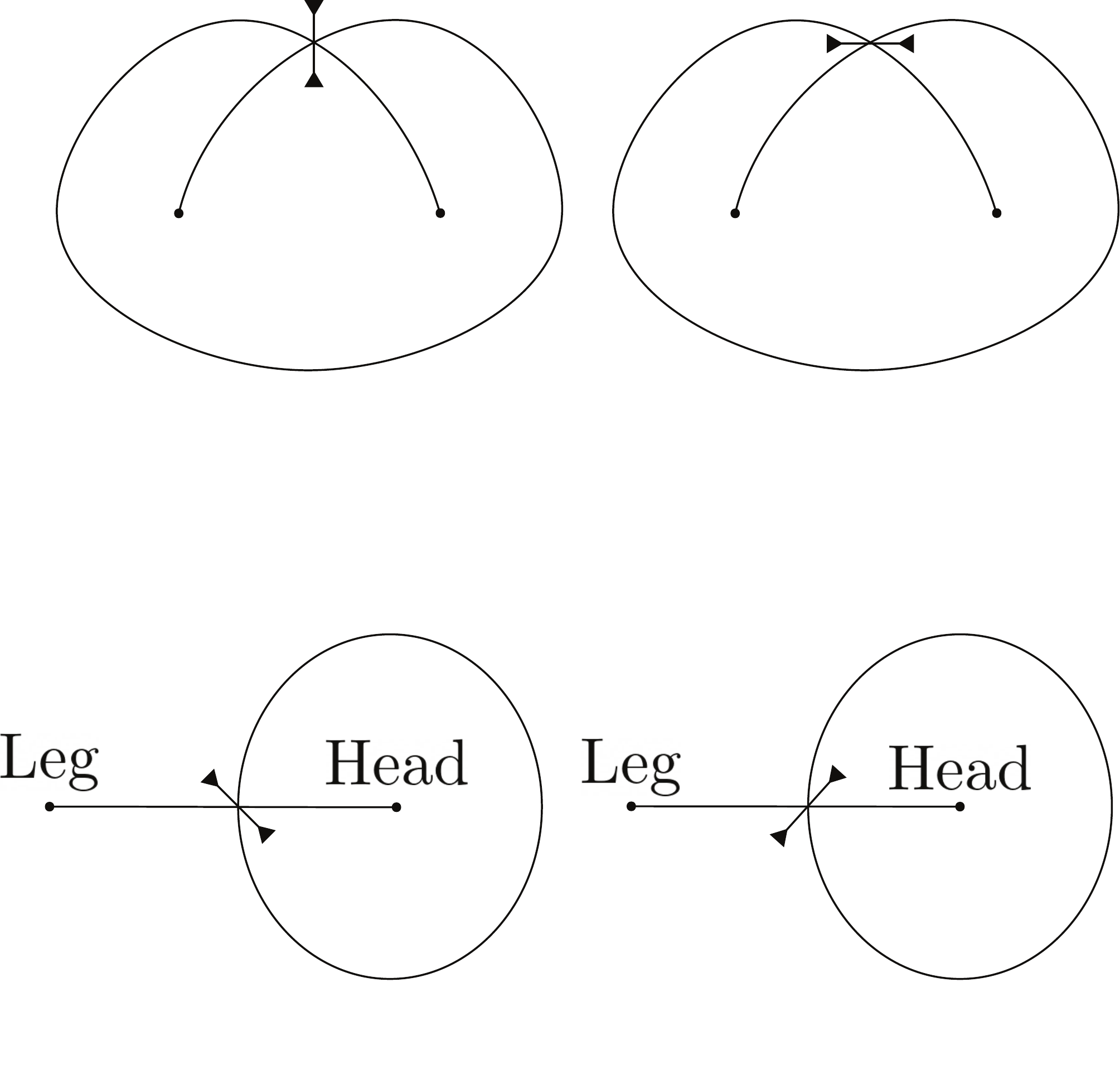}
        \caption{These maps are equivalent to the others in $S^2$ and are omitted.}
        \label{fig:omitted}    
    \end{subfigure}
    
    \caption{Local cobordisms and induced maps.}
    \label{fig:localcoborisms and induced maps}

\end{figure}
The maps are represented diagrammatically by an immersed segment together with a small line segment. The arrows on the line segment designate the local regions to be merged. In Figure~\ref{fig:anticurl}, the left endpoint stands for the leg and the right endpoint for the head of the knotoid, and one can observe that these maps are indistinguishable in $S^2$ from the maps in the second row of Figure~\ref{fig:omitted}. Let $s$ and $s'$ be two states that have the same value at all crossings except at $c_0$ where $s(c_0)=0$ and $s(c_0)=1$. The cobordism from $s$ to $s'$ is necessarily one of the local cobordism in Figure ~\ref{fig:localcoborisms and induced maps} and identity cobordism everywhere else. Thus, we define the \emph{edge map} $e_{s\rightarrow s'}$ as follows:
\begin{align}
    e_{s\rightarrow s'}: \A^{\otimes |s|-1}\otimes X\A \otimes \bigwedge\nolimits^{|\!|s|\!|}\V_s &\longrightarrow \A^{\otimes |s'|-1}\otimes X\A \otimes \bigwedge\nolimits^{|\!|s'|\!|}\V_s'\\
    v\otimes w \otimes [x] &\longmapsto (m_{1,2}(v), \Delta_{1,2}(v)\text{ or }\nabla(v))\otimes w \otimes [c_0\wedge x].
\end{align}
It is understood that $v$ belongs to a space associated to a component that undergoes a non-trivial cobordism and $w$ belongs to one that does not.\\
\begin{rem}~\label{rem:signs}
A clarification is needed for the last terms in the tensor products: $[x]$ represents the canonical generator $c_{i_1}\wedge \dots\wedge c_{i_{|\!|s|\!|}}$, (i.e. $i_1<\dots<i_{|\!|s|\!|}$) of the one-dimensional vector space $\bigwedge\nolimits^{|\!|s|\!|}\V_s$. The differential places the crossing $c_0$ to the front as $[c_0\wedge x]$. To bring $c_0$ to its proper location in the generator $c_{i_1}\wedge \dots\wedge c_{i_{|\!|s'|\!|}}$, one may need to apply an odd number of switches and that would make the edge map to pick a negative sign. In ~\cite{BarNatan}, this is explained as counting the number of 1's before the location of the change from 0 to 1, if the image of the state is thought of as a string of 0's and 1's written in the ordering of crossings. In this setting, the states (and corresponding vector spaces) are placed in the corners of an $n$ dimensional cube and the edge maps on the edges of the cube. The exterior power factor makes sure that each face of the cube contains an odd number of edge maps that pick a negative sign. That way any commutative face becomes skew-commutative.
\end{rem}

\begin{defn}
The differential $d_i:C_i(K)\longrightarrow C_{i+1}(K)$ is defined as the sum of all the edge maps whose domains are specified by a state $s$ with $|\!|s|\!|=i+n_-$:
\begin{equation}\label{equ:sumofedgemaps}
    d_i=\sum_{\underset{\underset{|\!|s'|\!|=i+n_-+1}{|\!|s|\!|=i+n_-}}{s, s' \in S(K)}} e_{s\rightarrow s'}.
\end{equation}
\end{defn}

To show that $d^2=0$, it suffices to prove that each face of the cube is skew-commutative. In the light of Remark~\ref{rem:signs}, this reduces to proving that each face of the cube is commutative when the wedge products are dropped from the definition of the chain complex. One way to prove the commutativity of faces is to directly verify, in a routine computation, that all possible ways of arranging the maps in Figure~\ref{fig:localcoborisms and induced maps} on the edges of a square constitute commutative squares. We will, instead, observe that all the maps $\nabla$ induced by the anti-curl move is 0. The rest of the elementary maps, merge and divide, are commutative multiplication and co-commutative co-multiplication, respectively, of a bialgebra that satisfy the Frobenius identity: $\Delta \circ m =(m\otimes\text{id}) \circ (\text{id} \otimes \Delta)$; see Sections 2.3 and 8.2 in~\cite{Khovanov}. Thus, all faces of the cube are commutative.

\subsection{Homology} Since all maps in Figure ~\ref{fig:localcoborisms and induced maps} decrease the degree ($\deg(\cdot)$) by $1$, the edge maps, and consequently the differential $d$ of the chain complex $C(K)$ has degree $(1,0,0)$, that is, the differential preserves the $q$ and $u$-gradings. Therefore, the graded Euler characteristics of the chain complex $C(K)=\bigoplus C_{i,j}^k(K)$, and its homology $H(K)=\bigoplus H_{i,j}^k(K)$ are equal:
\begin{equation}\label{equ:gradedeulers}
    \hat{\chi}(C(K))= \sum_{i,j,k}(-1)^i q^j u^k \text{rk}(C_{i,j}^k(K)) = \sum_{i,j,k}(-1)^i q^j u^k \text{rk}(H_{i,j}^k(K))=\hat{\chi}(H(K)).
\end{equation}

We state the main result of this paper:

\begin{thm}\label{thm:maintheorem}
For an oriented (multi-)knotoid represented by diagram $K$ in $S^2$, the homology groups $H_{i,j}^k(K)$ are invariants, up to isomorphism, of the knotoid, for all $i,j,k$.
\end{thm}
We give the proof of this theorem in Section ~\ref{sec.invariance}. We will call this invariant \emph{the winding homology} of knotoids. For knots (as a subset of knotoids), the winding homology and the reduced Khovanov homology agree.
\begin{thm}\label{thm.restrictW2knots}
The three-variable Poincar\'{e} polynomial $W_K(t,q,u)$ of $H(K)=\bigoplus H_{i,j}^k(K)$ is equivalent to that of the reduced Khovanov homology when $K$ is a knot or a multi-knot (a link with a base point).
\end{thm}
\begin{proof}
When $K$ is a knot, i.e. a knotoid whose endpoints are on the same region, the shortcut $\alpha$ can be chosen so that $\alpha$ does not intersect $K$. Then $\mu(v)=0$ for all $v\in C(K)$ and any saddle cobordism that involves two arcs of the segment component is necessarily either $m_2$ or $\Delta_2$ ($\nabla$ is not possible). Letting $p$ be a point on $\alpha$, the maps $m_2$ and $\Delta_2$ on a resolution of $K$ become $m_1$ and $\Delta_1$ on the same resolution of $K\cup \alpha$. Then the closed component of a resolution of $K\cup \alpha$ with the base point $p$ is assigned the vector space $X\A$ which is exactly how the reduced Khovanov complex of the knot $K\cup \alpha$ with base point $p$ is defined. The same argument holds to go from multi-knotoids to links. The only difference is that, in this case, the homology depends on which component of the link is considered the segment component of the multi-knotoid.
\end{proof}
\begin{cor}
Ignoring the $u$-grading gives the invariant $Kh(K)=\bigoplus H_{i,j}(K)$ of knotoids that generalize the reduced Khovanov homology of knots to knotoids. In particular, for a knotoid $K$, $W_K(t,q,u=1)=Kh_K(t,q)$ where $Kh_K$ denotes the Poincar\'{e} polynomial of $Kh(K)$. The Jones polynomial of knotoids is categorified by $Kh(K)$ in the sense that $Kh_K(t=-1,q)=J_K(q)$ for knotoid $K$.
\end{cor}
The winding homology categorifies the Turaev polynomial in the following sense.
\begin{thm}
For a (multi-)knotoid $K$ in $S^2$, $W_K(t=-1,q,u)=T_K(q,u)$.
\end{thm}
\begin{proof}
If one ignores all the edge maps and adds up the graded degrees with alternating signs at each vertex of the cube that makes $C(K)$, the computation is identical to that of the Turaev polynomial as in equation ~\eqref{equ:extendedJones}. The addition of 1 in the definition of $q-$grading is a normalization to make the trivial knotoid have its single generator at grading $(0,0,0)$, instead of $(0,-1,0)$. By equation ~\eqref{equ:gradedeulers}, we have $W_K(t=-1,q,u)=\hat{\chi}(H(K))=\hat{\chi}(C(K))=T_K(q,u)$.
\end{proof}

\subsection{Properties of the winding homology} We examine the behavior of the winding homology $H(K)$ under orientation reversal, taking mirror image, taking symmetric reflection (see Figure ~\ref{fig:bifoils}) and knotoid multiplication. We will also consider the connected sum and disjoint union of a knotoid (with single component in the case of connected sum) and a knot.
\begin{prop}\label{prop:propertiesofw}
Let $\textnormal{Rev}(K)$ be the same as the multi-knotoid $K$ but with reversed orientation on all components, $\textnormal{Mir}(K)$ be the mirror image of $K$, and $\textnormal{Sym}(K)$ be the symmetric reflection of $K$, then
\begin{align}
    W_{\textnormal{Rev}(K)}(t,q,u)&=W_{K}(t,q,u),\\
    W_{\textnormal{Mir}(K)}(t,q,u)&=W_K(t^{-1},q^{-1},u),\\
    W_{\textnormal{Sym}(K)}(t,q,u)&=W_K(t^{-1},q^{-1},u^{-1}).
\end{align}
\end{prop}
\begin{proof}
For the first identity, reversal of orientation on all components preserves the positions of all states in the cube of resolutions as well as the homological degrees and $q$-degrees. Since the leg and the head of the knotoid is switched under orientation reversal, the shortcut also reverses orientation. Thus, $k_s\cdot \alpha$ and $K\cdot \alpha$ do not change.

For the second, a 0-smoothing for a crossing in $K$ is a 1-smoothing for the same crossing in $\textnormal{Mir}(K)$, and vice versa. Therefore, there is a one-to-one correspondence between the states of $K$ and $\textnormal{Mir}(K)$. The same vector spaces are assigned to these states. By equation ~\eqref{equ:gradedeulers}, one can ignore the differentials and only consider the gradings of the generators of the chain complexes $C(K)$ and $C(\textnormal{Mir}(K))$. Indeed, there is a one-to-one correspondence of generators as follows: for a generator $v \in C(K)$ coming from state $s$, consider the state $\overline{s}$ of $\textnormal{Mir}(K)$ obtained from $s$ by changing 0's and 1's. Then we define $\overline{v} \in C(\textnormal{Mir}(K))$ as coming from the state $\overline{s}$ and obtained from $v$ by switching 1's and $X$'s only on the closed components of the resolution. Then we have 
\begin{align}
    i(\overline{v})&=|\!|\overline{v}|\!|-\overline{n_-}=n-|\!|v|\!|-n_+ \\
                &=-(|\!|v|\!|-n_-)=-i(v), \nonumber  \\
    q(\overline{v})&=\deg(\overline{v})+i(\overline{v})+\overline{n_+}-\overline{n_-}+1\\
                &=\deg(\overline{w})-1-i(v)-n_++n_-+1 \nonumber \\
                &=-\deg(w)-i(v)-n_++n_- \nonumber \\
                &=-\deg(v)-1-i(v)-n_++n_- =-q(v), \nonumber 
\end{align}
where bars over refer to terms for $\textnormal{Mir}(K)$ and $w$'s are the tensor factors of $v$ associated only with the closed components of the resolutions. Since taking the mirror image does not change the intersection numbers with the canonical shortcut, it follows that $\mu(\overline{v})=\mu(v)$.

The argument for the third identity is similar to that of the second identity except that the $u$-degrees also switch sign.
\end{proof}

For (multi-)knotoids $K_1$, $K_2$, and their product $K_1\cdot K_2$, we have 
\begin{equation}
    C(K_1\cdot K_2)\otimes X\A\{1\} \cong C(K_1)\otimes C(K_2),
\end{equation}where $\{1\}$ indicates a shift in the $q$-grading up by 1, that is, the single generator of $X\A\{1\}$ has $q$-grading 0. Using the K\"{u}nneth formula, we obtain
\begin{prop}\label{prop:product formula}
$\displaystyle H_{x,y}^z(K_1\cdot K_2)=\bigoplus_{i,j,k \in \Z} (H_{i,j}^k(K_1)\otimes H_{x-i,y-j}^{z-k}(K_2))$
\end{prop}
\noindent In other words, $W_{K_1\cdot K_2}=W_{K_1}\cdot W_{K_2}$. In particular, $W$ is invariant under change of order in knotoid multiplication.

For a knotoid $K$ (single component) and a knot $\kappa$, the connected sum $K\# \kappa$ is equivalent to the product $K\cdot \kappa^\bullet$. Thus, one can apply the Proposition ~\ref{prop:product formula}, to get a formula for $H_{i,j}^k(K\# \kappa)$.

For a multi-knotoid $K$ and a knot $\kappa$, the chain complex of their disjoint sum is given by
\begin{equation}
    C(K\sqcup\kappa)\cong C(K)\otimes CKh(\kappa),
\end{equation}
where $CKh(\kappa)$ is the unreduced Khovanov homology of $\kappa$. Here, the segment component of the knotoid $K\sqcup\kappa$ is the same as the segment component of $K$, $\kappa$ is considered a closed component of $K\sqcup\kappa$. Then, we similarly obtain the following, by K\"{u}nneth formula.
\begin{prop}
$\displaystyle H_{x,y}^z(K\sqcup \kappa)=\bigoplus_{i,j,k \in \Z} (H_{i,j}^k(K)\otimes \widehat{Kh}_{x-i,y-j}^{z-k}(\kappa))$,
where $\widehat{Kh}(\kappa)$ is the unreduced Khovanov homology of the knot $\kappa$.
\end{prop}

%------------------------------------------------------------------------

\section{Invariance of \texorpdfstring{\(H(K)\)}{}}\label{sec.invariance}

\subsection{Independence of the ordering of crossings}\label{sec:indpfromorder}
\begin{defn}
For an $m$ dimensional vector space $V$ with an ordered basis, and the one-dimensional exterior power $\bigwedge^m V$, the \emph{sign} $\text{sgn}(x)$ of a generator $x$ of the exterior power is the sign of the permutation that takes $x$ to the canonical generator $[x]$ of the exterior power. Here, the word ``permutation" refers to the permutation of the basis elements of $V$ as wedge factors in the elements of $\bigwedge^m V$. In other words, $\text{sgn}(x)=(-1)^{N(\sigma)}$ where $N(\sigma)$ is the number of adjacent transpositions in $\sigma$, which sends $x$ to the canonical generator $[x]$ by ordering the factors of $x$.
\end{defn}
Let $C=\{c_1,\dots,c_n\}$ with $c_1<\dots<c_n$, and $D=\{d_1,\dots,d_n\}$ with $d_1<\dots<d_n$ be two ordered sets denoting the crossings of $K$. The correspondence between the two sets is given by the map $\rho: C\rightarrow D$ with $\rho(c_i)=d_{\pi(i)}$, where $\pi \in \text{Sym}(n)$ is a permutation of $I=\{1,\dots,n\}$. For $J\subset I$, let $C_J=\{c_j\in C|j\in J \}$ be a subset of $C$ with inherited ordering, and $\V_J$ be a $\Q$-vector space with basis $C_J$. Similarly, we denote $\rho(C_J)$ as $D_{\pi(J)}$, and the associated vector space as $\V_{\pi(J)}$. Writing $C_J=\{c_{j_1},\dots,c_{j_{|J|}}\}$ with $j_1<\dots<j_{|J|}$, we set $c^*=c_{j_1} \wedge \dots \wedge c_{j_{|J|}}$ so that $c^*$ is the canonical generator $[c^*]\in \bigwedge^{|J|}\V_J$. Let $d^*=\rho(c^*)=d_{\pi(j_1)}\wedge \dots \wedge d_{\pi(j_{|J|})}$ with $[d^*]$ denoting the canonical generator in $\bigwedge^{|J|}\V_{\pi(J)}$, which is equal to $\text{sgn}(d^*)d^*$.

Suppose that $C(K)$ is constructed using the crossing set $C$, and $C'(K)$ by using $D$. Clearly, the generators of the chain groups $C_i(K)$ and $C'_i(K)$ are identical except for the exterior power factors. We define the map $\Phi:C(K)\rightarrow C'(K)$ such that $\Phi(v\otimes [c^*])=\text{sgn}(d^*) v'\otimes [d^*]$, where $v$, and $v'$ are identical generators of belonging to summands in $C(K)$, and $C'(K)$ of isotopic states. To show that $\Phi$ is a chain map, we only need to show that the following diagram commutes for all $c_x\not\in C_J$
\begin{center}
    \begin{tikzcd}[row sep=.5in, column sep=1in]
        \makebox{$[c^*]$} \arrow[scale=2,r, "\makebox{$(-1)^{\text{pos}_{c^*}(c_x)}$}"] \arrow[d,"\text{sgn}(d^*)"]
        &\makebox{$[c_x \wedge c^* ]$}\arrow[d, "\text{sgn}(d_{\pi(x)}\wedge d^*)"] \\
        \makebox{$[d^*]$} \arrow[r, "(-1)^{\text{pos}_{d^*}(d_{\pi(x)})}"]
        & \makebox{$[d_{\pi(x)}\wedge d^*]$}
    \end{tikzcd}
\end{center}
where $\text{pos}_{c^*}(c_x)=N(\sigma_{c^*}(c_x))$, i.e. the number of adjacent transpositions of the permutation $\sigma_{c^*}(c_x)$ bringing $c_x$ from the leftmost position to its proper position in the canonical generator $[c_x \wedge c^*]$. Similarly, $\text{pos}_{d^*}(d_{\pi(x)})=N(\sigma_{d^*}(d_{\pi(x)}))$. Commutativity of the diagram follows from the equality
\begin{equation}
    \text{sgn}(d_{\pi(x)}\wedge d^*)=(-1)^{\text{pos}_{d^*}(d_{\pi(x)})+\text{pos}_{c^*}(c_x)}\text{sgn}(d^*).
\end{equation}
To see this, the map $\rho: [c_x \wedge c^* ] \mapsto [d_{\pi(x)}\wedge d^*]$ is written as the composition 
\begin{equation*}
    [c_x \wedge c^* ] \xmapsto{(-1)^{\text{pos}_{c^*}(c_x)}} c_x \wedge c^* \xmapsto{\text{sgn}(d^*)} d_{\pi(x)}\wedge d^* \xmapsto{(-1)^{\text{pos}_{d^*}(d_{\pi(x)})}} [d_{\pi(x)}\wedge d^*].
\end{equation*}
We end by noting that $\Phi$ is invertible.

\subsection{Independence of the choice of a shortcut}
In the construction of the chain complex $C(K)$, a choice of shortcut $\alpha$ is made to define the $u$-gradings of all generators $v$ of the vector space associated with the state $s$. More precisely, the $u$-grading of all generators $v\in \A^{\otimes |s|-1}\otimes X\A \otimes \bigwedge^{|\!|s|\!|}\V_s$ is given by $\mu(v)=k_s\cdot \alpha-K\cdot \alpha$. 

Any two shortcuts $\alpha_1$ and $\alpha_2$ for a (multi-)knotoid $K$ in $S^2$ are related by
\begin{itemize}
    \item[i.] passing a small arc of the shortcut through an arc of $K$ that creates two extra intersection points,
    \item[ii.] passing the shortcut through a crossing of $K$, or
    \item[iii.] adding a spiral to the shortcut near the head or leg of $K$ that creates an extra intersection point.
\end{itemize}
For the first two of these moves, $k_s\cdot \alpha_1=k_s\cdot \alpha_2$ and $K\cdot \alpha_1=K\cdot \alpha_2$. For the third move, we have $k_s\cdot \alpha_1-K\cdot \alpha_1=k_s\cdot \alpha_2-K\cdot \alpha_2$, since the local orientations of $k_s$ and $K$ near the head or the leg agree for all states. Therefore, the $u$-gradings of the generators are preserved, and the triply graded chain complexes obtained from $\alpha_1$ and $\alpha_2$ are identical.

\subsection{Invariance under the Reidemeister move I}
To prove invariance under the Reidemeister moves, we will make use of the following well-known fact; also see Section 2.1 of ~\cite{BHL}.
\begin{lem}[Zigzag or cancellation lemma]
Let $C$ be a freely generated chain complex. For generators $x$, $y$, $a\in C$, suppose that \begin{equation}
    d(x)=y+d_{\hat{y}}(x) \,\, \text{ and }\,\,\, d(a)=a_y y+d_{\hat{y}}(a),
\end{equation}
where $a_y$ denotes the coefficient of $y$ in $d(a)$, and $d_{\hat{y}}(x)$ (resp. $d_{\hat{y}}(a)$) denotes the image of $x$ (resp. $a$) under $d$, in the free subgroup of $C$ with all the same generators except $y$. Then $(C,d)$ is chain homotopy equivalent to $(C',d')$ where $C'= C/\langle x,y\rangle$, and $d'$ is given by
\begin{equation}
    d'([a])=[d_{\hat{y}}(a)]-a_y[d_{\hat{y}}(x)].
\end{equation}
\end{lem}
\begin{proof} We give a diagrammatic proof in Figure ~\ref{fig:zigzag}. The first two steps are change of basis. The last one is a chain homotopy equivalence with inclusion and projection maps. 
\end{proof}
\begin{figure}[hbt!]
    \centering
    \includegraphics[scale=.65]{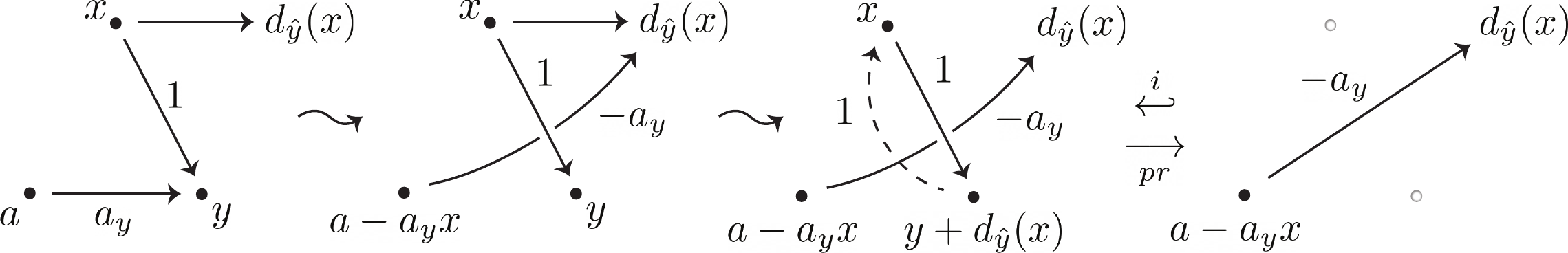}
    \caption{Proof of the zigzag lemma.}
    \label{fig:zigzag}
\end{figure}

Now, we show that there is a ($q$,$u$)-grading preserving chain homotopy equivalence from the chain complex $C(\NegCros)$ with a negative twist to $C(\Arc)$ with the twist removed. Without loss of generality, the crossing in question, labeled as $c$, is assumed to be the first in the ordering. To simplify the notation, from now on, we drop the wedge product symbols, and write $[abc\dots]$ instead of $[a\wedge b\wedge c\wedge \dots]$. Note that the chain complexes $C(\NegCros)$ and $C(\NoCirc)[\![-1]\!] \oplus C(\WithCirc)$ are isomorphic as triply-graded vector spaces, where the number inside the square bracket represents the amount of homological grading shift, i.e. $C_{i,j}^k(\NoCirc)[\![-1]\!]=C_{i+1,j}^k(\NoCirc)$. Therefore; the generators of $C(\NegCros)$, together with (some of) the differentials between them, is diagrammatically listed in Figure ~\ref{fig:negtwist}.

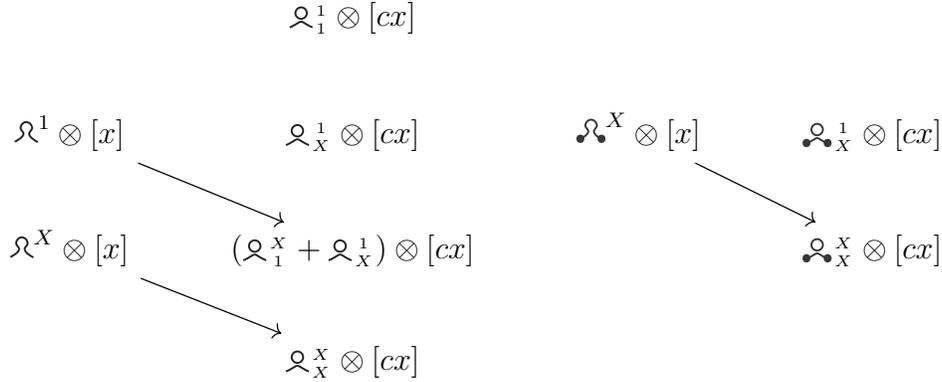
\begin{figure}[hbt!]
    \centering
       \begin{tikzcd}
            & \vcenter{\hbox{\WithCirc}}\myfrac{1}{1}\otimes [cx] & & \\
 \NoCirc^1 \otimes [x]\arrow[rd] & \vcenter{\hbox{\WithCirc}}\myfrac{1}{X}\otimes [cx] & \NoCircSeg^X \otimes [x] \arrow[rd] & \vcenter{\hbox{\WithCircSeg}}\myfrac{1}{X}\otimes [cx]\\
 \NoCirc^X \otimes [x] \arrow[rd] & (\vcenter{\hbox{\WithCirc}}\myfrac{X}{1}+\vcenter{\hbox{\WithCirc}}\myfrac{1}{X})\otimes [cx] & & \vcenter{\hbox{\WithCircSeg}}\myfrac{X}{X}\otimes [cx]\\
  & \vcenter{\hbox{\WithCirc}}\myfrac{X}{X}\otimes [cx] & &
    \end{tikzcd}
    \caption{Components of $d$, for the negative twist, between the generators that are distinct inside the local disk.}
    \label{fig:negtwist}
\end{figure}
Here, each small drawing with labels ($1$ or $X$) represents a set of generators that agrees with the local picture of the state and the label(s) on the component(s). The dotted intervals (on the 3rd and 4th columns) mean that the pictures are part of the segment components. Also, there are no negative signs on the arrows since we assumed $c$ to be the first crossing in the ordering. Using the change of basis on $C(\WithCirc)[\![1]\!]$ as in the diagram, each arrow represents a bijection between the sets of generators. These bijections are referred as \emph{type 1 arrows}, which means that they are components of $d$ between generators with different local resolutions. We would like to cancel all type 1 arrows using the cancellation lemma such that the set of arrows gets smaller in a monotonic fashion. However, these arrows are only some components of the total differential $d$. For example, there could be other arrows among the generators represented by the set $\vcenter{\hbox{\WithCirc}}\myfrac{X}{X}\otimes [cx]$. Any component of $d$ that is not of type 1 is referred as an \emph{arrow of type 2}. It is possible that arrows of type 1 form zigzags when considered together with arrows of type 2; see Figure~\ref{fig:zigzagcube}.
\begin{figure}[hbt!]
    \centering
    \begin{tikzcd}
        & D_{100} \arrow[r, red] \arrow[rd]& D_{110}\arrow[from=dl, red, shift right=1ex, crossing over] \arrow[rd] &\\
    D_{000} \arrow[ru, blue, "\underline{3}", dashed] \arrow[r] \arrow[rd] & D_{010} \arrow[ru, crossing over, blue, "\underline{2}", near end, dashed] & D_{101} \arrow[r] & D_{111}\\
        & D_{001} \arrow[ru, blue, "\underline{2}", near end, dashed] \arrow[r] & D_{011}\arrow[from=ul, crossing over, red]\arrow[ru, blue, "\underline{1}", dashed] &
    \end{tikzcd}
    \caption{Dashed blue arrows are the type 1 arrows to be cancelled. Notice that the type 1 arrow from $D_{010}$ to $D_{110}$ together with the type two arrows from $D_{100}$ to $D_{110}$, and from $D_{010}$ to $D_{011}$ form the red zigzag. Therefore, cancelling the type 1 arrow from $D_{010}$ to $D_{110}$ first would create a new arrow from $D_{100}$ to $D_{011}$. This would be against the monotonic reduction strategy. Underlined numbers describe the proper cancellation order of type 1 arrows, so that the set of arrows gets smaller monotonically.}
    \label{fig:zigzagcube}
\end{figure}
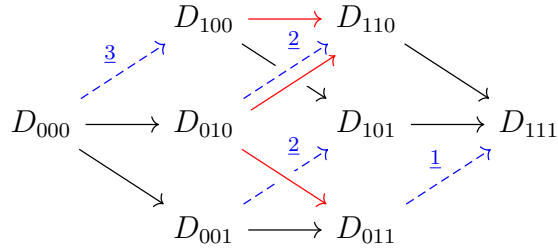
Cancellation of such type 1 arrows would introduce new arrows, which goes against the monotic reduction strategy. To avoid this possibility, we cancel the arrows in a specific order. Note that there is a single type 1 arrow and no type 2 arrows coming out of each generator at state $D_{01\cdots1}$. These arrows can be cancelled without creating new arrows. Since no generators are then left at state $D_{01\cdots1}$, there are no type 2 arrows coming out of generators at states $D_{001\cdots1}$, $D_{0101\cdots1}$, \dots, $D_{011\cdots10}$. Similarly, there is a single type 1 arrow coming out of each generator at these states to the states $D_{101\cdots1}$, $D_{1101\cdots1}$, \dots, $D_{111\cdots10}$; see Figure ~\ref{fig:zigzagcube}. Thus, these arrows can also be cancelled monotonically. Continuing by induction, all arrows of type 1 are cancelled in a monotonic fashion. The resulting chain complex is chain homotopy equivalent to $C(\Arc)$ by the maps:
\begin{align}
    \vcenter{\hbox{\WithCirc}}\myfrac{1}{1}\otimes [cx]& \xmapsto{(-1)^i} \vcenter{\hbox{\Arc}}^1\otimes[x]\label{equ:r1che1},\\
      \vcenter{\hbox{\WithCirc}}\myfrac{1}{X}\otimes [cx]& \xmapsto{(-1)^i} \vcenter{\hbox{\Arc}}^X\otimes[x],\\
        \vcenter{\hbox{\WithCircSeg}}\myfrac{1}{X}\otimes [cx]& \xmapsto{(-1)^i} \vcenter{\hbox{\ArcSeg}}^X\otimes[x],\label{equ:r1che3}
\end{align}
where the sign $(-1)^i$ makes commutative squares by accounting for the sign difference between the parallel edge maps caused by the extra crossing $c$ in the front.

These maps preserve the homological grading ($i(v)=|\!|v|\!|-n_-$) since the terms on the left come from a knotoid with an extra negative crossing. The $q$-grading is preserved similarly. The $u$-grading is preserved because the canonical shortcut has the same algebraic number of intersections with both diagrams on left and on the right of the maps ~\eqref{equ:r1che1}\,-~\eqref{equ:r1che3}, since the two intersections on the left cancel each other regardless of whether the shortcut goes around the small circle component or makes a twist inside the circle; see Figure~\ref{fig:R1uDeg}.
\begin{figure}[hbt!]
    \centering
    \includegraphics[width=.7\textwidth]{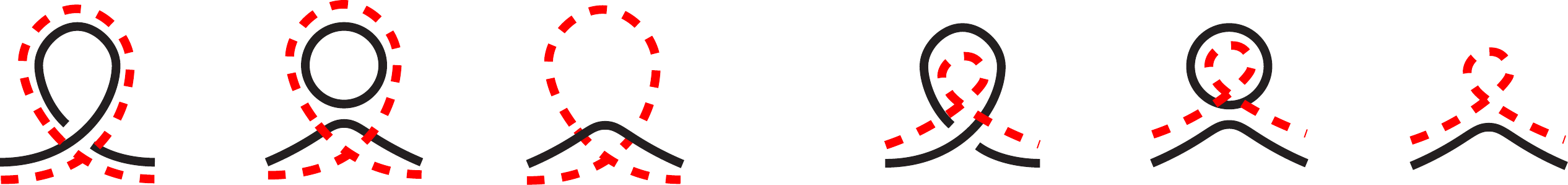}
    \caption{Possibilities for the intersections of the canonical shortcut (red dashed line) with a twist and its resolutions.} 
    \label{fig:R1uDeg}
\end{figure}
Finally, we point out that the chain homotopy equivalences obtained from the applications of the zigzag lemma above preserve the $(q,u)$-grading. In Figure ~\ref{fig:zigzag}, the generator $a$ is replaced by $a-a_y x$ after the cancellation. Since $a$ and $x$ map to the same generator, and $d$ preserves the $(q,u)$-grading, the replacement $a\rightsquigarrow a-a_y x$ also preserves the $(q,u)$-grading.

\begin{rem}
The cancellation process of arrows can be considered as a spectral sequence starting with the chain complex $E_0=C(\NegCros)$ and converging to $E_\infty=E_n=C(\Arc)$, where the page $E_{k\geq 1}$ is obtained by monotonically cancelling all type 1 arrows from all states $D_{0\alpha}$ with the string $\alpha$ containing $k-1$ zeros. The pages of the spectral sequence are chain homotopy equivalent.
\end{rem}

The case for the positive twist, $C(\reflectbox{\NegCros})\overset{c.h.e.}{\sim} C(\Arc)$, follows similarly from the diagrammatic listing of the generators and type 1 arrows as shown in Figure ~\ref{fig:postwist}.
\begin{figure}[hbt!]
    \centering
   \begin{tikzcd}
    \vcenter{\hbox{\WithCirc}}\myfrac{1}{1}\otimes [x] \arrow[rd] & & & \\
 \vcenter{\hbox{\WithCirc}}\myfrac{1}{X}\otimes [x] \arrow[rd] &\NoCirc^1 \otimes [x] & \vcenter{\hbox{\WithCircSeg}}\myfrac{1}{X}\otimes [x]\arrow[r] & \NoCircSeg^X \otimes [x]  \\
 (\vcenter{\hbox{\WithCirc}}\myfrac{X}{1}-\vcenter{\hbox{\WithCirc}}\myfrac{1}{X})\otimes [x] &\NoCirc^X \otimes [x]  & \vcenter{\hbox{\WithCircSeg}}\myfrac{X}{X}\otimes [x] & \\
  \vcenter{\hbox{\WithCirc}}\myfrac{X}{X}\otimes [x] & & & 
    \end{tikzcd} 
    \caption{Type 1 arrows for the positive twist.}
    \label{fig:postwist}
\end{figure}
After the cancellations, the resulting complex is equivalent to $C(\Arc)$ by the following maps:
\begin{align}
    (\vcenter{\hbox{\WithCirc}}\myfrac{X}{1}-\vcenter{\hbox{\WithCirc}}\myfrac{1}{X}) \otimes [x]&\xmapsto{\hspace{.3in}} \vcenter{\hbox{\Arc}}^1\otimes[x],\\
    \vcenter{\hbox{\WithCirc}}\myfrac{X}{X} \otimes [x]&\xmapsto{\hspace{.3in}} \vcenter{\hbox{\Arc}}^X\otimes[x],\\
    \vcenter{\hbox{\WithCircSeg}}\myfrac{X}{X} \otimes [x]&\xmapsto{\hspace{.3in}} \vcenter{\hbox{\ArcSeg}}^X\otimes[x].\label{equ:r1che6}
\end{align}

\subsection{Invariance under the Reidemeister move II} In the picture $\RIIcros$, let the left and right crossings be $c_1$ and $c_2$, respectively. We assume that $c_1<c_2<\cdots$. Using the cancellation lemma again, we show that $C(\RIIcros )\overset{c.h.e.}{\sim}C(\RIInocros)$. In contrast to the previous case, there are four boundary points of the (local) tangle. A complete resolution can connect these end points in 8 different ways outside the local picture -- 2 come from those that involve only closed components and 6 from those that involve the segment component. Up to symmetry, there are five cases to be considered.

Before we start listing and working through the cases, it is worth pointing out that we will present a way of unifying all five cases into a diagram using a symbolic notation in Remark~\ref{rem:r2summary} afterwards. The cases are listed to check the validity of this unified diagram.

\textit{Case 1.} The left (resp. right) end points connect to each other in the complete resolution, and the segment component is not involved. After a change of basis, we obtain the diagram in Figure ~\ref{fig:case1}.
\begin{figure}[hbt!]
    \centering
\begin{tikzcd}[row sep=1ex,column sep=13ex]
            & \LeftSCircRightS \otimes [c_1 x] & \\
            & \LeftSoft^1 \SimpCirc^1 \RightSoft^1 \arrow[rdddd,"-"] & \\
            & \LeftSoft^1\SimpCirc^1\RightSoft^X \arrow[rdddd,"-"] & \\
            & \LeftSoft^X\SimpCirc^1\RightSoft^1 \arrow[rdddd,"-"] & \\
    \LeftBRightS \otimes [x] & \LeftSoft^X\SimpCirc^1\RightSoft^X \arrow[rdddd,"-"] & \reflectbox{$\LeftBRightS$}\otimes [c_1 c_2 x] \\
    \LeftBump^1\RightSoft^1 \arrow[r] & \LeftSoft^1\SimpCirc^X\RightSoft^1+\LeftSoft^X\SimpCirc^1\RightSoft^1+\TwoWaves^1& \LeftSoft^1\RightBump^1 \\
    \LeftBump^1\RightSoft^X \arrow[r]& \LeftSoft^1\SimpCirc^X\RightSoft^X+\LeftSoft^X\SimpCirc^1\RightSoft^X+\TwoWaves^X& \LeftSoft^1\RightBump^X \\
    \LeftBump^X\RightSoft^1 \arrow[r] & \LeftSoft^X\SimpCirc^X\RightSoft^1+\TwoWaves^X & \LeftSoft^X\RightBump^1 \\  
    \LeftBump^X\RightSoft^X \arrow[r] & \LeftSoft^X\SimpCirc^X\RightSoft^X & \LeftSoft^X\RightBump^X \\
    & & \\
    & & \\
            & \WavesClosed \otimes [c_2 x]& \\
            & \TwoWaves^1 + \LeftSoft^1\SimpCirc^1\RightSoft^X + \LeftSoft^X\SimpCirc^1\RightSoft^1 & \\
            & \TwoWaves^X + \LeftSoft^X\SimpCirc^1\RightSoft^X &
\end{tikzcd} 
    \caption{Case 1: After the change of basis, arrows of type 1 become bijections and are ready to be cancelled monotonically.}
    \label{fig:case1}
\end{figure}
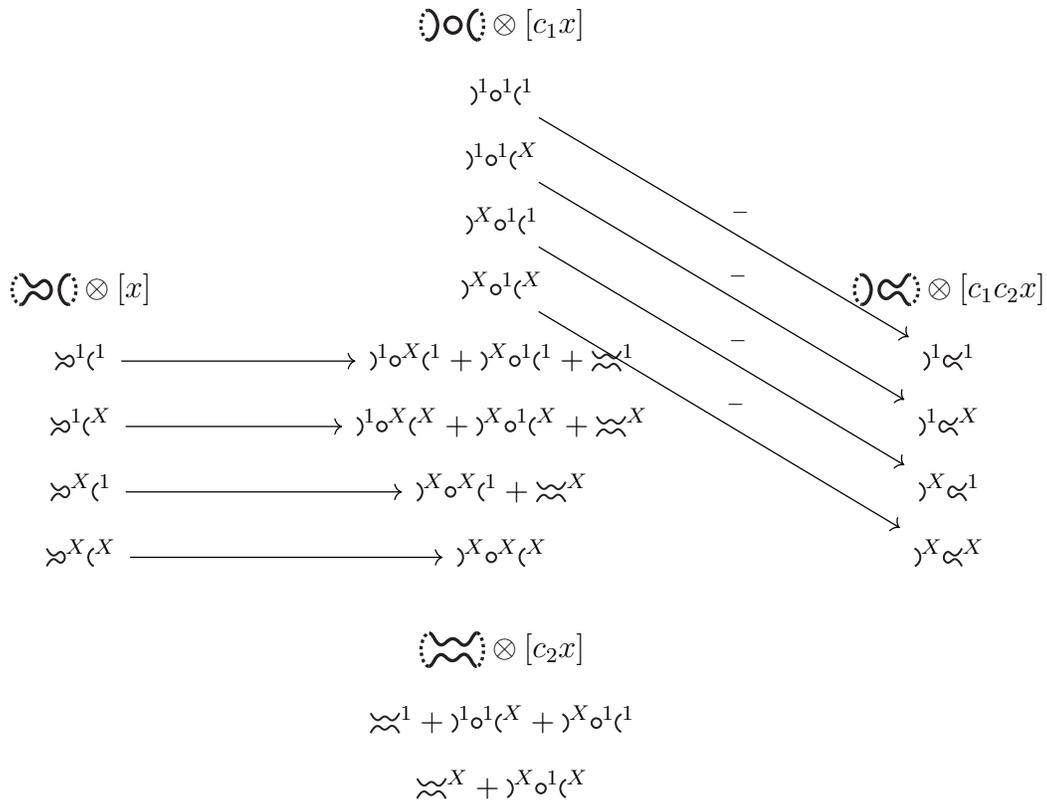
It is important to observe that, with this change of basis, all type 1 arrows are represented in the diagram. For example, the generator $\LeftSoft^1\SimpCirc^X\RightSoft^1+\LeftSoft^X\SimpCirc^1\RightSoft^1+\TwoWaves^1$ maps to zero, as well as others that have no arrows on them. Then, we monotonically cancel the type 1 arrows with the same trick of cancelling from right to left within each sub-cube. More precisely, the arrows at states $D_{101\cdots1}$ and $D_{001\cdots1}$ are cancelled first without introducing any new arrows, since there are no type 2 arrows at these positions. Then, we move on to the states $D_{1001\cdots1},D_{10101\cdots1},\dots,D_{101\cdots10}$; and $D_{0001\cdots1},D_{00101\cdots1},\dots,D_{001\cdots10}$ ...etc. The left over generators in the resulting chain complex are sent to $C(\RIInocros)$ by the chain homotopy equivalence:
\begin{align}
    \TwoWaves^1\otimes [c_2 x] + \LeftSoft^1 \SimpCirc^1 \RightSoft^X \otimes [c_1 x] + \LeftSoft^X \SimpCirc^1 \RightSoft^1 \otimes [c_1 x] &\xmapsto{(-1)^i} \RIInocros^1 \otimes [x],\label{equ:cheR2first}\\
    \TwoWaves^X\otimes [c_2 x] + \LeftSoft^X\SimpCirc^1\RightSoft^X \otimes [c_1 x] &\xmapsto{(-1)^i} \RIInocros^X \otimes [x].
\end{align}

\textit{Case 2.} The top (resp. bottom) end points connect to each other in the complete resolution, and the segment component is not involved. Same argument as in the first case holds with the diagram in Figure ~\ref{fig:case2}.
\begin{figure}[hbt!]
    \centering
\begin{tikzcd}[row sep=.5ex,column sep=12ex]
 & \CaseIILSCRS \otimes [c_1 x]& \\
 & \LeftSoft\,\,\SimpCirc^1\RightSoft^1 \arrow[rdddd,"-"] & \\
 & \LeftSoft\,\,\SimpCirc^1\RightSoft^X \arrow[rdddd,"-"]& \\
 & \LeftSoft\,\,\SimpCirc^X\RightSoft^1+\LeftSoft\,\,\SimpCirc^1\RightSoft^X+\TwoWaves\myfrac{1}{X}+\TwoWaves\myfrac{X}{1} & \\
 \CaseIILBRS \otimes [x] & \LeftSoft\,\,\SimpCirc^X\RightSoft^X+\TwoWaves\myfrac{X}{X} &  \reflectbox{$\CaseIILBRS$} \otimes [c_1 c_2 x] \\
 \LeftBump^1\RightSoft \arrow[ruu] &  & \LeftSoft^1\RightBump \\
 \LeftBump^X\RightSoft \arrow[ruu] & \CaseIITwoWaves \otimes [c_2 x] & \LeftSoft^X\RightBump \\
 & \TwoWaves\myfrac{1}{1}+\LeftSoft\,\,\SimpCirc^1\RightSoft^1  & \\
 & \TwoWaves\myfrac{1}{X}+\LeftSoft\,\,\SimpCirc^1\RightSoft^X  & \\
 & \TwoWaves\myfrac{X}{1}+\LeftSoft\,\,\SimpCirc^1\RightSoft^X  & \\
 & \TwoWaves\myfrac{X}{X} & 
\end{tikzcd}
    \caption{Case 2: Type 1 arrows.}
    \label{fig:case2}
\end{figure}
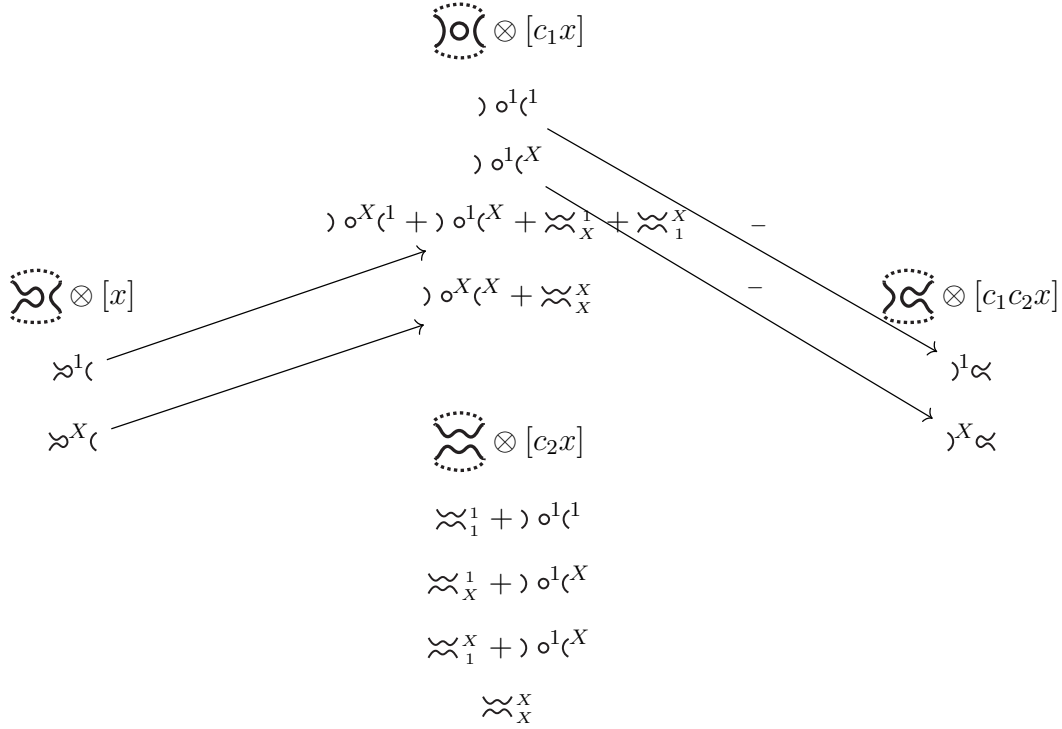
The left over generators, after cancellations, are sent to $C(\RIInocros)$ by the chain homotopy equivalence:
\begin{align}
    \TwoWaves\myfrac{1}{1}\otimes [c_2 x]+\LeftSoft\,\,\SimpCirc^1\RightSoft^1 \otimes [c_1 x] &\xmapsto{(-1)^i} \RIInocros\myfrac{1}{1} \otimes [x],\\
    \TwoWaves\myfrac{1}{X}\otimes [c_2 x]+\LeftSoft\,\,\SimpCirc^1\RightSoft^X \otimes [c_1 x] &\xmapsto{(-1)^i} \RIInocros\myfrac{1}{X} \otimes [x],\\
    \TwoWaves\myfrac{X}{1}\otimes [c_2 x]+\LeftSoft\,\,\SimpCirc^1\RightSoft^X \otimes [c_1 x] &\xmapsto{(-1)^i} \RIInocros\myfrac{X}{1} \otimes [x],\\
    \TwoWaves\myfrac{X}{X}\otimes [c_2 x] &\xmapsto{(-1)^i} \RIInocros\myfrac{X}{X} \otimes [x].
\end{align}

\textit{Case 3.} Left end points connect to each other and the right end points are part of segment component. Then we use the diagram in Figure ~\ref{fig:case3} for the same argument.
\begin{figure}[hbt!]
    \centering
\begin{tikzcd}[row sep=.5ex,column sep=12ex]
& \CaseIIIio \otimes [c_1 x] & \\
& \LeftSoft^1\SimpCirc^1\RightSoftDot^X \arrow[rdd,"-"] & \\
\CaseIIIoo \otimes [x]& \LeftSoft^X\SimpCirc^1\RightSoftDot^X \arrow[rdd,"-"] & \CaseIIIii\otimes [c_1 c_2 x]\\
\LeftBump^1\RightSoftDot^X \arrow[r] & \LeftSoft^1\SimpCirc^X\RightSoftDot^X+\LeftSoft^X\SimpCirc^1\RightSoftDot^X+\TwoWavesDot^X &  \LeftSoft^1 \RightBumpDot^X\\
\LeftBump^X\RightSoftDot^X \arrow[r] & \LeftSoft^X\SimpCirc^X\RightSoftDot^X & \LeftSoft^X \RightBumpDot^X\\ 
& \hspace{.1cm} & \\
& \CaseIIIoi \otimes [c_2 x]& \\
& \TwoWavesDot^X+\LeftSoft^X\SimpCirc^1\RightSoftDot^X &
\end{tikzcd}
    \caption{Case 3: Type 1 arrows.}
    \label{fig:case3}
\end{figure}
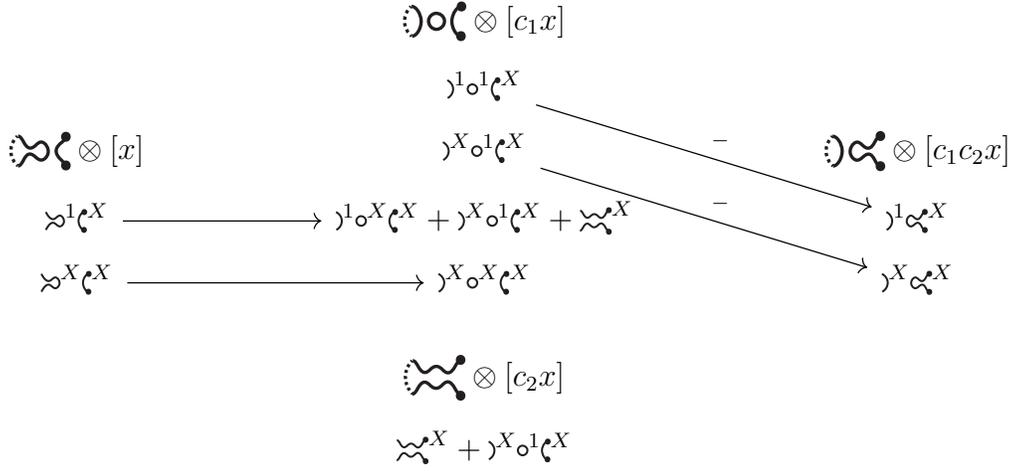
Similarly, the equivalence to $C(\RIInocros)$ is given by
\begin{equation}
    \TwoWavesDot^X\otimes [c_2x] +\LeftSoft^X\SimpCirc^1\RightSoftDot^X \otimes [c_1x] \xmapsto{(-1)^i}\RIInocrosDot^X\otimes [x].
\end{equation}

\textit{Case 4.} Top end points connect to each other and the bottom end points are part of segment component. We use the diagram in Figure ~\ref{fig:case4} for the usual argument.

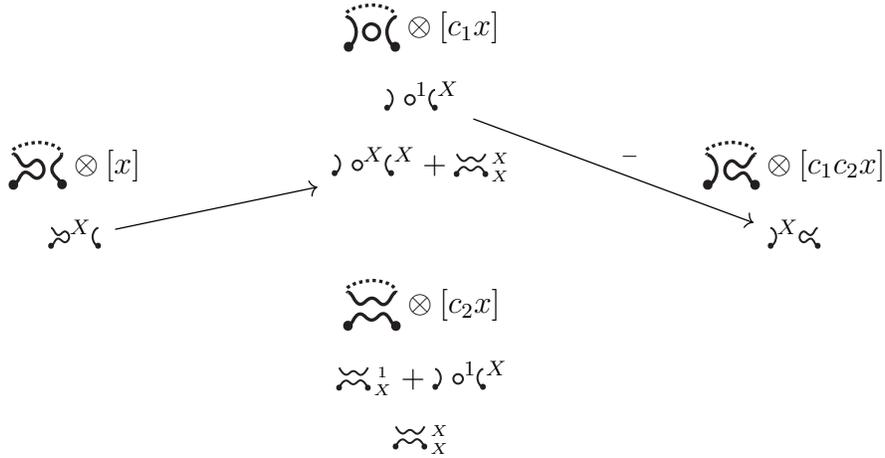
\begin{figure}[hbt!]
    \centering
\begin{tikzcd}[row sep=.5ex,column sep=12ex]
& \CaseIVio \otimes [c_1 x] & \\
& \LeftSoftSubDot\,\,\SimpCirc^1\reflectbox{$\LeftSoftSubDot$}^X \arrow[rdd,"-"]& \\
\CaseIVoo \otimes [x]& \LeftSoftSubDot\,\,\SimpCirc^X\reflectbox{$\LeftSoftSubDot$}^X+\TwoWavesSubDot\myfrac{X}{X} & \CaseIVii \otimes [c_1c_2x] \\
\LeftBumpSubDot^X\reflectbox{$\LeftSoftSubDot$}\arrow[ur]& &\LeftSoftSubDot^X\reflectbox{$\LeftBumpSubDot$}\\ 
& \CaseIVoi \otimes [c_2x] &\\
& \TwoWavesSubDot\myfrac{1}{X} + \LeftSoftSubDot\,\,\SimpCirc^1\reflectbox{$\LeftSoftSubDot$}^X \\
& \TwoWavesSubDot\myfrac{X}{X} &
\end{tikzcd}
    \caption{Case 4: Type 1 arrows.}
    \label{fig:case4}
\end{figure}
Then the equivalence map is given by
\begin{align}
    \TwoWavesSubDot\myfrac{1}{X} \otimes [c_2 x]+ \LeftSoftSubDot\,\,\SimpCirc^1\reflectbox{$\LeftSoftSubDot$}^X  \otimes [c_1 x] &\xmapsto{(-1)^i} \RIInocrosSubDot\myfrac{1}{X} \otimes [x],\\
    \TwoWavesSubDot\myfrac{X}{X} \otimes [c_2 x] &\xmapsto{(-1)^i} \RIInocrosSubDot\myfrac{X}{X} \otimes [x].
\end{align}

\textit{Case 5.} Two diagonal end points connect to each other and other two end points are part of segment component. We use the diagram in Figure~\ref{fig:case5}, where the equivalence map is given by
\begin{equation}\label{equ:case5}
    \TwoWavesDiagDot^X \otimes [c_2 x] \xmapsto{(-1)^i} \RIInocrosDiagDot^X \otimes [x] 
\end{equation}
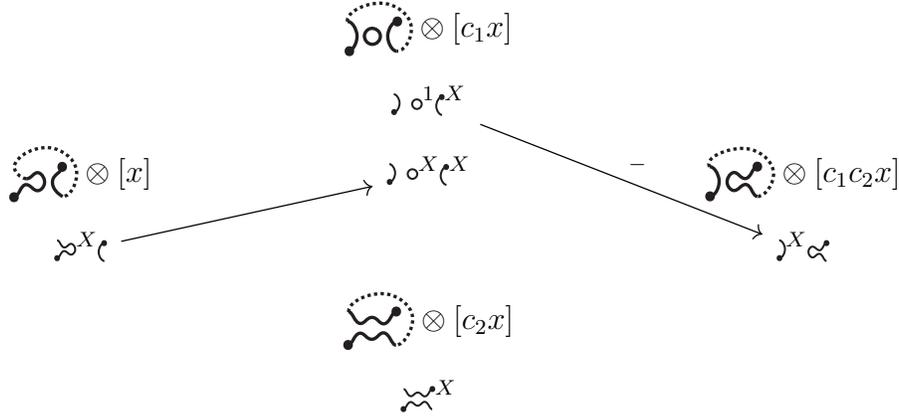
\begin{figure}[hbt!]
    \centering
\begin{tikzcd}[row sep=.5ex,column sep=12ex]
& \CaseVio \otimes [c_1 x] & \\
& \LeftSoftSubDot\,\,\SimpCirc^1\raisebox{.2cm}{\scalebox{-1}[-1]{$\LeftSoftSubDot$}}^X \arrow[rdd,"-"]& \\
\CaseVoo \otimes [x]& \LeftSoftSubDot\,\,\SimpCirc^X\raisebox{.2cm}{\scalebox{-1}[-1]{$\LeftSoftSubDot$}}^X & \CaseVii \otimes [c_1c_2x] \\
\LeftBumpSubDot^X\raisebox{.2cm}{\scalebox{-1}[-1]{$\LeftSoftSubDot$}}\arrow[ur]& &\LeftSoftSubDot^X\raisebox{.2cm}{\scalebox{-1}{$\LeftBumpSubDot$}}\\ 
& \CaseVoi \otimes [c_2x] &\\
& \TwoWavesDiagDot^X &
\end{tikzcd}
    \caption{Case 5: Type 1 arrows.}
    \label{fig:case5}
\end{figure}

Lastly, we show that chain equivalence from $C(\RIIcros )$ to $C(\RIInocros)$ preserves the $(q,u)$-gradings. The change of bases and chain homotopy equivalences involved in the zigzag lemma preserve the $(q,u)$-gradings as explained in Reidemeister move I case. So, we only need to check that the chain equivalences defined at the end of each case, namely the maps~\eqref{equ:cheR2first} through~\eqref{equ:case5}, preserve the gradings. It is easy to see that the generators on both left and right sides of the equivalence maps in all cases above have the same $i(v)=|\!|v|\!|-n_-$, and $q(v)=\deg(v)+i(v)+n_+-n_-+1$ values. To compare the values of $\mu(v)=k_s\cdot \alpha-K\cdot \alpha$ for the terms on either side of equivalence maps, we use the canonical shortcut so that $K\cdot \alpha=0$. Since choice of orientation on closed components have no effect on $k_s\cdot \alpha$, one can assign orientations on closed states without changing $u$-gradings. \begin{convention}\label{conv:orionclosedcomps}
Consider two states that differ only by the local pictures $\LeftSoft\,\SimpCirc\,\RightSoft$ and $\RIInocros$ (or $\TwoWaves$). We may assume that whenever a closed component of one of these states has a common arc with a (closed or segment) component of the other, then the closed component is oriented to have the same orientation with the other component(s) on the common arc(s). This way, the contribution of all algebraic intersection numbers coming from outside the local picture are the same, and we only need to count the intersection numbers inside.
\end{convention}

 Inside the local pictures, the diagram $\RIInocros$ has no crossings with the shortcut, the diagram $\TwoWaves$ has 2 or 4 intersections which cancel each other; see Figure ~\ref{fig:ugradingforR2}.
\begin{figure}[hbt!]
    \centering
    \includegraphics[width=.7\textwidth]{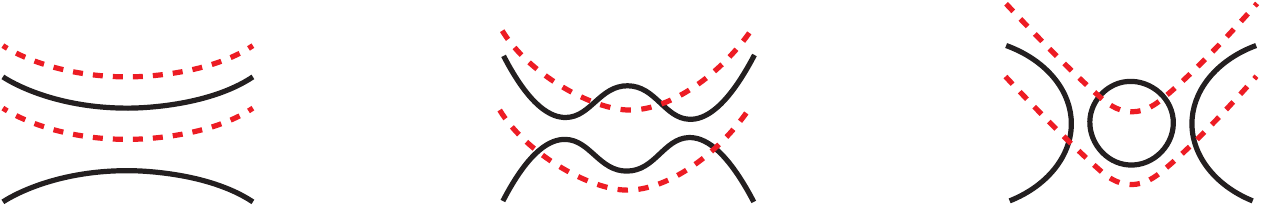}
    \caption{Possibilities for the intersections of the canonical shortcut (red dashed line) within the local picture. Horizontal reflections of the dashed lines are omitted for clarity.} 
    \label{fig:ugradingforR2}
\end{figure}
For the diagram $\LeftSoft\,\SimpCirc\,\RightSoft$, if the shortcut goes through the small circle in the middle then the two intersections cancel each other trivially. If he shortcut goes through the left and right arcs of $\LeftSoft\,\SimpCirc\,\RightSoft$, then the two intersections again cancel each other by Convention ~\ref{conv:orionclosedcomps} in all cases from 1 through 4, since the left and right arcs have opposite orientations, that is, one goes upwards, the other downwards. In case 5, the left and right arcs of $\LeftSoft\,\SimpCirc\,\RightSoft$ have the same orientation and thus the intersection number inside the local picture is not zero, yet this is not a problem because there is no $\LeftSoft\,\SimpCirc\,\RightSoft$ term on the left side of the equivalence map; see the formula ~\eqref{equ:case5}. Therefore, the $u$-grading is also preserved under our chain equivalences in all cases.

\begin{rem}\label{rem:r2summary}
The first two cases involve only the the maps $m_1$ and $\Delta_1$, hence provide a proof of invariance under Reidemeister move II for the (regular) Khovanov homology of links. In Section 3.3 of ~\cite{Jacobsson}, these cases are summarized by a symbolic notation. Even though the maps $m_2$, $\Delta_2$ and $\nabla$ are involved in the last three cases, we can still use Jacobsson's symbolic notation to abridge all five cases to a single picture as in Figure ~\ref{fig:generalcase}.
\begin{figure}[hbt!]
    \centering
\begin{tikzcd}[row sep=.5ex,column sep=12ex]
& \raisebox{-0.065cm}{\scalebox{1.8}{$\LeftSoft\,\SimpCirc\,\RightSoft$}} \otimes [c_1x]& \\
& \LeftSoft\,\,\SimpCirc^1 \RightSoft\arrow[rdd,"-"] & \\
\raisebox{-0.065cm}{\scalebox{1.8}{$\LeftBump\,\RightSoft$}} \otimes[x] & d\,(\LeftBump\,\RightSoft) & \raisebox{-0.065cm}{\scalebox{1.8}{$\LeftSoft\,\RightBump$}} \otimes [c_1c_2x]\\
\LeftBump\,\RightSoft\arrow[ru] & & \LeftSoft\,\RightBump \\
& \raisebox{-0.065cm}{\scalebox{1.8}{$\TwoWaves$}} \otimes[c_2x] & \\
& \TwoWaves\myfrac{p}{q}+\LeftSoft^{p:q}\SimpCirc^1\RightSoft^{q:p} & \\
\end{tikzcd}
    \caption{The common diagram for all five cases.}
    \label{fig:generalcase}
\end{figure}
Here, it is implied that (1) all diagrams may (or may not) connect to the head/leg of the knotoid in any possible way as in the cases above, (2) missing labels can be filled by any valid labeling within the state, and (3) the map $\raisebox{-.05cm}{\rotatebox{90}{$\LeftSoft\,\RightSoft$}} \myfrac{p}{q}\longrightarrow \raisebox{-.05cm}{$\LeftSoft^{p:q}\,\RightSoft^{q:p}$}$ stands for all possible local changes in labels under all elementary maps $m_{1,2}$, $\Delta_{1,2}$, $\nabla$. It is straight forward to check that the diagram in Figure ~\ref{fig:generalcase} with the symbolic notation specializes to any of the five cases by substituting labels for $p$, $q$, and using the definitions of the elementary maps as shown in Figure ~\ref{fig:localcoborisms and induced maps}. Then the chain equivalence maps are as follows:\[\TwoWaves\myfrac{p}{q}\otimes [c_2x]+\LeftSoft^{p:q}\SimpCirc^1\RightSoft^{q:p} \otimes [c_1x] \xmapsto{(-1)^i} \RIInocros \myfrac{p}{q} \otimes [x].\]
\end{rem}

\begin{rem}
One could also use the change of basis $\TwoWaves\myfrac{p}{q}$, instead of $\TwoWaves\myfrac{p}{q}+\LeftSoft^{p:q}\SimpCirc^1\RightSoft^{q:p}$, before the application of the cancellation lemma. That way, there would be extra type 1 arrows from $\TwoWaves$ to $\LeftSoft\,\RightBump$, which are still cancelled after the monotonic cancellation of arrows from $\LeftSoft\,\,\SimpCirc^1\RightSoft$ to $\LeftSoft\,\RightBump$. Then the chain equivalence map would reduce to \[\TwoWaves\myfrac{p}{q}\otimes [c_2x] \xmapsto{(-1)^i} \RIInocros \myfrac{p}{q} \otimes [x].\] The advantage of using the basis $\TwoWaves\myfrac{p}{q}+\LeftSoft^{p:q}\SimpCirc^1\RightSoft^{q:p}$ is that the resulting chain complex after cancellations is a subcomplex of $C(\RIIcros)$. This is explained in Section 5.3 of ~\cite{Khovanov} for the cases 1 and 2, by using a chain map induced from a (1+1)-dimensional cobordism $\TwoWaves \rightarrow \LeftSoft\,\SimpCirc\,\RightSoft$. This explanation extends to case 3-5 if one starts with a category $\mathcal{M}$ whose objects consist of a single segment component and some number of closed components. The morphisms are given by the cobordisms in Figure ~\ref{fig:localcoborisms and induced maps}, and the composition of morphisms is by stacking surfaces. The tensor product is given by joining the segment components in $S^2$ as in knotoid multiplication. Then there is a similar monoidal functor $F:\mathcal{M}\rightarrow \text{Vect}_\Q$, so that one can define a map on vector spaces induced from a cobordism $\TwoWaves \rightarrow \LeftSoft\,\SimpCirc\,\RightSoft$ in Mor($\mathcal{M}$).
\end{rem}

\subsection{Invariance under the Reidemeister move III} There are six end points on a (local) Reidemeister move III diagram. A complete resolution can connect these end points in $\binom{6}{2}\cdot C_2+C_3=35$ different ways outside, where $C_2=2$ and $C_3=5$ are the 2nd and 3rd Catalan numbers. The first summand gives the number of resolutions such that the segment component intersects the local disk, and the second gives the number of those where the segment does not intersect the local disk. Writing individual chain homotopy equivalences on generators explicitly, as before, would become unmanageable for this many cases. Instead, we will first reduce to the Reidemeister move II case, and then use Jacobsson's symbolic notation to address all possible cases at once in the manner described in Remark~\ref{rem:r2summary}. That way, it will be sufficient to apply the monotonic reduction on only two diagrams - one for each side of the Reidemeister move III.

\begin{figure}[hbt!]
    \centering
    \begin{subfigure}[b]{0.08\textwidth}
    \centering
    \includegraphics[width=\textwidth]{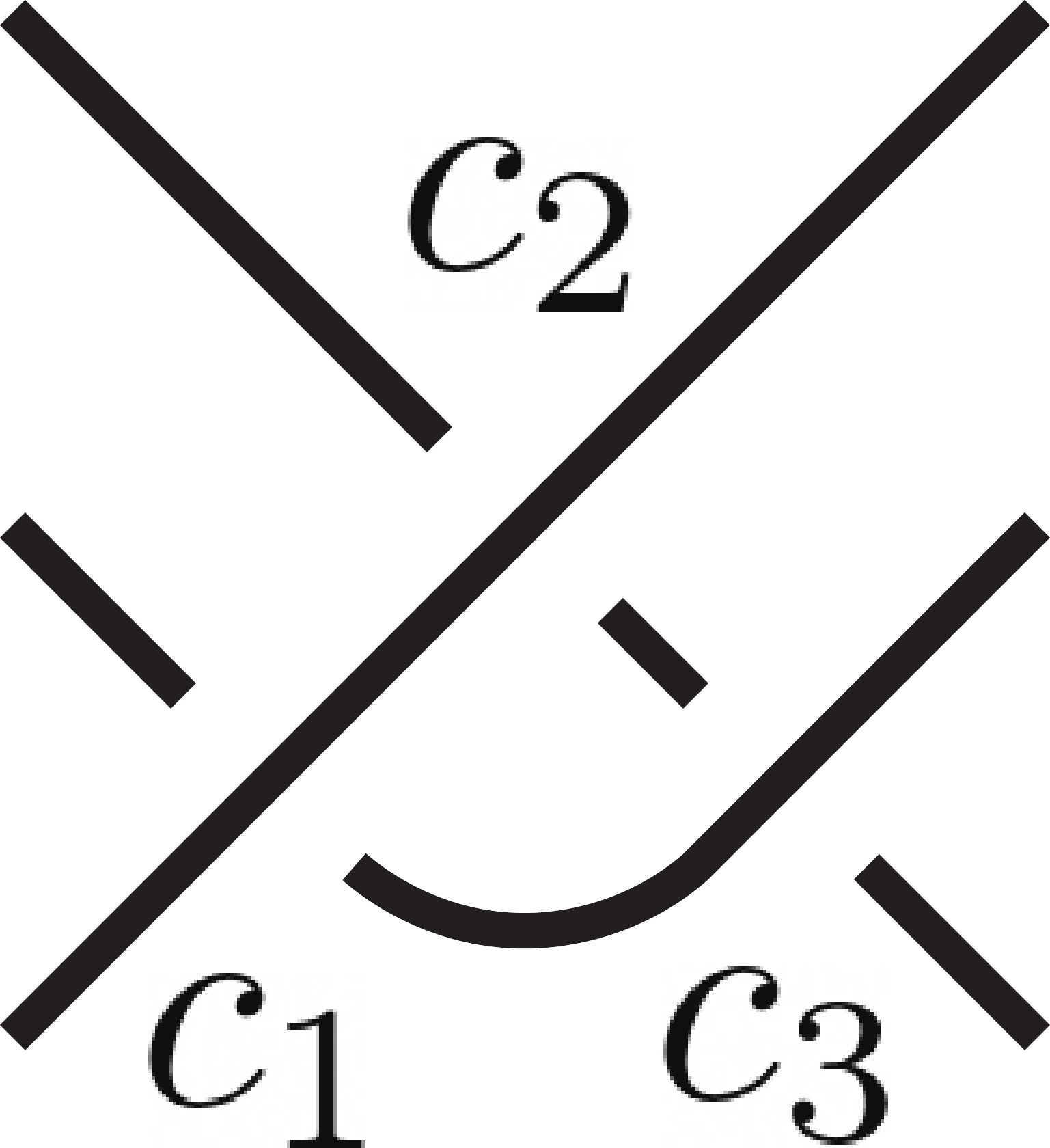}
    \caption{$D$}
    \label{fig:r3d}
    \end{subfigure}
    \hfill
    \begin{subfigure}[b]{0.08\textwidth}
    \centering
    \includegraphics[width=\textwidth]{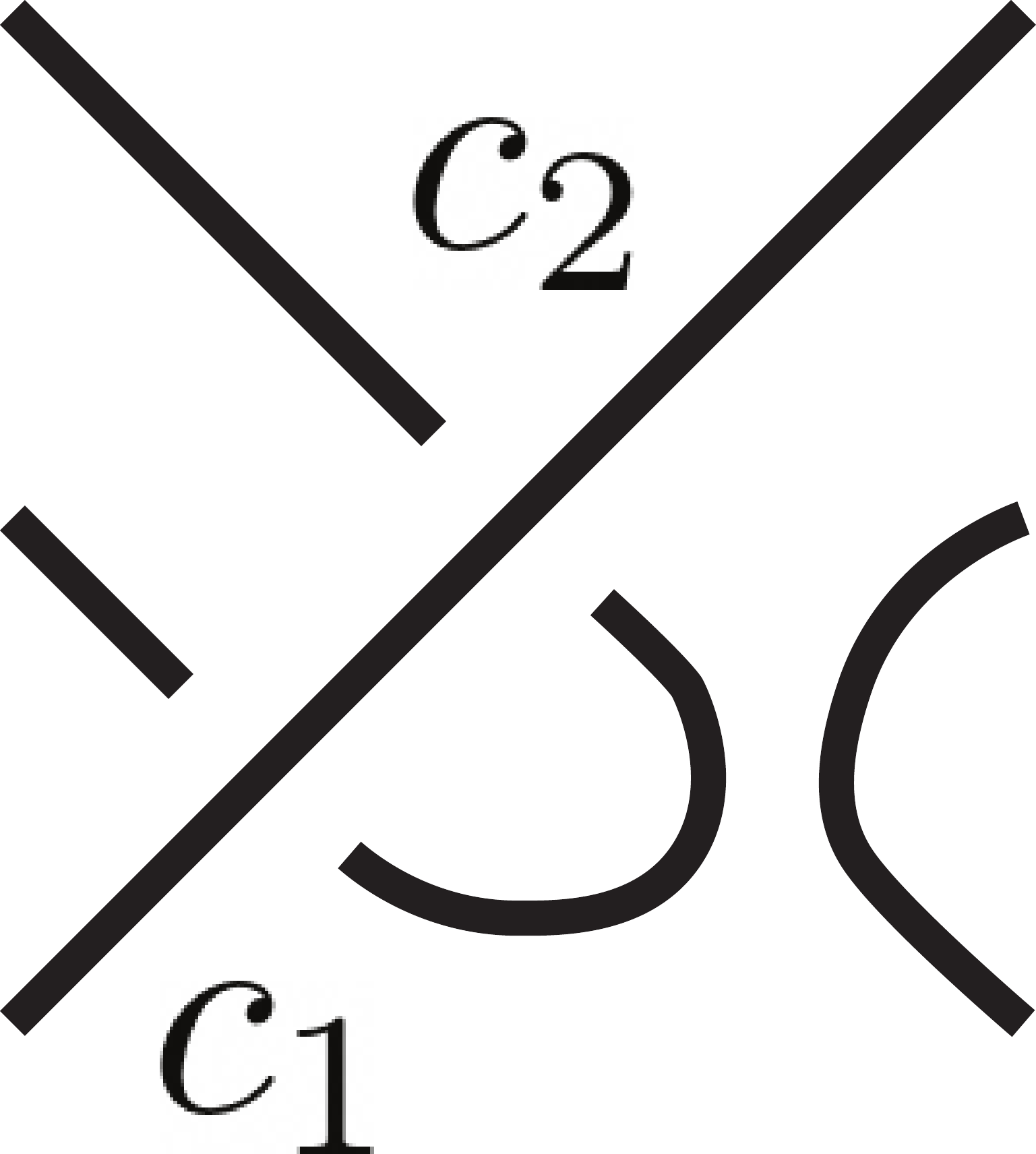}
    \caption{$D_1$}
    \label{fig:r3d1}
    \end{subfigure}
    \hfill
    \begin{subfigure}[b]{0.08\textwidth}
    \centering
    \includegraphics[width=\textwidth]{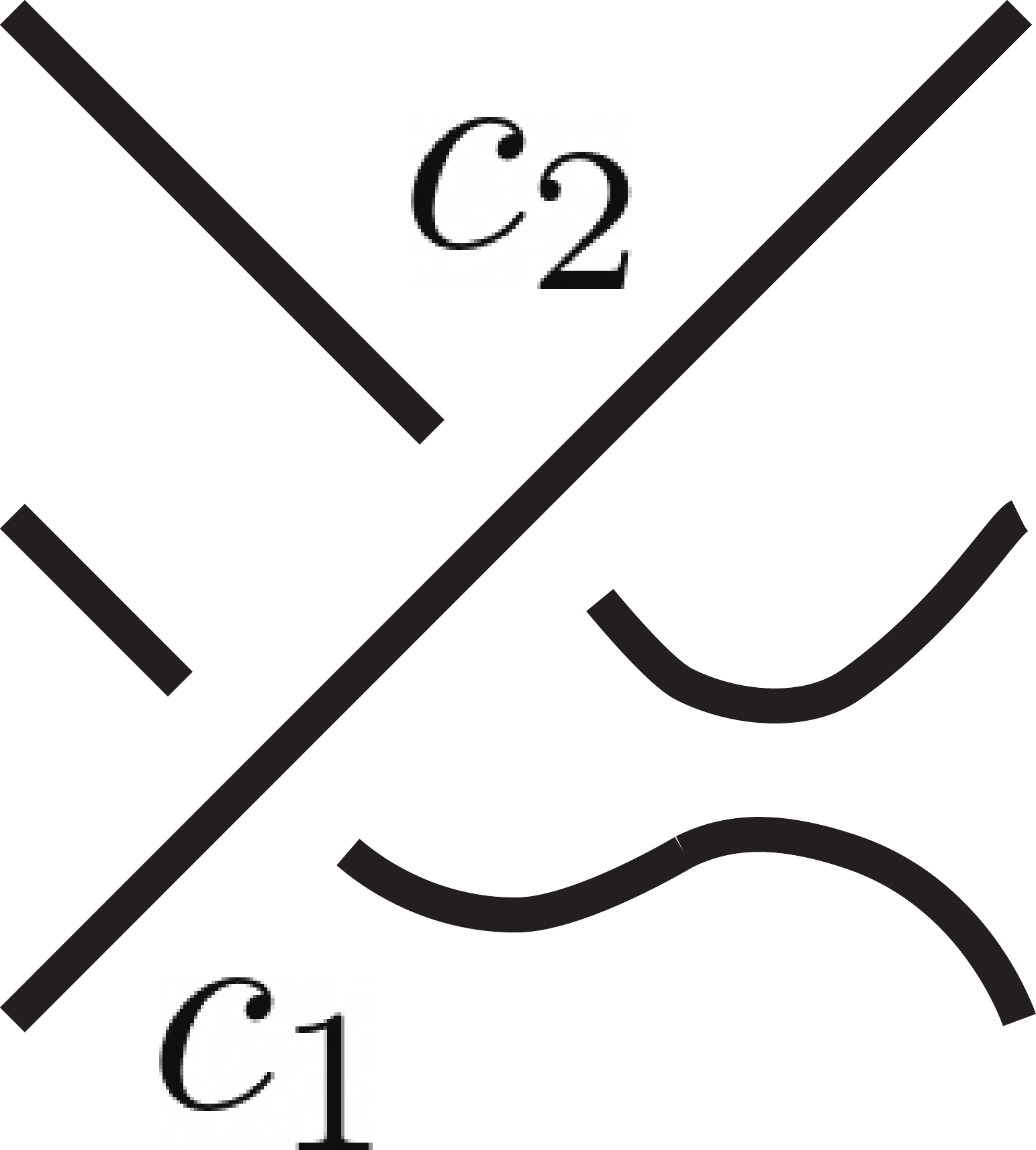}
    \caption{$D_2$}
    \label{fig:r3d2}
    \end{subfigure}
    \hfill
    \begin{subfigure}[b]{0.08\textwidth}
    \centering
    \includegraphics[width=\textwidth]{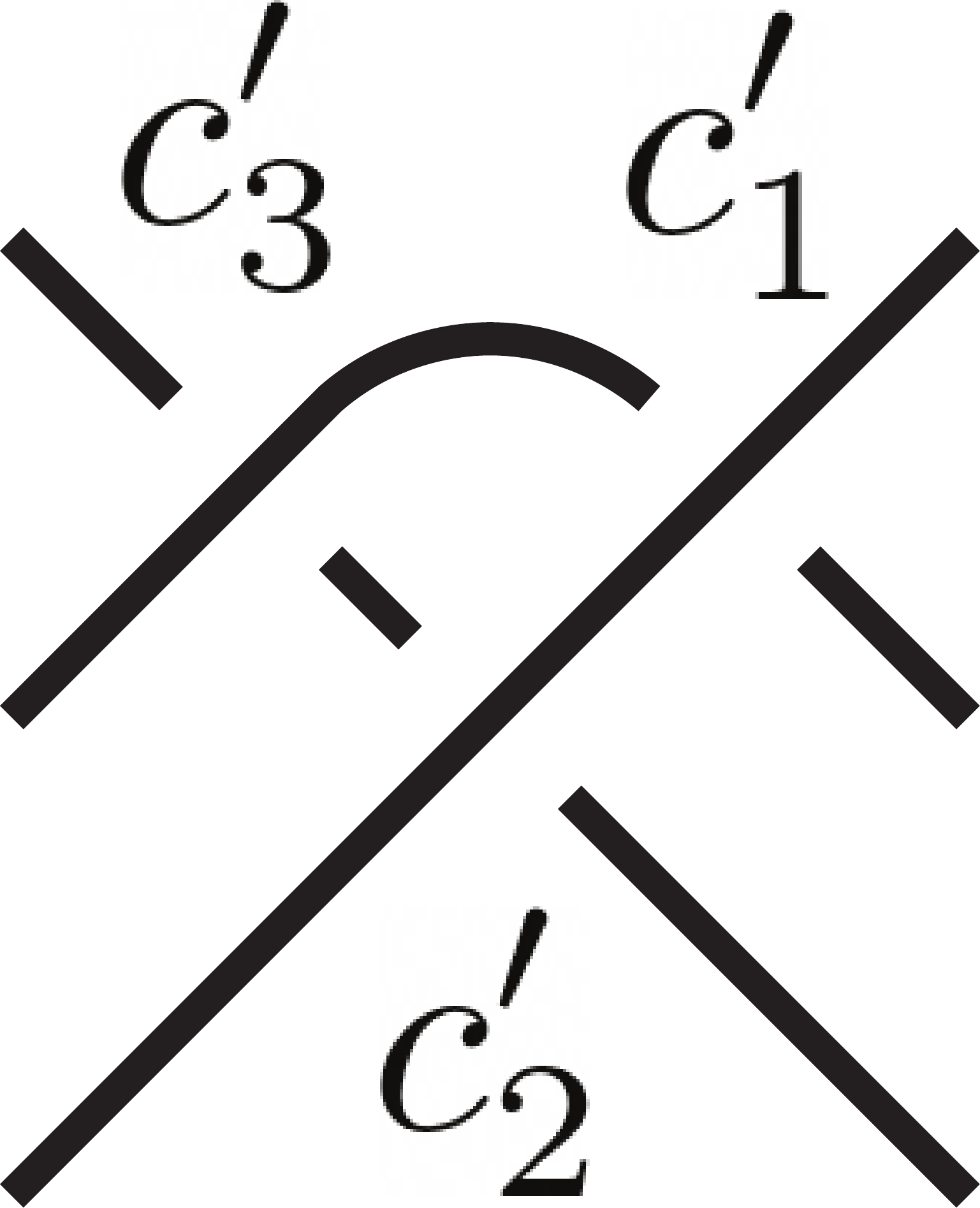}
    \caption{$D'$}
    \label{fig:r3dprime}
    \end{subfigure}
    \hfill
    \begin{subfigure}[b]{0.08\textwidth}
    \centering
    \includegraphics[width=\textwidth]{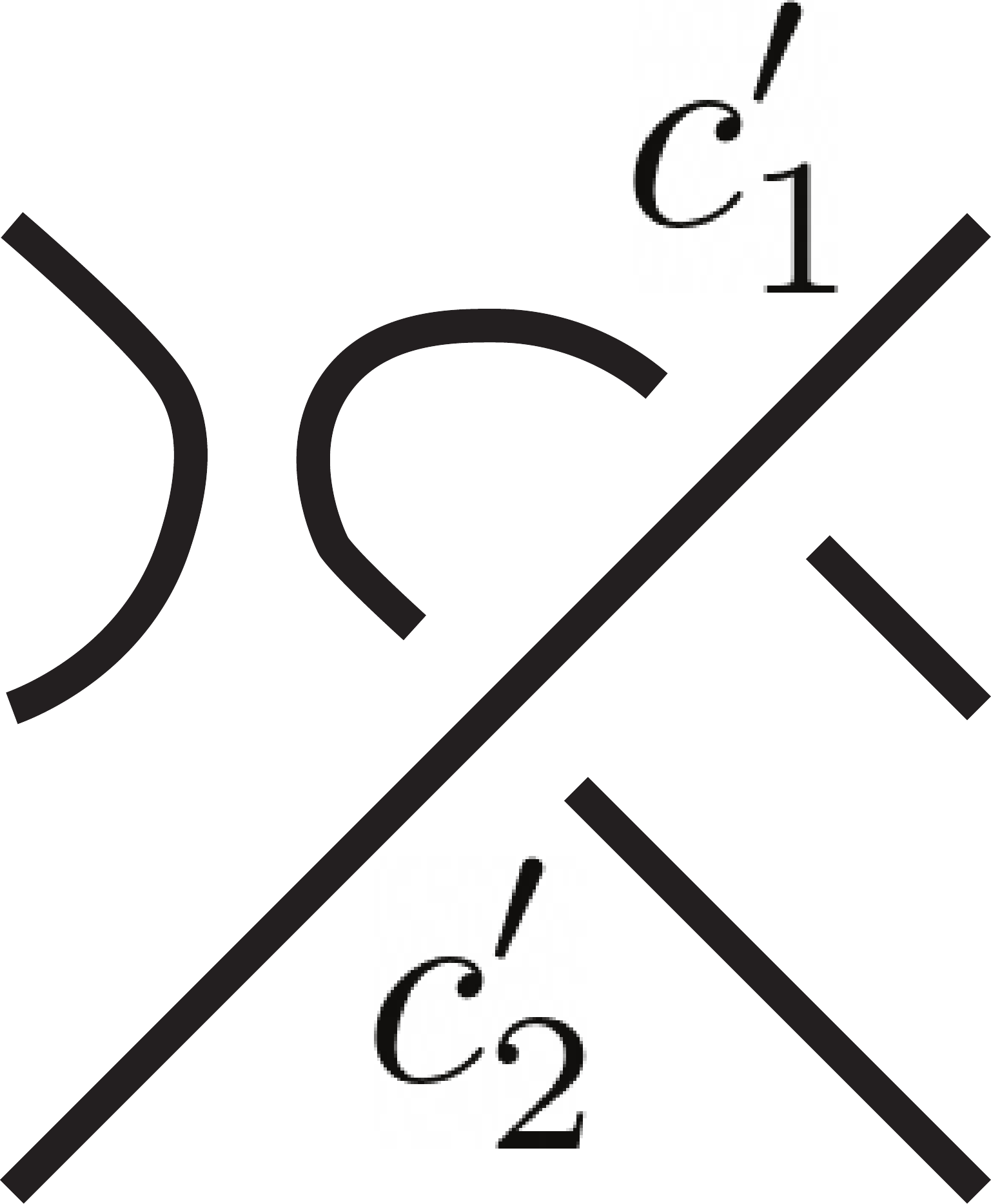}
    \caption{$D'_1$}
    \label{fig:r3dprime1}
    \end{subfigure}
    \hfill
    \begin{subfigure}[b]{0.08\textwidth}
    \centering
    \includegraphics[width=\textwidth]{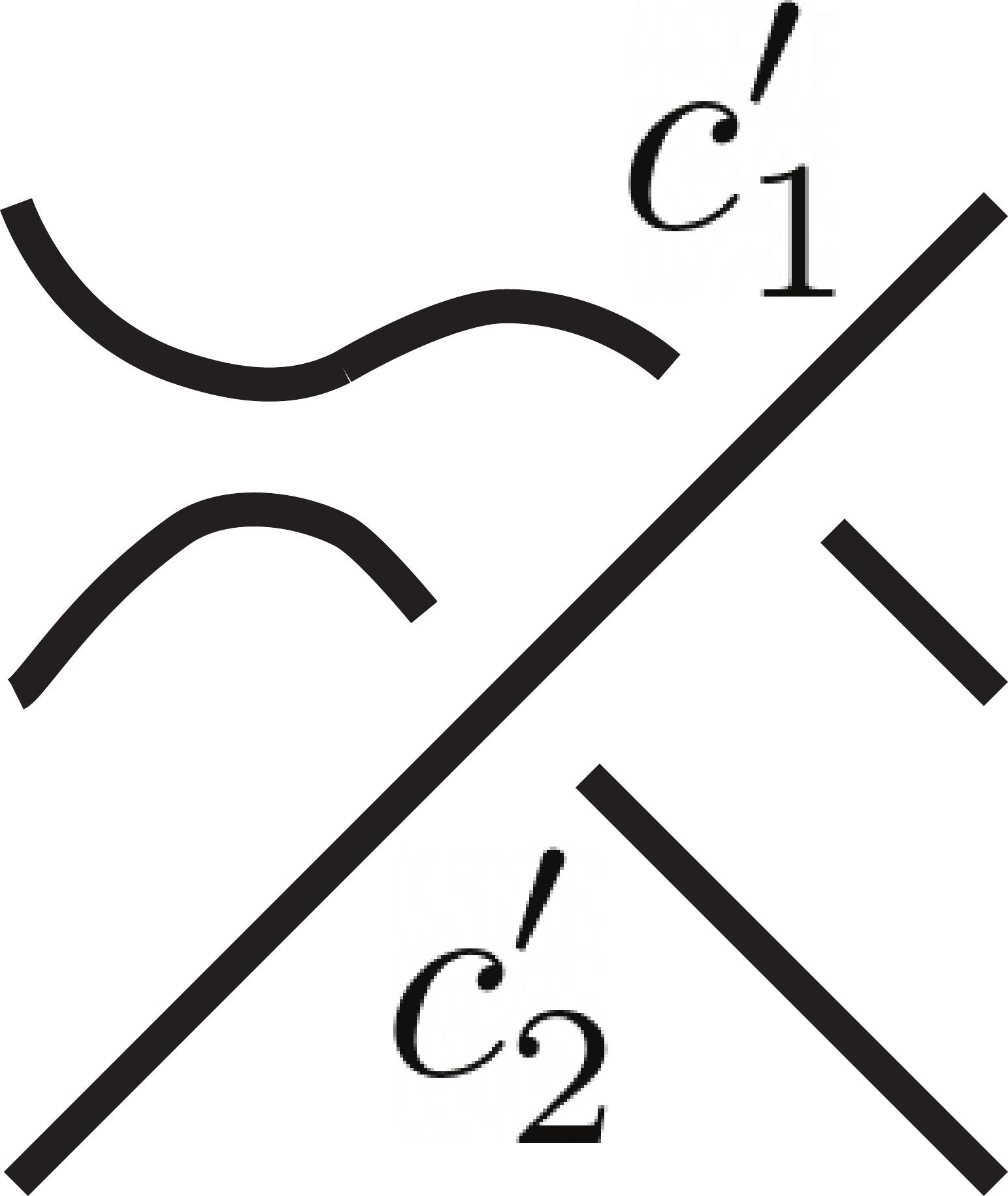}
    \caption{$D'_2$}
    \label{fig:r3dprime2}
    \end{subfigure}
    \caption{Resolution of the crossings $c_3$ and $c'_3$.}
    \label{fig:r3figures}
\end{figure}
We would like to show that $C(D)\overset{c.h.e.}{\sim}C(D')$; see Figure ~\ref{fig:r3figures}, where it is assumed that $c_1<c_2<c_3<c_4<c_5<\cdots$ and $c'_1<c'_2<c'_3<c_4<c_5<\cdots$. When $c_3$ is a positive (resp. negative) crossing, we have
\begin{equation}
    C(D)\cong C(D_1) \oplus C(D_2)[\![1]\!] \text{ (resp. } C(D)\cong C(D_1)[\![-1]\!] \oplus C(D_2)),
\end{equation}
and similarly
\begin{equation}
    C(D')\cong C(D'_1) \oplus C(D'_2)[\![1]\!] \text{ (resp. } C(D')\cong C(D'_1)[\![-1]\!] \oplus C(D'_2)),
\end{equation}
as triply graded vector spaces (not as chain complexes). We will consider only the positive crossing case -- the argument is identical for the negative crossing. Since a Reidemeister II move can be applied to diagram $D_1$, one can use the symbolic notation of Remark ~\ref{rem:r2summary} to list all generators of the subspace $C(D_1)$ in $C(D)$ together with the type 1 arrows in $C(D)$, without mentioning which end points connect to the head/leg of the knotoid or to each other; see Figure ~\ref{fig:r3forD}.
\begin{figure}
    \centering
\begin{tikzcd}[row sep=.5ex,column sep=4ex]
\raisebox{-.4cm}{\includegraphics[scale=0.07]{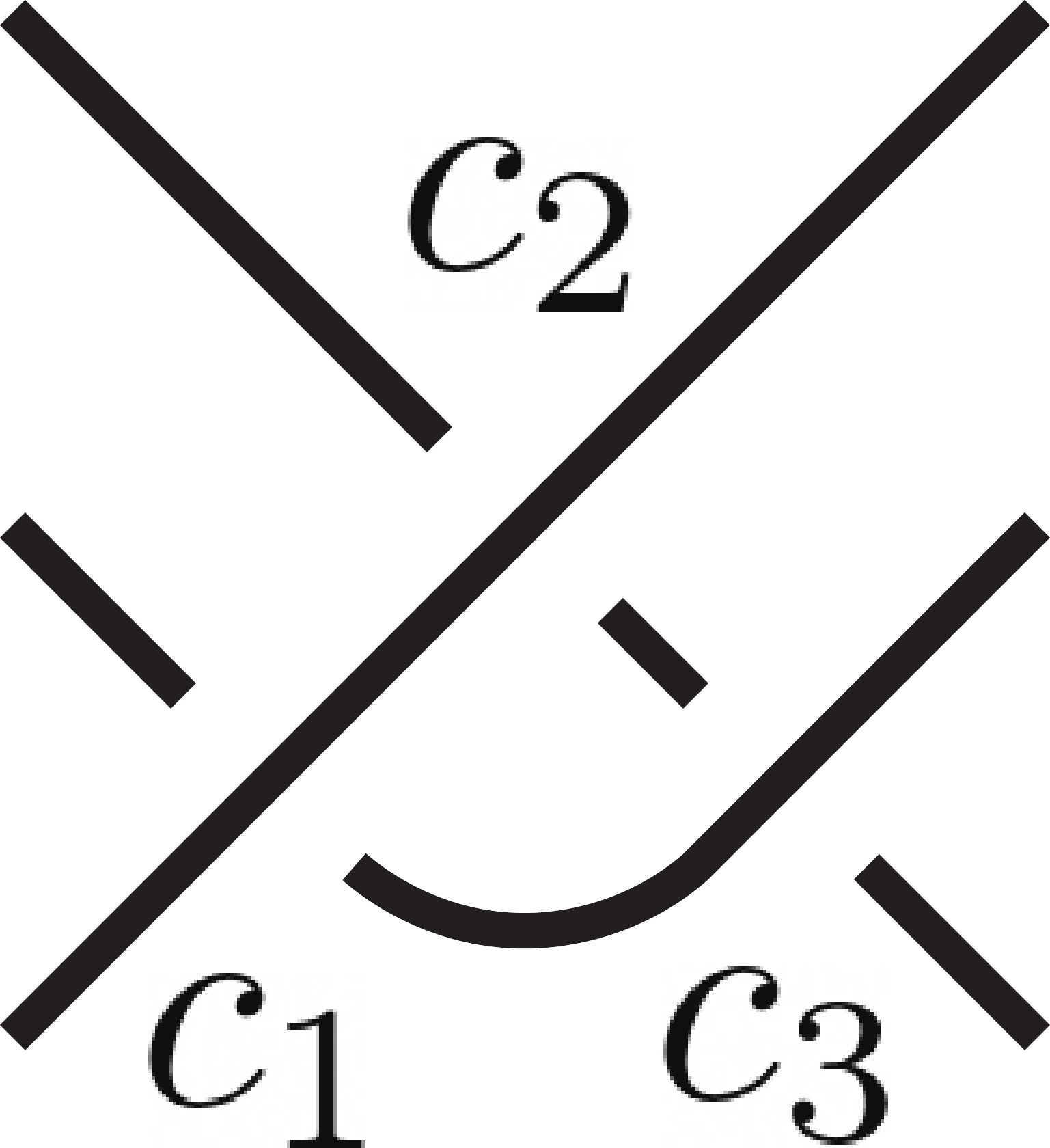}} & \scalebox{1.4}{$\Dioo$} \otimes [c_1x]& \scalebox{1.4}{$\Diio$} \otimes [c_1c_2x] \\
& \DiooWithLetters{p}{q}{r}+\DoioWithLetters{p:q}{q:p}{1}{\widetilde{r}} \arrow[rdd,"-", near start] \arrow[rdddd,"-"]& d\left(\,\DoioWithLetters{}{}{1}{}\right) \\
\scalebox{1.4}{$\Dooo$} \otimes [x] & \scalebox{1.4}{$\Doio$} \otimes [c_2x] & \scalebox{1.4}{$\Dioi$} \otimes [c_1c_3x] & \scalebox{1.4}{$\Diii$} \otimes [c_1c_2c_3x] \\
\Dooo \arrow[r, red, dashed] & d\left(\,\Dooo\,\right)\,\,,\,\,\,\,\DoioWithLetters{}{}{1}{}\,\,\,\,\,\,\,\,\, \arrow[ruu, crossing over, red, dashed] & \Dioi \arrow[r,"-"] & 
\Diii \\
& \scalebox{1.4}{$\Dooi$} \otimes [c_3x] & \scalebox{1.4}{$\Doii$} \otimes [c_2c_3x] & \\
& \Dooi \arrow[r] \arrow[ruu] & \Doii \arrow[ruu] \\
\end{tikzcd}   
    \caption{The chain complex $\widetilde{C}(D)$ after the cancellation of the dashed arrows.}
    \label{fig:r3forD}
\end{figure}
After this change of basis, there are no other type 1 arrows in $C(D)$ at the generators $\Dooo$ and \DoioWithLetters{}{}{\scriptsize{1}}, then the same trick of monotonically cancelling these arrows from right to left is used within sub-cubes of $\Dooo$ and \DoioWithLetters{}{}{\scriptsize{1}}, so that forming zigzags with type 2 arrows is avoided. The resulting chain complex, denoted as $\widetilde{C}(D)$, is chain equivalent to $C(D)$ and isomorphic to \[\,\left\langle\,\, \DiooWithLetters{p}{q}{r}+\,\DoioWithLetters{p:q}{q:p}{1}{\widetilde{r}}\,\,\right\rangle \oplus C(D_2)[\![1]\!]\vspace{.2cm}\]
as a triply graded vector space. Similarly, we have the diagram in Figure ~\ref{fig:r3forDp} for $C(D')$.
\begin{figure}
    \centering
\begin{tikzcd}[row sep=.5ex,column sep=4ex]
\raisebox{-.4cm}{\includegraphics[scale=0.07]{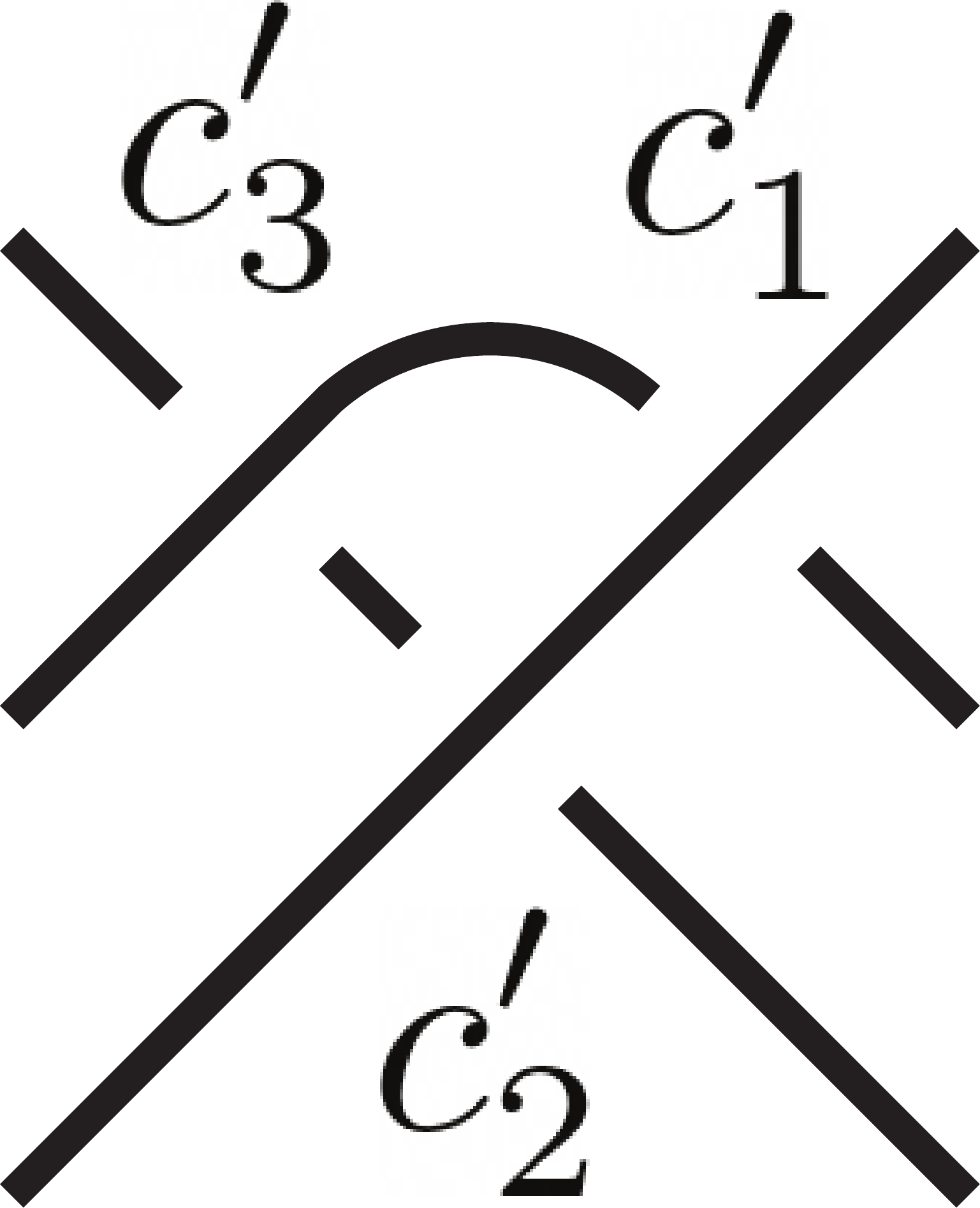}}& \scalebox{1.4}{$\Dpioo$} \otimes [c'_1x]& \scalebox{1.4}{$\Dpiio$} \otimes [c'_1c'_2x] \\
& \DpiooWithLetters{p}{q}{r}+\DpoioWithLetters{q:r}{r:q}{1}{\widetilde{p}} \arrow[rdd,"-", near start] \arrow[rdddd,"-"] & d\left(\,\DpoioWithLetters{}{}{1}{}\right) \\
\scalebox{1.4}{$\Dpooo$} \otimes [x] & \scalebox{1.4}{$\Dpoio$} \otimes [c'_2x] & \scalebox{1.4}{$\Dpioi$} \otimes [c'_1c'_3x] & \scalebox{1.4}{$\Dpiii$} \otimes [c'_1c'_2c'_3x] \\
\Dpooo \arrow[r,red, dashed] & d\left(\,\Dpooo\,\right)\,\,,\,\,\,\,\DpoioWithLetters{}{}{1}{}\,\,\,\,\,\,\,\,\, \arrow[ruu,crossing over, red, dashed] & \Dpioi \arrow[r,"-"] & 
\Dpiii \\
& \scalebox{1.4}{$\Dpooi$} \otimes [c'_3x] & \scalebox{1.4}{$\Dpoii$} \otimes [c'_2c'_3x] & \\
& \Dpooi \arrow[r] \arrow[ruu] & \Dpoii \arrow[ruu] \\
\end{tikzcd}
    \caption{The chain complex $\widetilde{C}(D')$ after the cancellation of the dashed arrows.}
    \label{fig:r3forDp}
\end{figure}
Cancelling the arrows the same way, we obtain a chain complex, denoted as $\widetilde{C}(D')$, that is chain equivalent to $C(D')$ and isomorphic to \[\left\langle\,\, \DpiooWithLetters{}{r}{s}+\DpoioWithLetters{r:s}{s:r}{1}{}\,\,\,\,\,\,\,\right\rangle \oplus C(D'_2)[\![1]\!]\vspace{.2cm}\] as a triply graded vector space. 

Now, it is clear that $C(D_2)$ and $C(D'_2)$ are isomorphic as chain complexes -- the only difference between the diagrams $D_2$ and $D'_2$ is that the first two crossings are switched in the ordering. Then the chain equivalence from $\widetilde{C}(D)$ to $\widetilde{C}(D')$ is defined by (1) the invertible map $\Phi_{(12)}$ for the adjacent transposition of 1 and 2 (see Section ~\ref{sec:indpfromorder}) on the generators of $C(D_1)$, and (2) the map that takes the single remaining generator of $\widetilde{C}(D)$ to that of $\widetilde{C}(D')$. More explicitly,
\begin{align}
    \Dooi\otimes [c_3x] &\xmapsto{\hspace{0.3in}} \Dpooi \otimes [c'_3x] \label{equ:cheR3map1} \\
    \Dioi\otimes [c_1c_3x] &\xmapsto{\hspace{0.3in}} \Dpoii \otimes [c'_2c'_3x] \\
    \Doii\otimes [c_2c_3x] &\xmapsto{\hspace{0.3in}} \Dpioi \otimes [c'_1c'_3x] \\
    \Diii\otimes [c_1c_2c_3x] &\xmapsto{\hspace{0.075in}-1\hspace{0.075in}} \Dpiii \otimes [c'_1c'_2c'_3x] \label{equ:cheR3map4} \\
    \DiooWithLetters{p}{q}{r}\otimes [c_1x]+\DoioWithLetters{p:q}{q:p}{1}{\widetilde{r}}\,\,\otimes [c_2x] & \xmapsto{\hspace{0.3in}} \,\, \DpiooWithLetters{p}{q}{r}\otimes [c'_1x]+\,\,\DpoioWithLetters{q:r}{r:q}{1}{\widetilde{p}}\,\,\,\,\,\,\,\otimes [c'_2x] \label{equ:cheR3map5}
\end{align}\\
where $\widetilde{r}$ is one of $r$, $p:q$, or $q:p$ depending on whether the arcs on the local picture are connected outside, and $\widetilde{p}$ is one of $p$, $q:r$, or $r:q$ similarly. Missing labels on the first four maps~\eqref{equ:cheR3map1} through ~\eqref{equ:cheR3map4} mean that any labeling is allowed as long as the labels for the same components on both sides are the same. Note that this piecewise definition is a consistent chain map as the following diagram commutes:
\begin{center}
    \begin{tikzcd}
        \DiooWithLetters{p}{q}{r} \otimes [c_1x] + \DoioWithLetters{p:q}{q:p}{1}{\widetilde{r}}\,\otimes [c_2x] \arrow[r] \arrow{d}{-}[swap]{d}
        & \DpiooWithLetters{p}{q}{r} \otimes [c'_1x] +\, \DpoioWithLetters{q:r}{r:q}{1}{\widetilde{p}}\,\,\,\,\,\,\, \otimes [c'_2x] \arrow{d}{-}[swap]{d'} \\
        \DioiWithLetters{\widetilde{p}}{q:r}{r:q} \,\,\,\,\,\,\, \otimes [c_1c_3x] + \DoiiWithLetters{p:q}{q:p}{\widetilde{r}}\, \otimes [c_2c_3x] \arrow[r]
        & \DpioiWithLetters{p:q}{q:p}{\widetilde{r}} \otimes [c'_1c'_3x] +\,\, \DpoiiWithLetters{\widetilde{p}}{q:r}{r:q} \,\,\,\,\,\,\, \otimes [c'_2c'_3x]\\
    \end{tikzcd}\vspace{-.2cm}
\end{center}

We finish by showing that the chain equivalence above preserve the $(q,u)$-grading. For the $q$-grading, this follows from the observations (1) that the terms on both sides of the equivalence maps have the same $q(v)=\deg(v)+i(v)+n_+-n_-$ value, and (2) that the chain equivalences coming from the application of zigzag lemma preserve the $q$-grading. For the $u$-grading, it is also true that chain equivalences of the zigzag lemma preserve the $u$-grading as explained in Reidemeister move I case, so one only needs to check if the terms on either sides of the equivalence maps above have the same $u$-gradings. Using the canonical shortcut $\alpha$, this is immediate for the first four maps~\eqref{equ:cheR3map1} through ~\eqref{equ:cheR3map4} since (1) the states on either sides are isotopic in $S^2$ which makes the value of $k_s\cdot \alpha$ the same for them, and (2) $K\cdot \alpha$ is the same for diagrams $D$ and $D'$; see Figure ~\ref{fig:r3figures}. For the map ~\eqref{equ:cheR3map5}, there are three different positions, up to symmetry, that the canonical shortcut can pass through the local disk. Using the notation \lq\lq100" (resp. \lq\lq010") for resolutions with $s(c_1)=1=s(c'_1)$, $s(c_{2,3})=0=s(c'_{2,3})$ (resp. $s(c_2)=1=s(c'_2)$, $s(c_{1,3})=0=s(c'_{1,3})$), the three shortcut positions are listed together with their intersections with the four resolutions in Figure ~\ref{fig:u4r31stEx}.
\begin{figure}
    \centering
    \begin{tikzcd}
    & \includegraphics[scale=0.05]{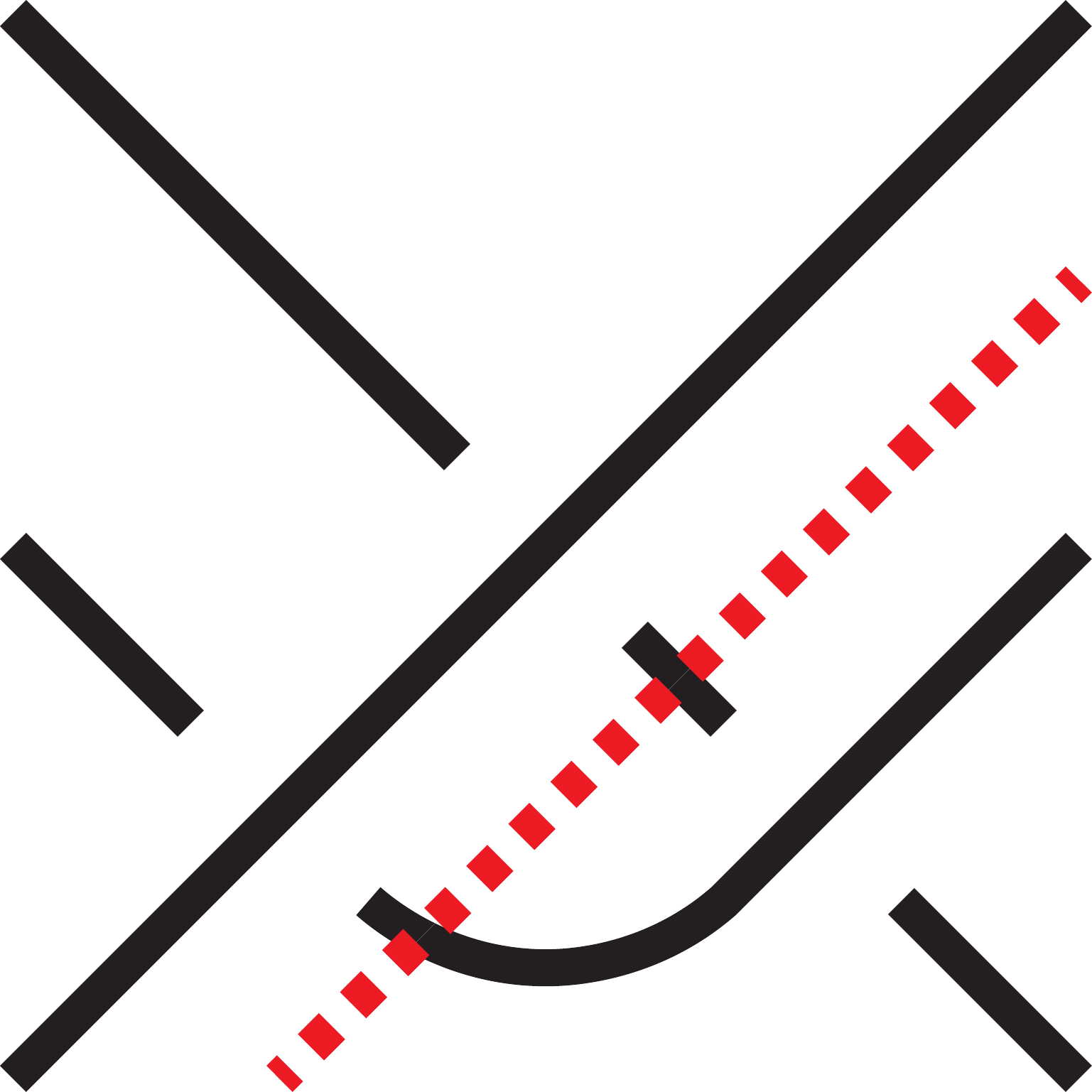} \arrow{ld}[swap]{100} \arrow[rd,"010"]& & & \includegraphics[scale=0.05]{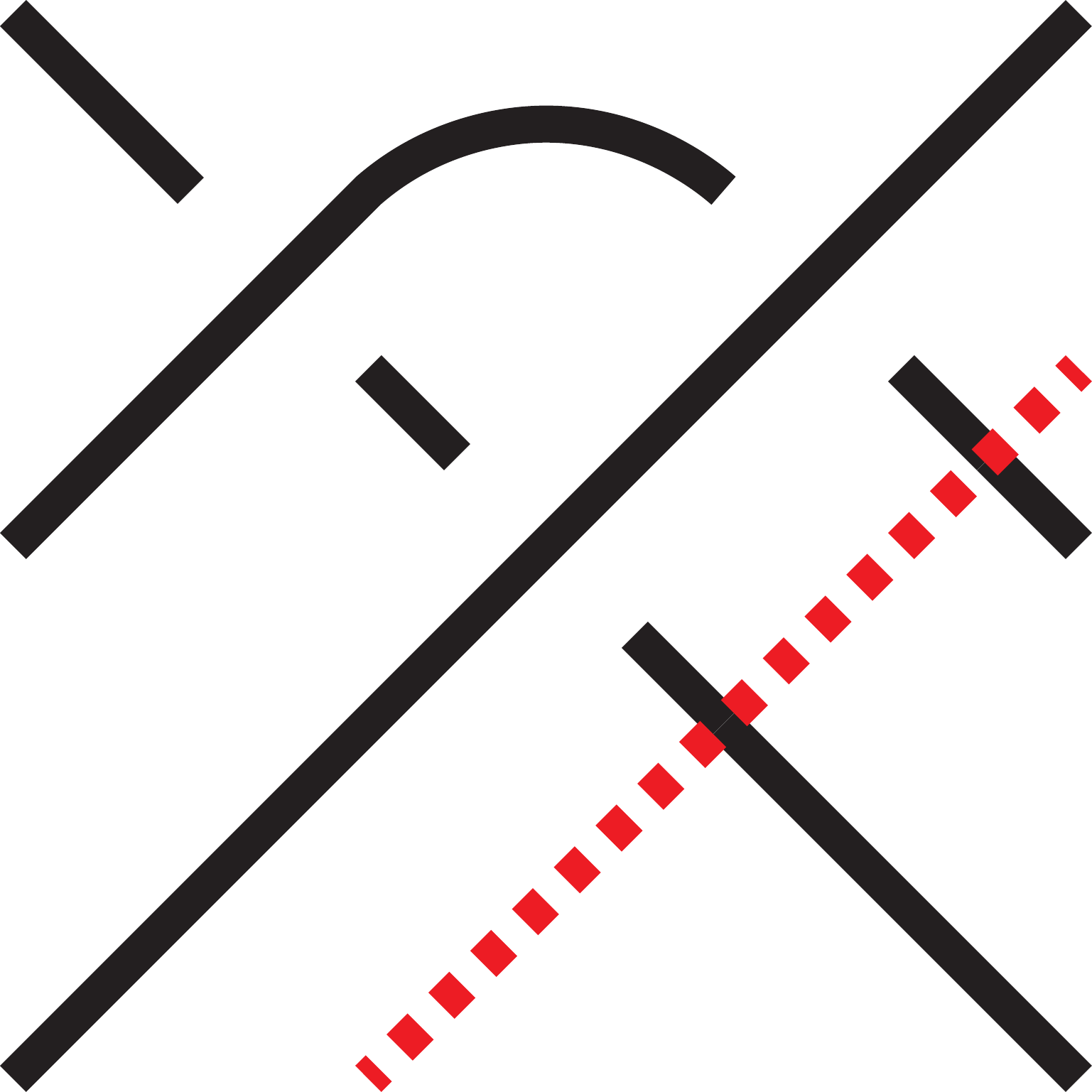} \arrow{ld}[swap]{100} \arrow[rd,"010"]& \\ 
    \includegraphics[scale=0.05]{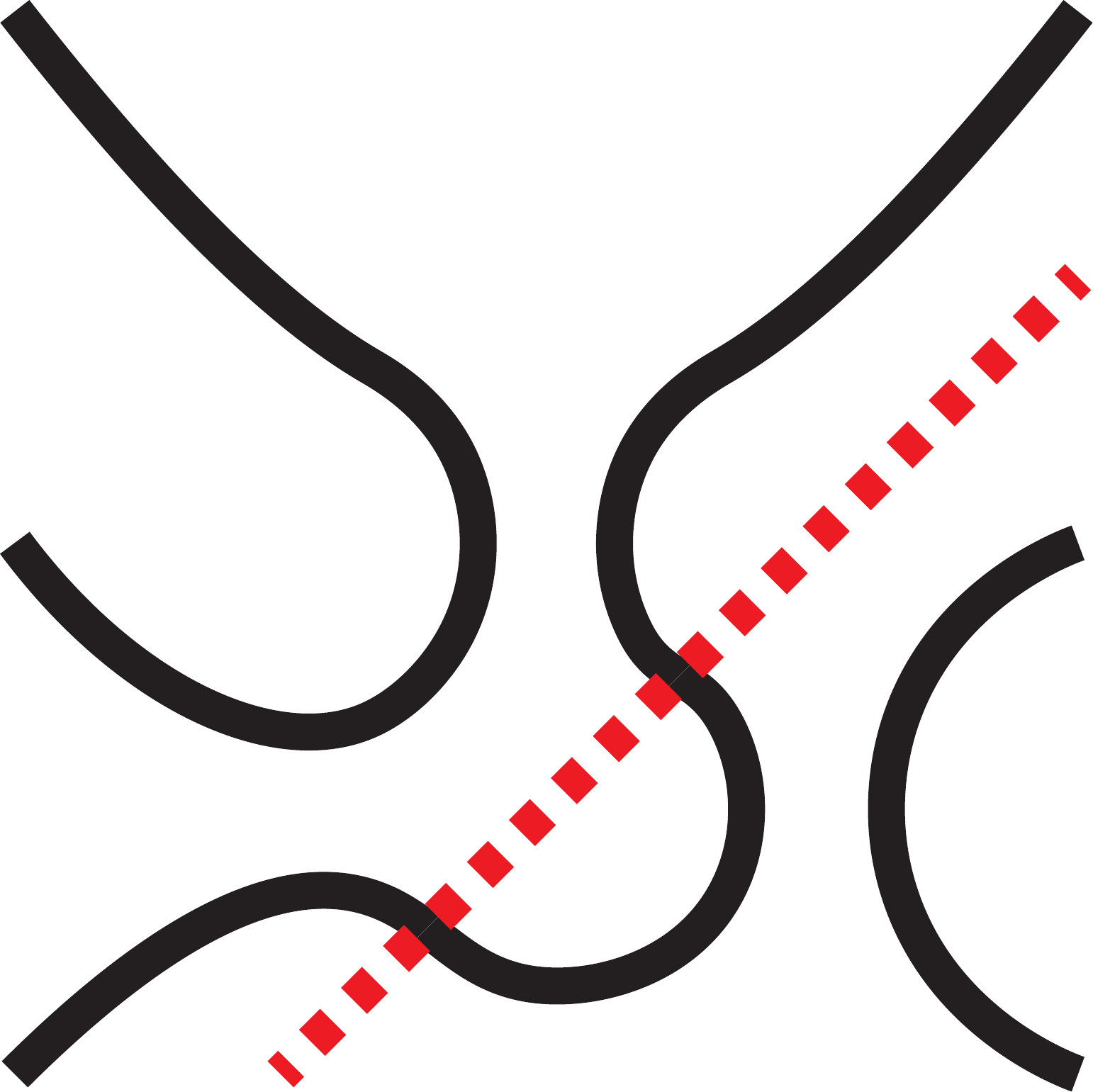} & & \includegraphics[scale=0.05]{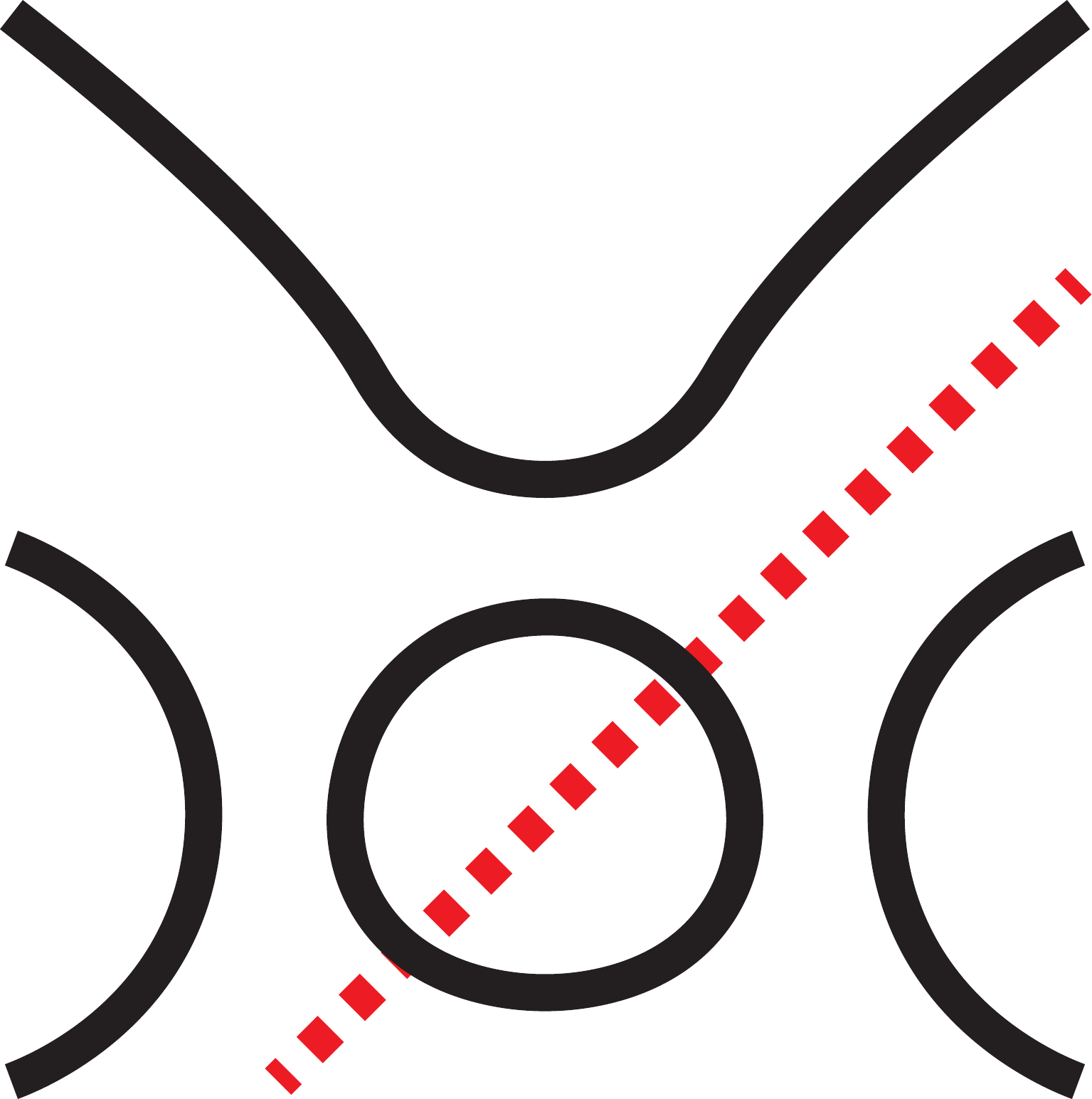} & \includegraphics[scale=0.05]{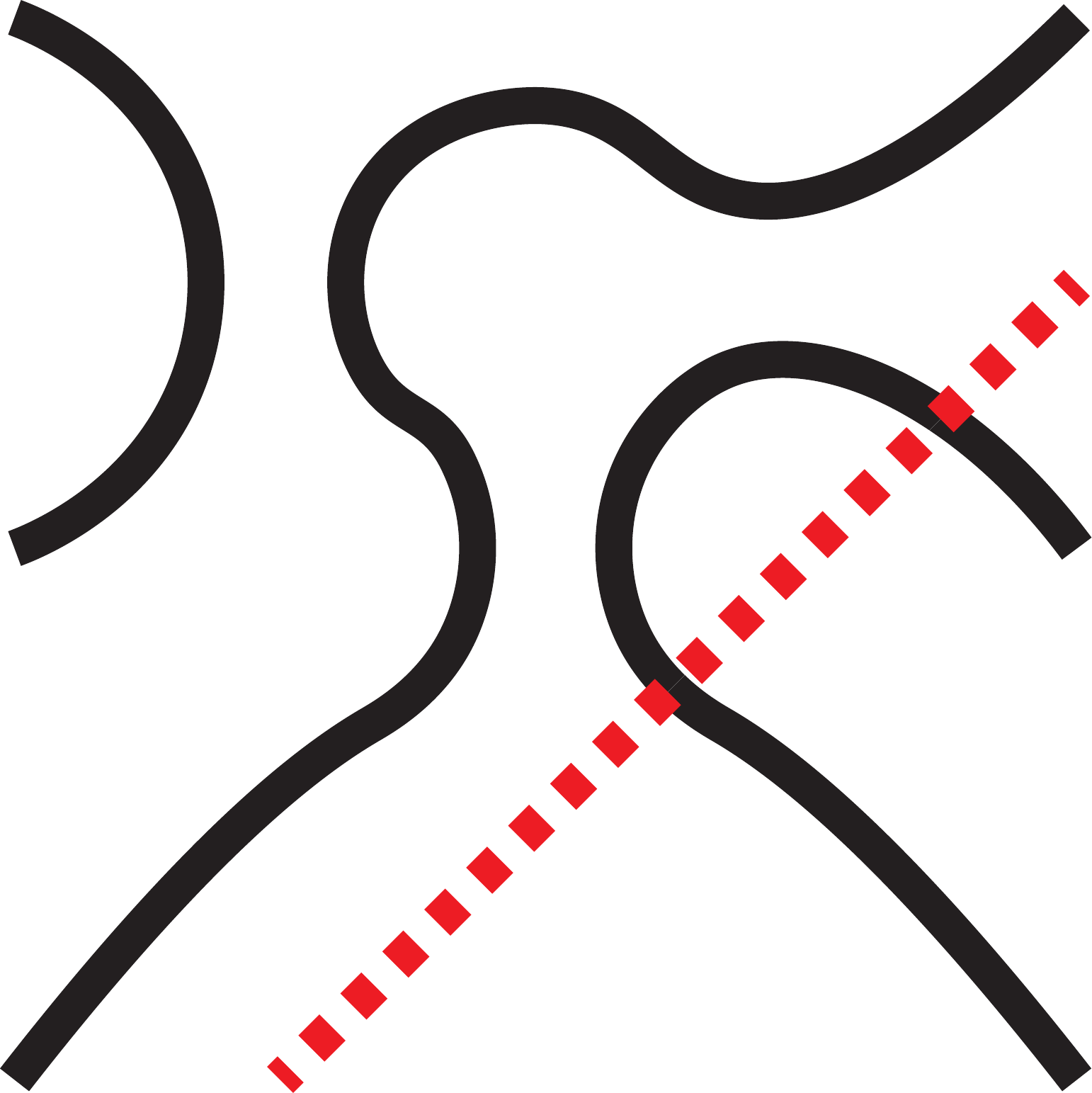} & & \boxed{\includegraphics[scale=0.05]{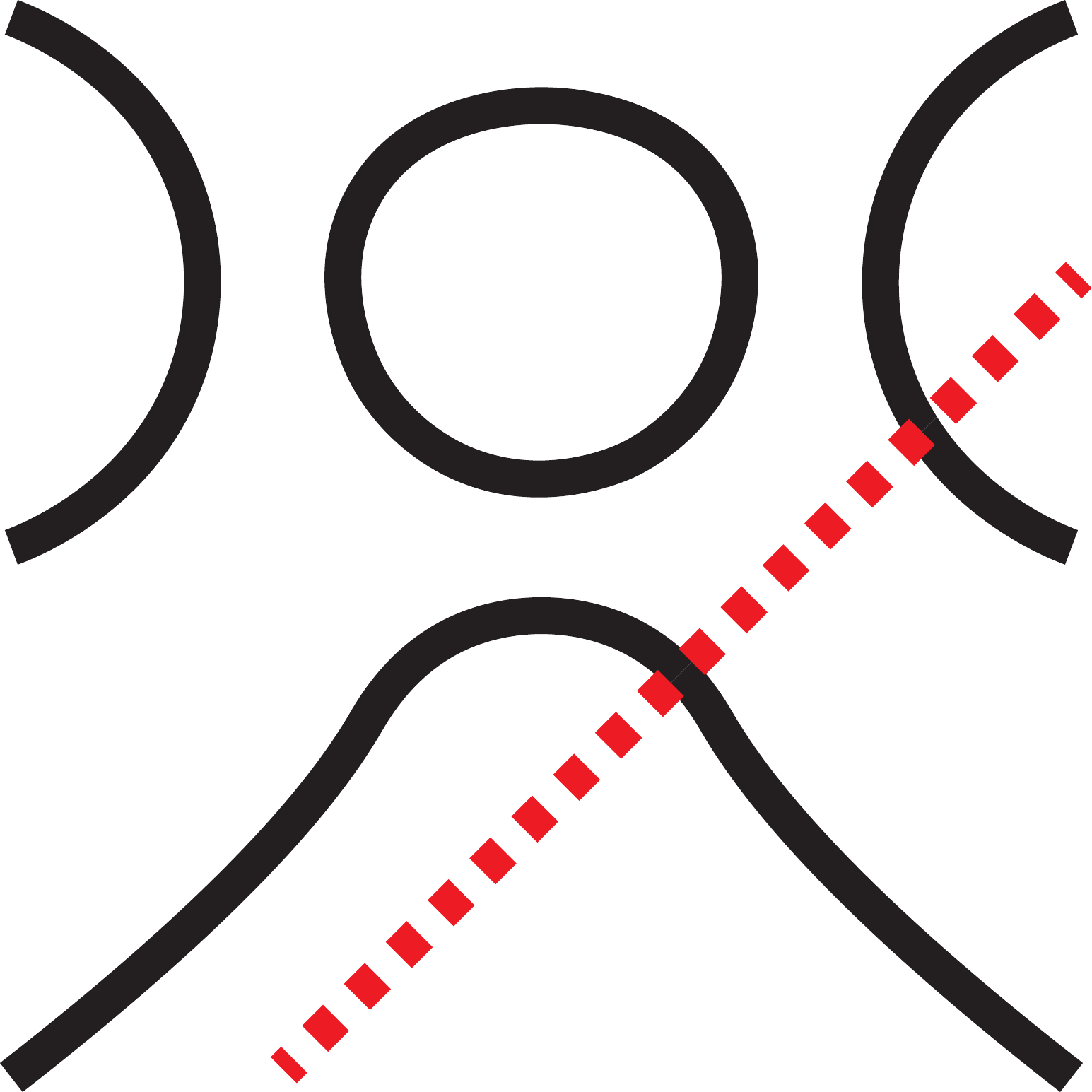}}
    \end{tikzcd}\vspace{0.1in}\hrule \vspace{0.1in}
    \begin{tikzcd}
    & \includegraphics[scale=0.05]{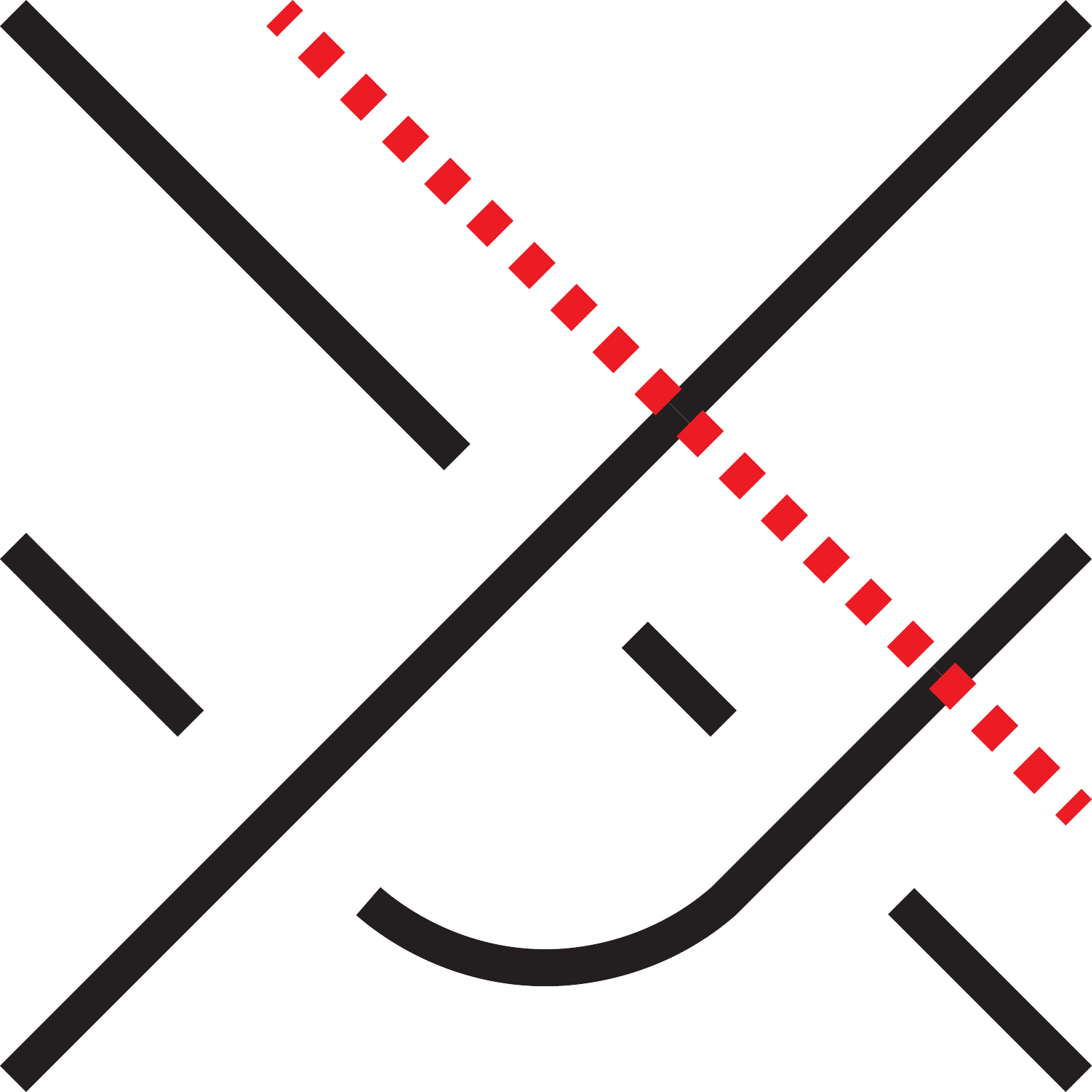} \arrow{ld}[swap]{100} \arrow[rd,"010"]& & & \includegraphics[scale=0.05]{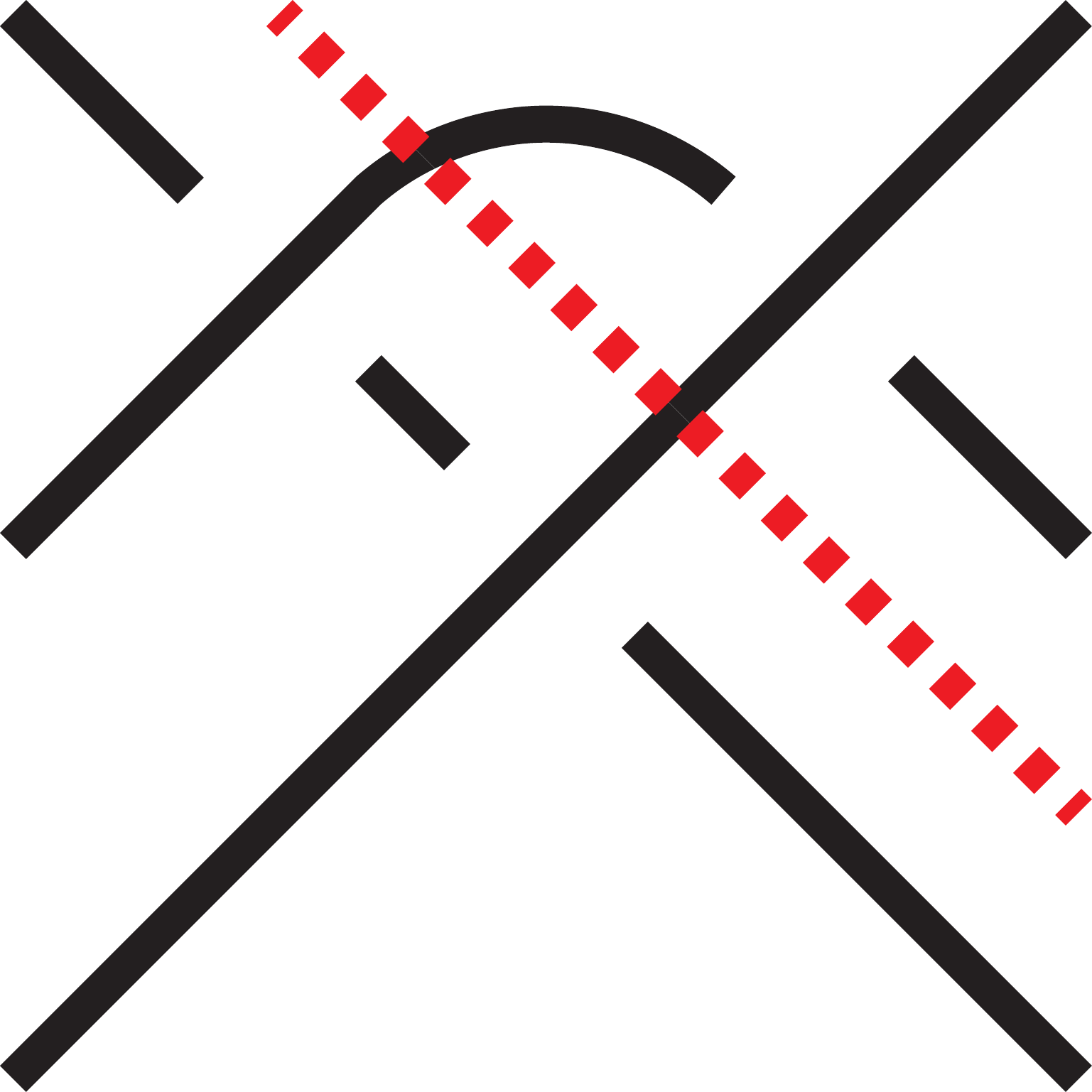} \arrow{ld}[swap]{100} \arrow[rd,"010"]& \\ 
    \includegraphics[scale=0.05]{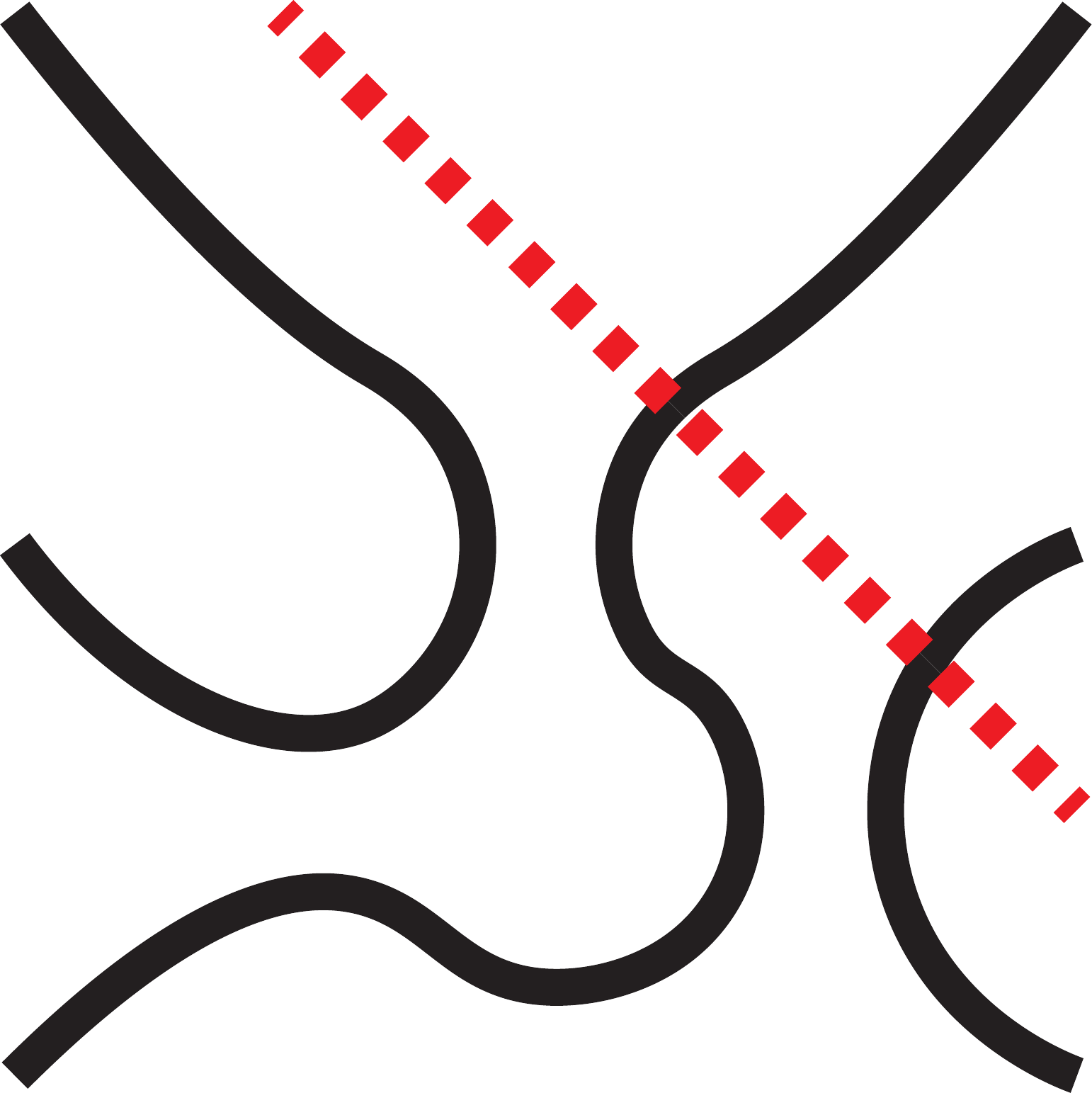} & & \includegraphics[scale=0.05]{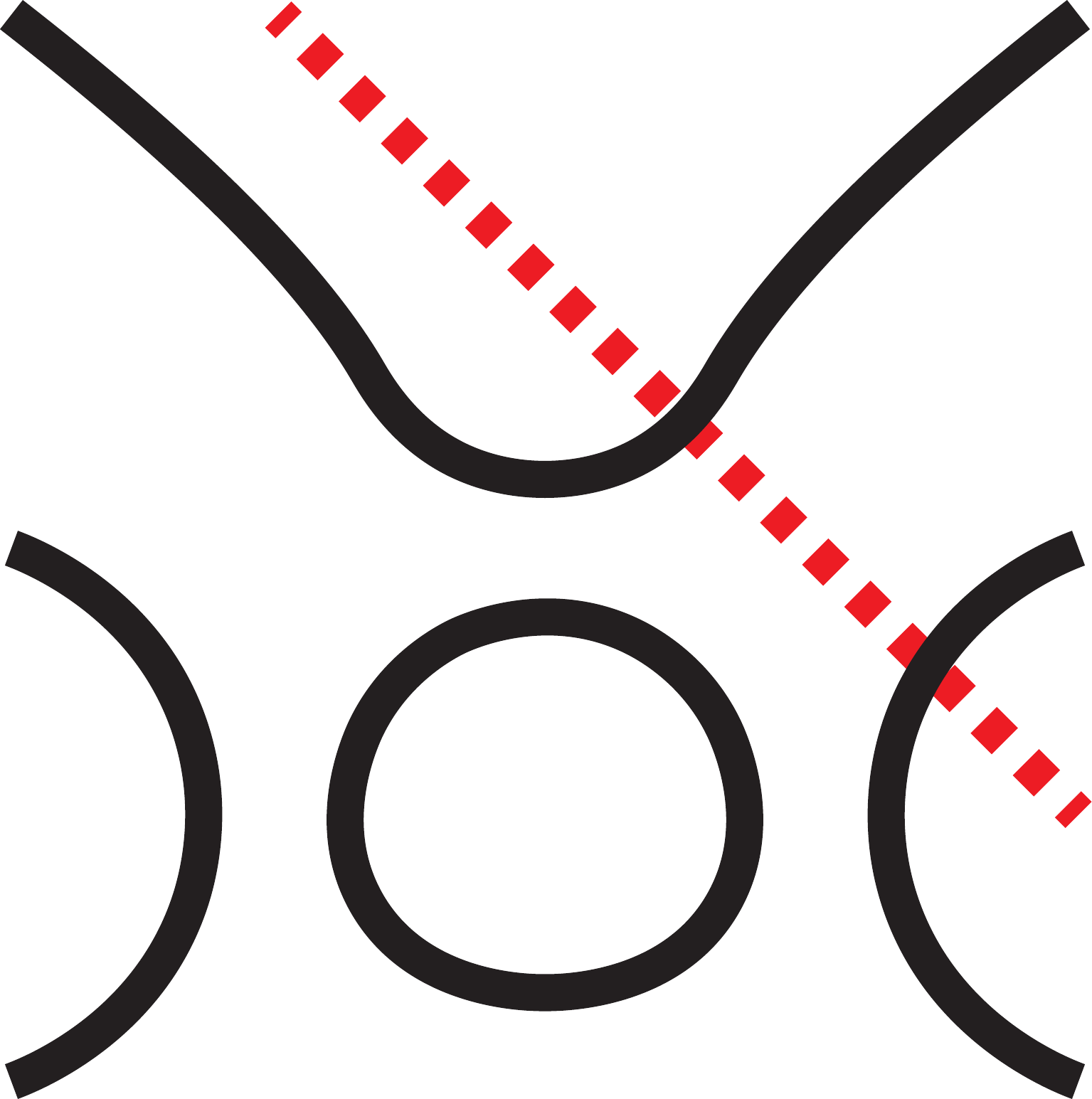} & \includegraphics[scale=0.05]{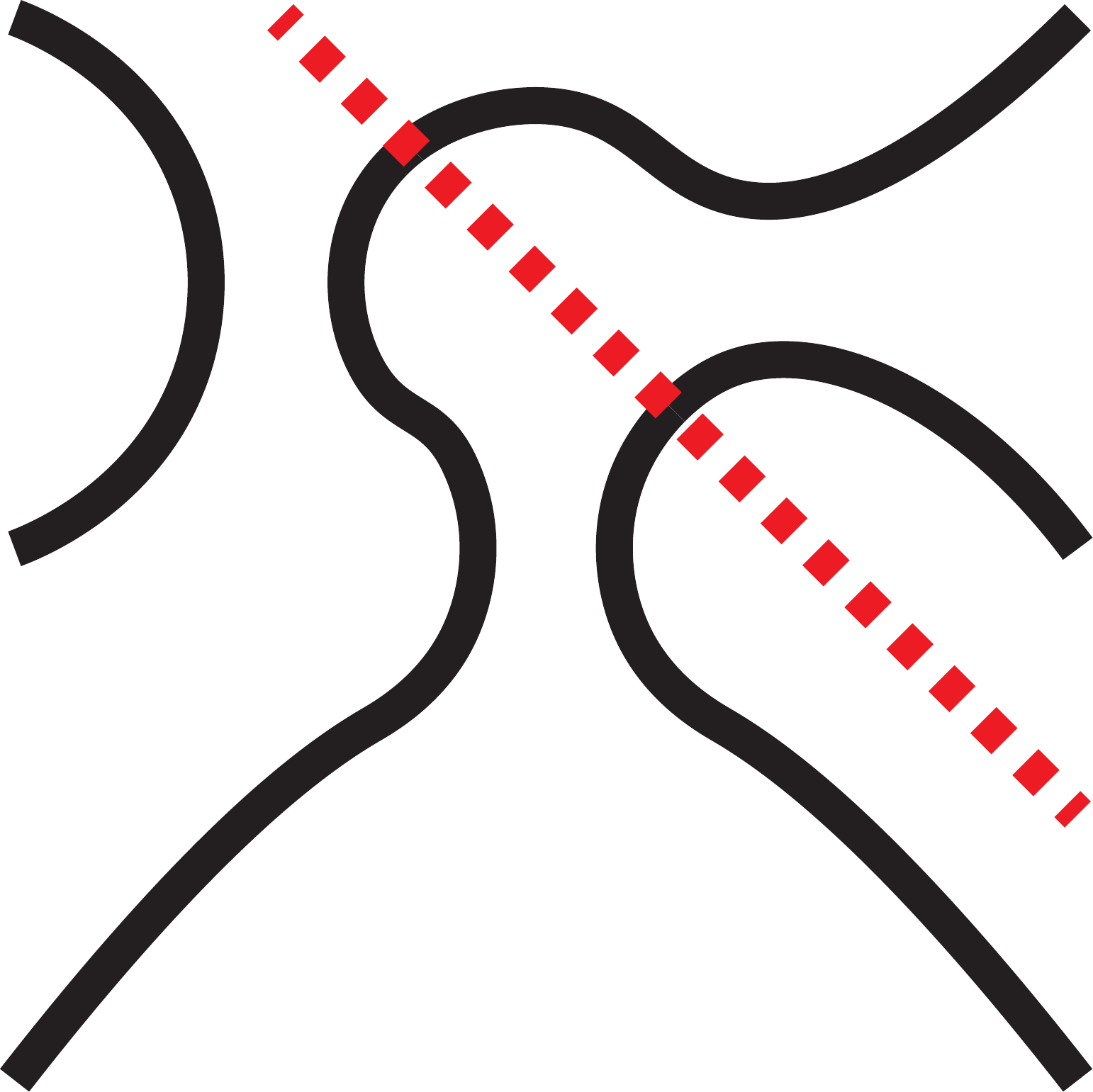} & & \boxed{\includegraphics[scale=0.05]{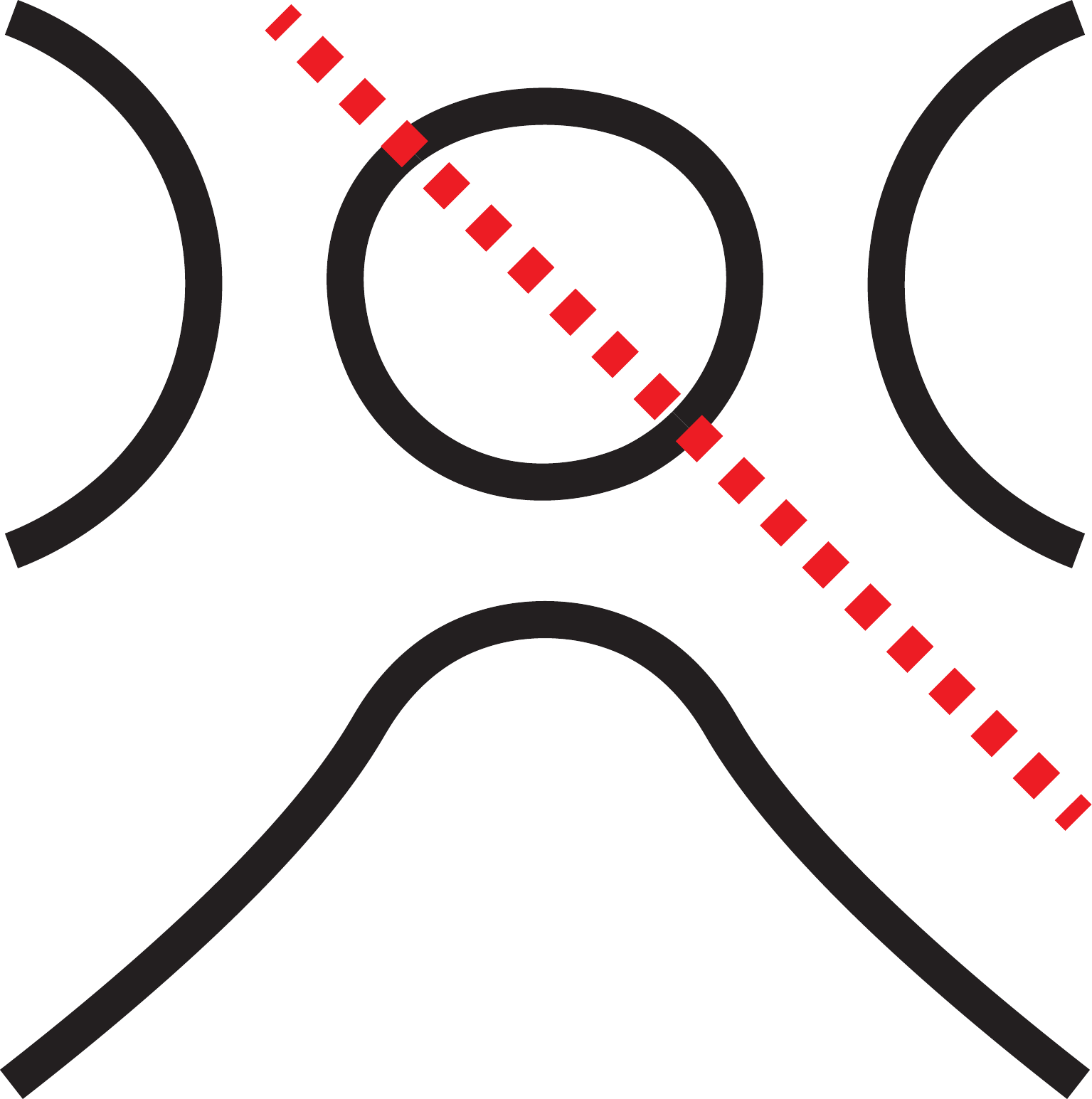}}
    \end{tikzcd}\vspace{0.1in}\hrule \vspace{0.1in}
    \begin{tikzcd}
    & \includegraphics[scale=0.05]{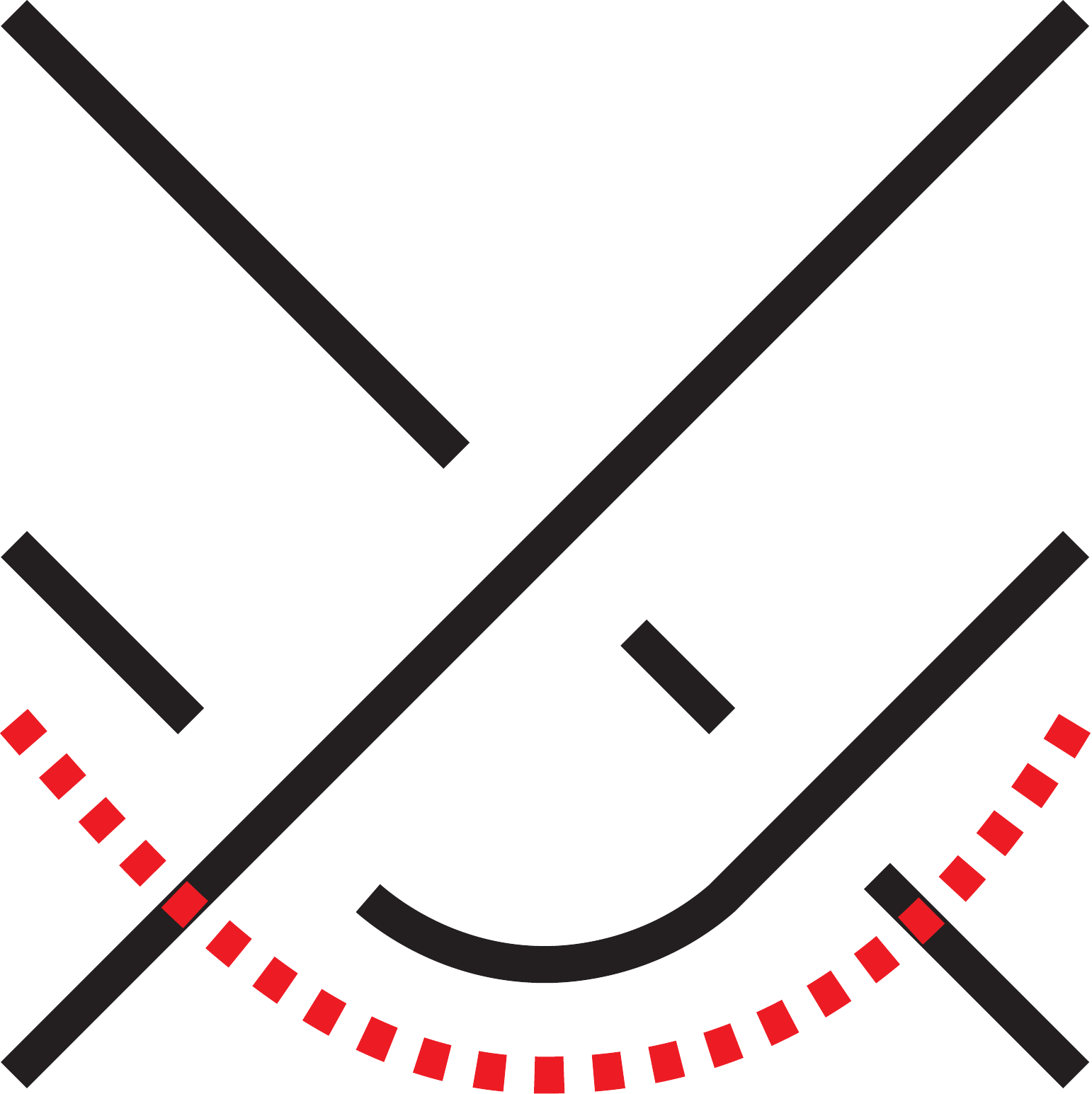} \arrow{ld}[swap]{100} \arrow[rd,"010"]& & & \includegraphics[scale=0.05]{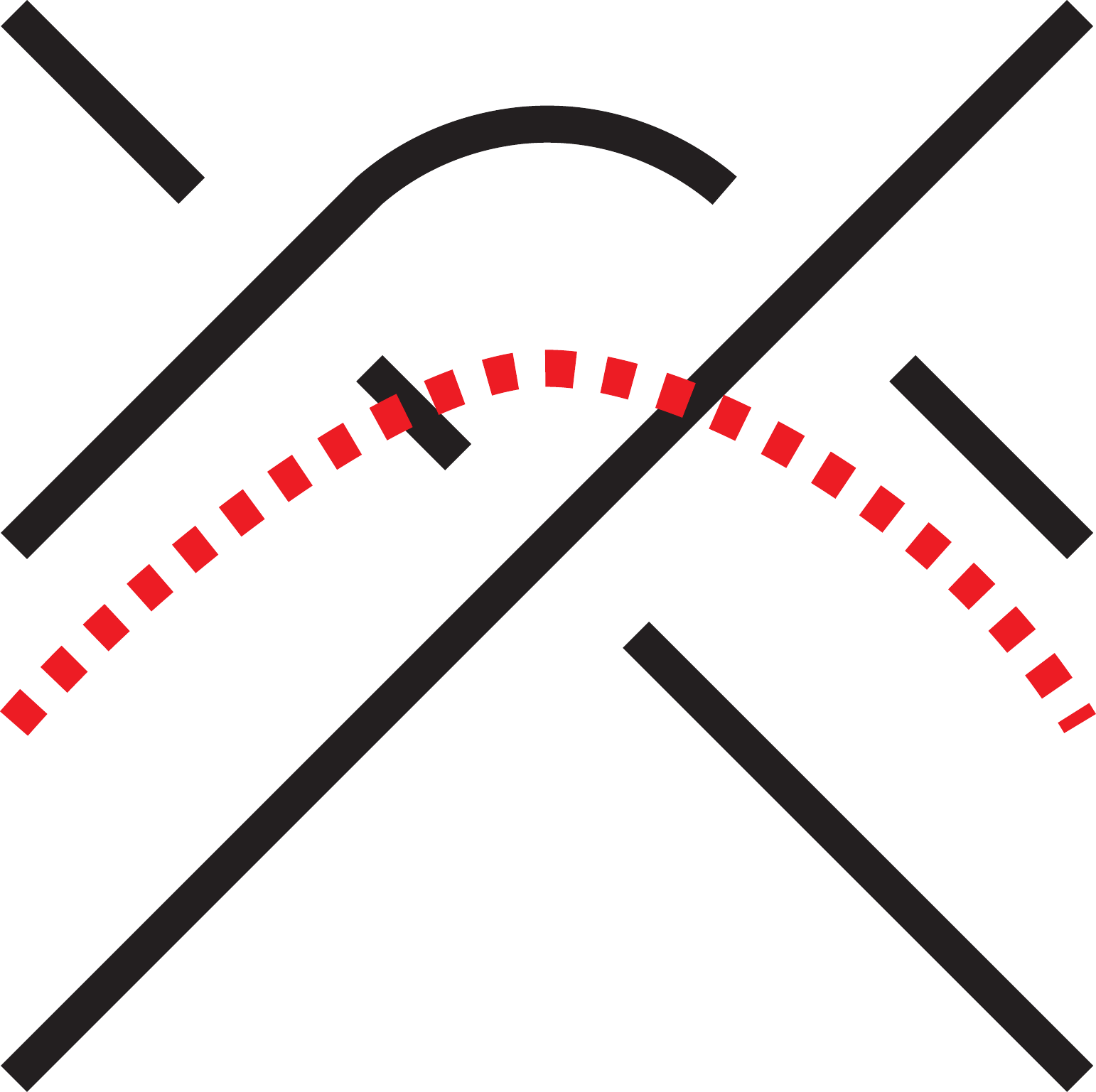} \arrow{ld}[swap]{100} \arrow[rd,"010"]& \\ 
    \includegraphics[scale=0.05]{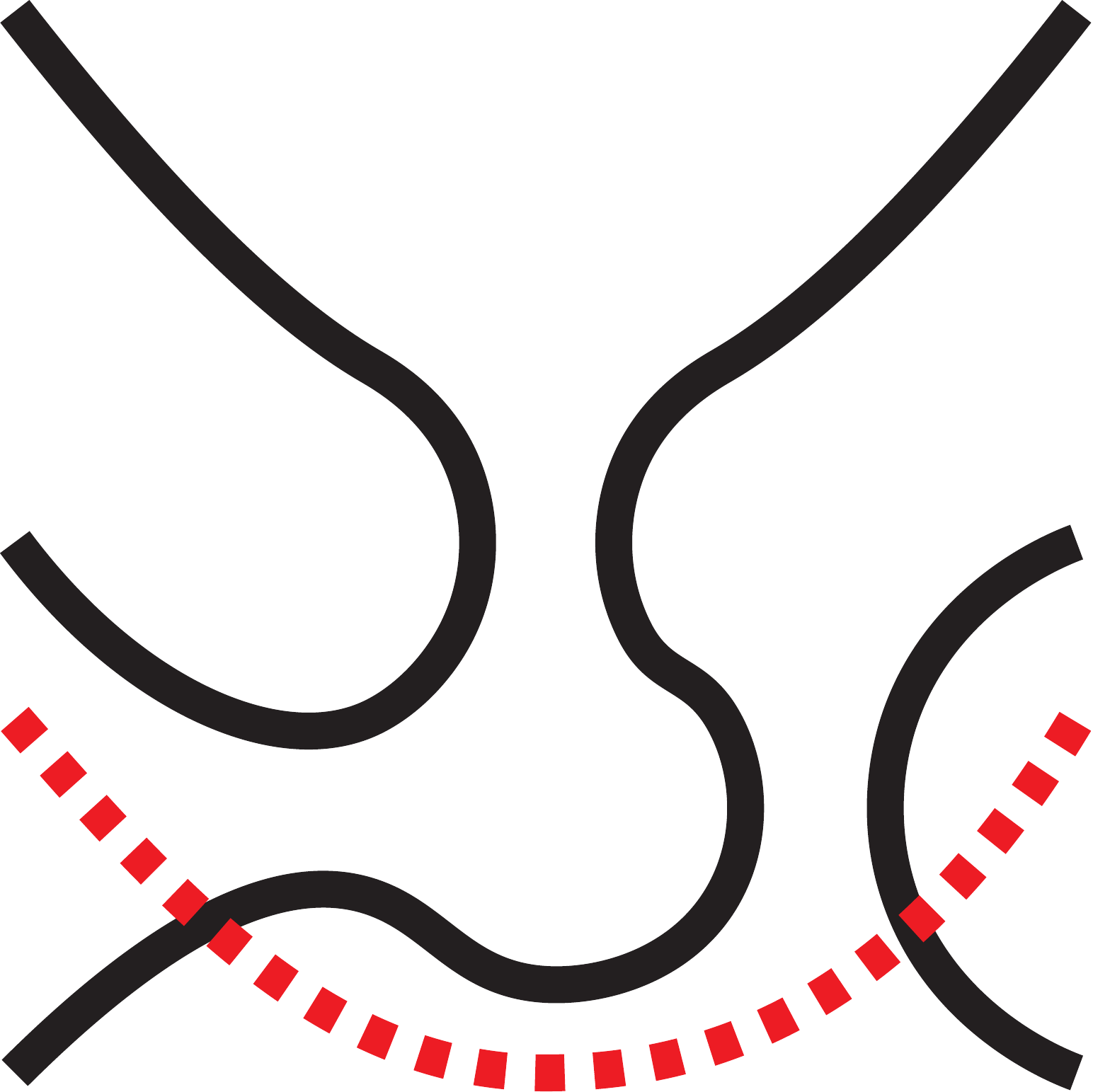} & & \includegraphics[scale=0.05]{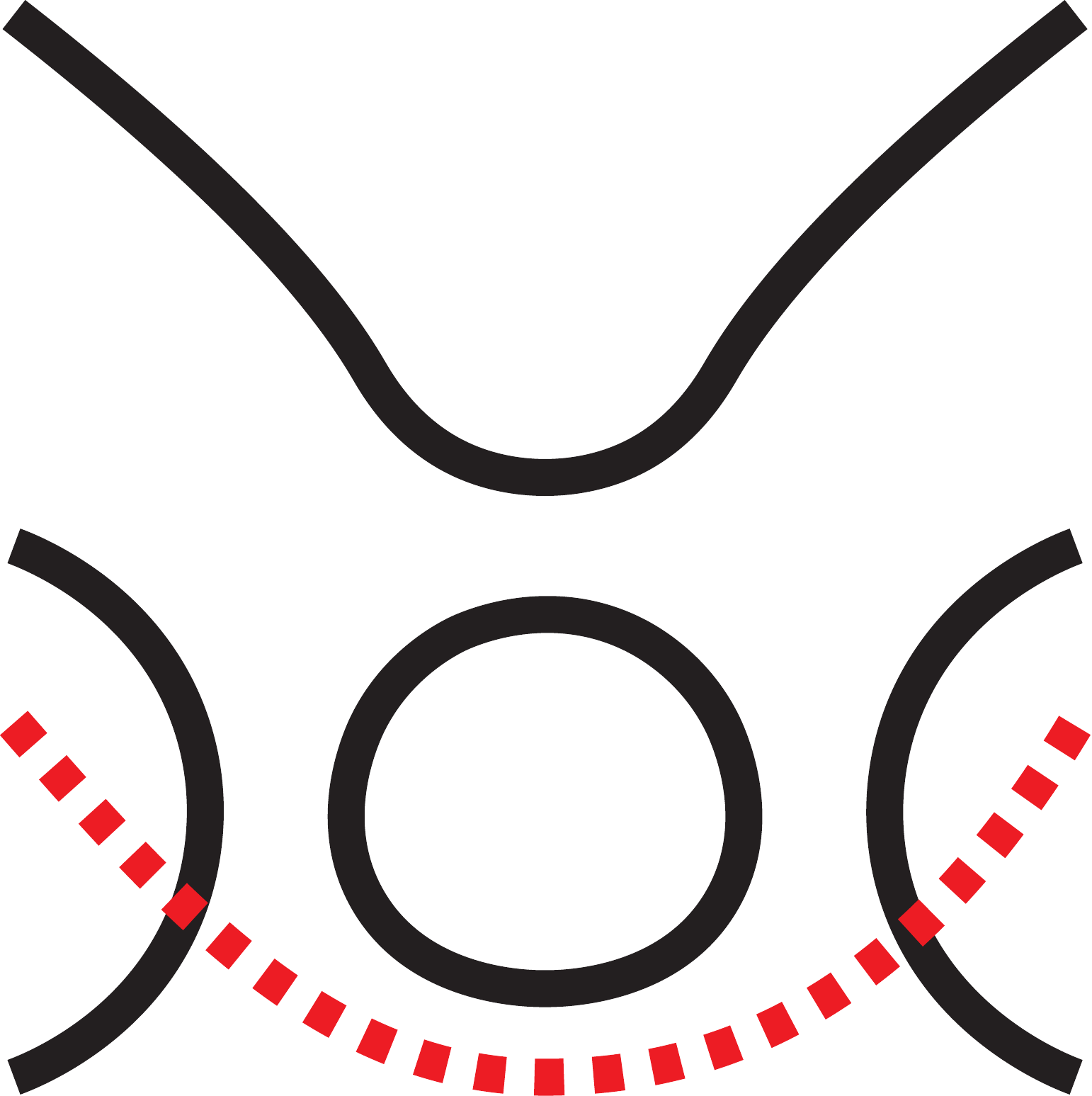} & \includegraphics[scale=0.05]{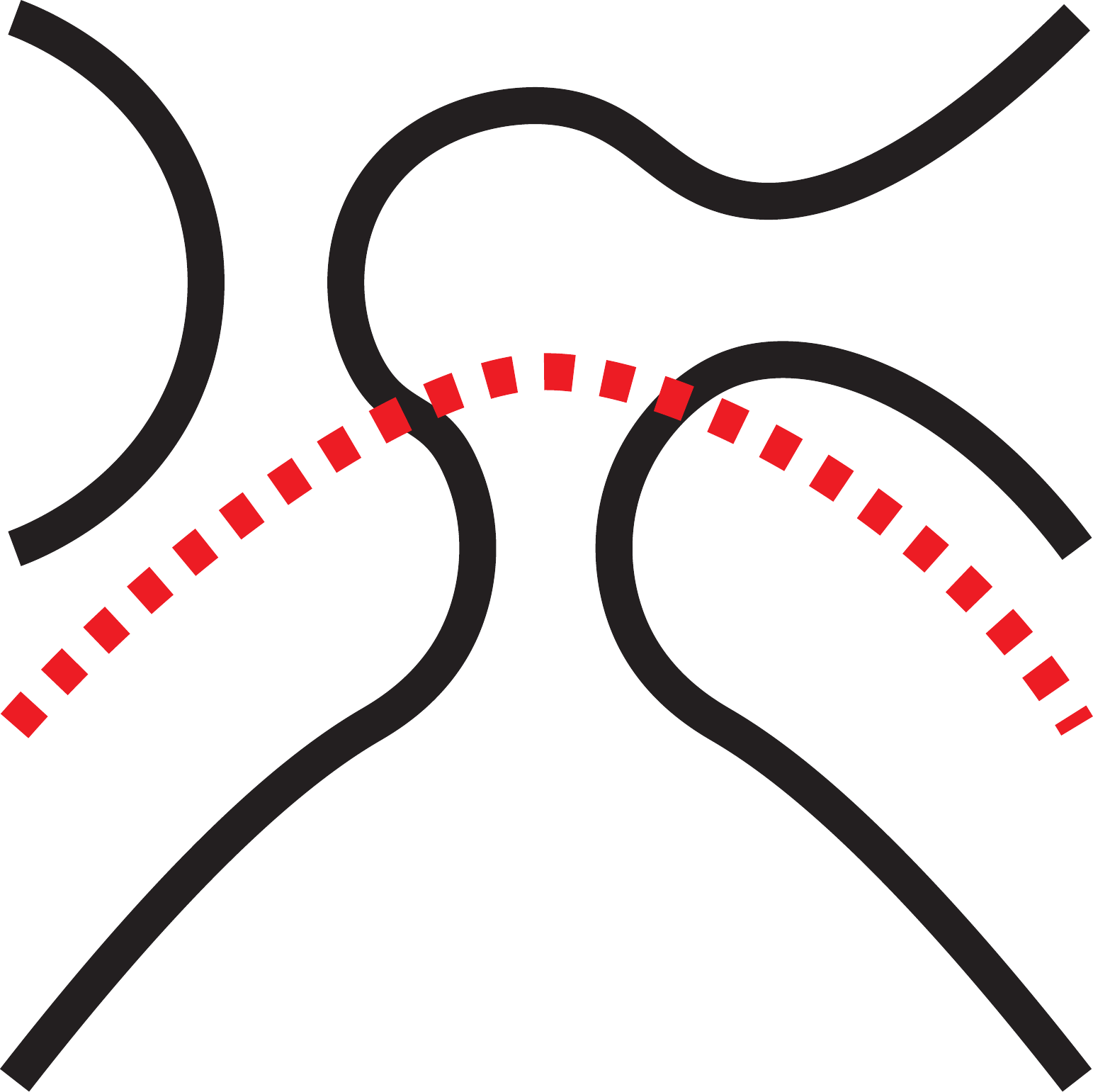} & & \includegraphics[scale=0.05]{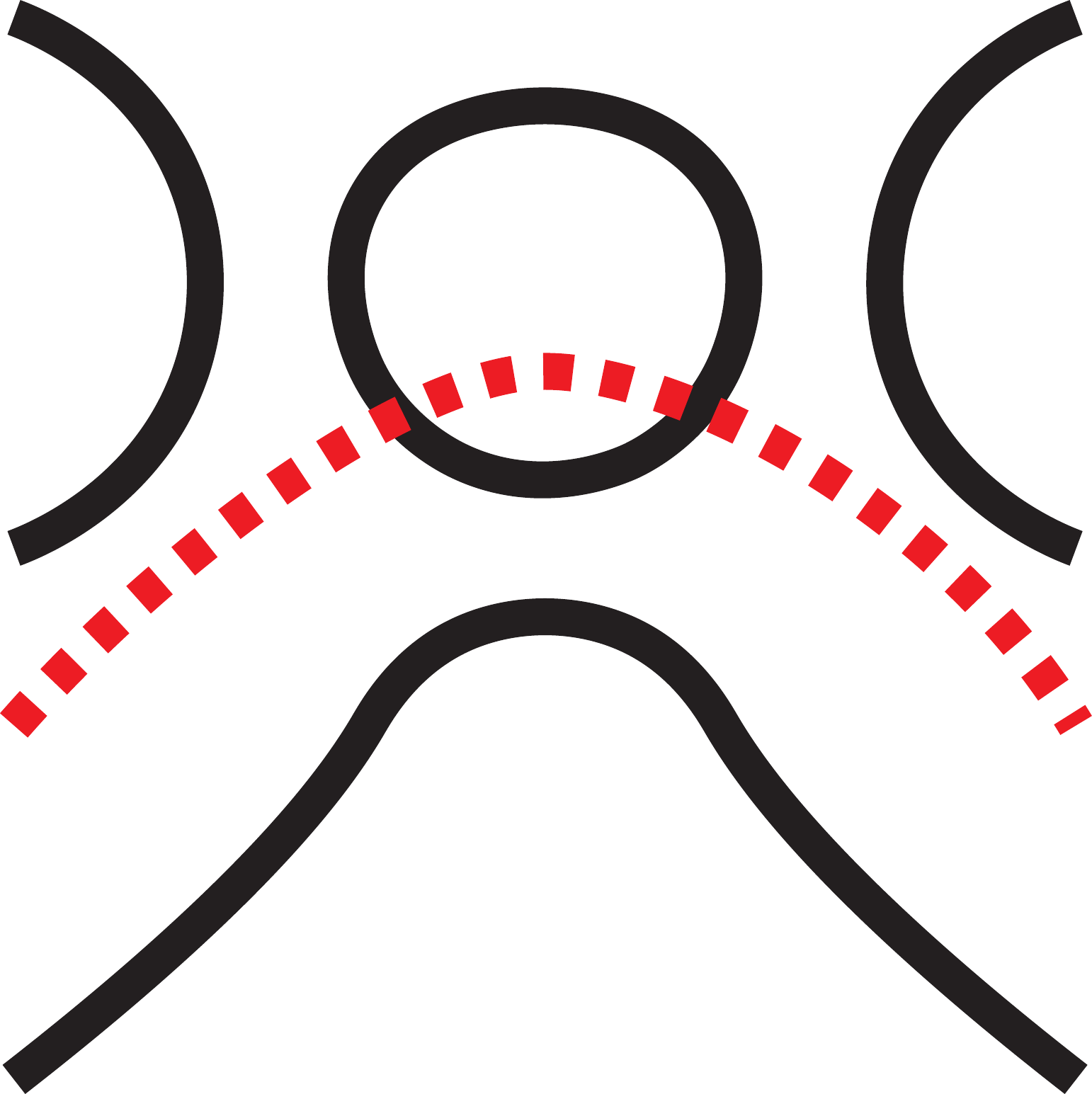}
    \end{tikzcd}
    \caption{Local intersections of the canonical shortcut with resolutions in the map ~\eqref{equ:cheR3map5}. On all three rows, either a resolution has the same local intersection number as those of 100-resolutions, or the term in ~\eqref{equ:cheR3map5} corresponding to the resolution vanishes.}
    \label{fig:u4r31stEx}
\end{figure}
Similar to Convention~\ref{conv:orionclosedcomps}, we may assume that the closed components of each one of the four resolutions are oriented to agree with the segment and closed components of the other three resolutions. Then, the algebraic intersection numbers outside the local disk are the same, and one only needs to count the intersection numbers inside, to show that $k_s\cdot \alpha$ is the same for all states in ~\eqref{equ:cheR3map5}.

On the first row of Figure ~\ref{fig:u4r31stEx}, the two intersections cancel each other trivially at every resolution except the boxed one. For the boxed resolution, if the intersections cancel each other, then all states have the same $k_s\cdot \alpha$. If not, then our convention for the choice of orientations on the closed components necessitates that the two arcs intersected by the shortcut are part of the segment component, and the map \[\scalebox{0.7}{\Dpioo}\otimes [c'_1x] \rightarrow\scalebox{0.7}{\Dpiio} \otimes [c'_1c'_2x]\] in $C(D')$ is $\nabla$. Therefore, the term\[\DpoioWithLetters{q:r}{r:q}{1}{\widetilde{p}}\,\,\,\,\,\,\,\otimes [c'_2x]\] in ~\eqref{equ:cheR3map5} corresponding to the boxed term vanishes.

On the second row, both 100-resolutions are isotopic and have the same intersection number inside the local disk. The 010-resolution on the left side either has the same local intersection number as 100-resolutions or the term itself vanishes. That is because a different local intersection number would imply a change of orientation outside the local disk, which would mean that the map \[\scalebox{0.7}{$\Dioo$}\,\otimes\, [c_1x] \rightarrow\,\scalebox{0.7}{$\Diio$} \otimes\, [c_1c_2x]\] is $\nabla$, and consequently the term \[\raisebox{0.2cm}{$\DoioWithLetters{p:q}{q:p}{1}{\widetilde{r}}$} \,\otimes [c_2x]\] vanishes. The boxed term on the second row has zero local intersection number. If the 100-resolution on the right side has nonzero local intersection number, then the two arcs (of the 100-resolution) inside the local disk intersected by the shortcut must have the same orientation (both upwards or downwards) which would mean that the map \[\scalebox{0.7}{\Dpioo}\otimes [c'_1x] \rightarrow\scalebox{0.7}{\Dpiio} \otimes [c'_1c'_2x]\] is $\nabla$, and then the boxed term vanishes.

On the third row, by a similar argument, either all local intersection numbers for all states are the same as the intersection numbers of 100-resolutions, or those with different intersection numbers vanish. 

Hence, the map in the formula ~\eqref{equ:cheR3map5} preserves the $u$-grading. This concludes the proof of Theorem ~\ref{thm:maintheorem}.

\begin{rem}
As stated in Theorem~\ref{thm.restrictW2knots}, restriction of the winding homology to knots yields the reduced Khovanov homology of knots. Therefore, in the special case of knots, the proof of invariance for the winding homology provides an alternative proof of invariance for the Khovanov homology of knots. Broadly speaking, the original proof of Khovanov in~\cite{Khovanov} and Bar-Natan's geometric proof in~\cite{Bar_Natan_2005} are ``categorifications of the proof" of invariance for the Jones polynomial. In our proof, we make this categorification explicit at the generator level. That is, a cancellation of the Kauffman brackets of two diagrams turns into a cancellation of arrows between the generators for these diagrams. For example, in the proof of invariance under Reidemeister move I, the cancellation $\langle \vcenter{\hbox{\NoCirc}} \rangle -q\cdot q^{-1} \cdot \langle \Arc \rangle$ becomes the cancellation of the arrows $\vcenter{\hbox{\NoCirc}}^1 \longrightarrow \vcenter{\hbox{\WithCirc}}\myfrac{X}{1}+\vcenter{\hbox{\WithCirc}}\myfrac{1}{X}$ and $\vcenter{\hbox{\NoCirc}}^X \longrightarrow \vcenter{\hbox{\WithCirc}}\myfrac{X}{X}$. However, these cancellations are carried out by using the zigzag lemma, which comes at the price of potentially introducing extra arrows. We resolve this issue by utilizing the monotonic reduction algorithm.
\end{rem}

\begin{rem}
In the proofs of invariance under the Reidemeister move II and III, the case-by-case analysis of the invariance of $u$-gradings under the chain homotopy equivalence maps by using the canonical shortcut is made with objective of making the verification clear to the reader. It is not a coincidence that $u$-gradings happen to be preserved in all cases. There is brief (but less clear) explanation of the reason behind, with the symbolic notation as follows. For the Reidemeister move II, the map $\raisebox{-.05cm}{\rotatebox{90}{$\LeftSoft\,\RightSoft$}} \myfrac{p}{q}\longrightarrow \raisebox{-.05cm}{$\LeftSoft^{p:q}\,\RightSoft^{q:p}$}$ (see Remark~\ref{rem:r2summary}) preserves $u$-gradings, as all edge maps in Figure~\ref{fig:localcoborisms and induced maps}. Therefore, the generators $\TwoWaves\myfrac{p}{q}$ and $\LeftSoft^{p:q}\SimpCirc^1\RightSoft^{q:p}$ have the same $u$-gradings, because the the small circle in the middle of the term $\LeftSoft^{p:q}\SimpCirc^1\RightSoft^{q:p}$ has no effect on the $u$-grading. Thus, the chain equivalence map
\begin{equation}
    \TwoWaves\myfrac{p}{q}\otimes [c_2x]+\LeftSoft^{p:q}\SimpCirc^1\RightSoft^{q:p} \otimes [c_1x] \xmapsto{(-1)^i} \RIInocros \myfrac{p}{q} \otimes [x]
\end{equation}
preserves $u$-gradings. For Reidemeister move III, the argument for $u$-gradings is similar.
\end{rem}

%------------------------------------------------------------------------

\section{Computations and Examples}\label{sec.computations}

In this section, we present examples of knotoid pairs comparing the strength of the invariants, as claimed in Figure ~\ref{fig:allinvariantscompared}, and other computational results.

The program to compute the winding homology of knotoids is written in Mathematica language, and is available online; see ~\cite{Kutluay}. We took Bar-Natan's program for Khovanov homology of knots in his Mathematica package $\lq\lq$KnotTheory" (see ~\cite{BarNatanWeb}) as the base for our program, and expanded on the implementation to knotoids, and the winding homology. Knotoid diagrams are presented in planar diagram (PD) notation, and the problem of the choice of shortcut on diagrams is circumvented by Lemma ~\ref{lem:shortcutformula}. Our program also includes commands for direct (and naturally faster) computation of the Turaev polynomial, the Khovanov homology, and the Jones polynomial of knotoids. The simplest separation examples that we were able to find are presented below.
\begin{ex}
The knotoids $K_1$ and $K_2$ (see Figure ~\ref{fig:k1andk2}\footnote{The knotoid diagrams in the figures of this section are generated with the help of $\lq\lq$DrawPD" command of KnotTheory package.}) have the same Jones polynomial, and the same Khovanov homology, but they are distinguished by the Turaev polynomial.
\end{ex}
Recalling that $W_{K}(t,q,u)$ denotes the Poincar\'e polynomial of the winding homology $H(K)=\bigoplus_{i,j,k} H_{i,j}^k(K)$ for a knotoid $K\subset S^2$, we compute
\begin{align*}
    W_{K_1}(t,q,u)&=A(t,q)+u^2B(t,q), \\
    W_{K_2}(t,q,u)&=A(t,q)+\frac{1}{u^2}B(t,q),
\end{align*}
where
\begin{align*}
    A(t,q)&= 7+\frac{1}{q^4 t^2} + \frac{3}{q^2 t} + 12 q^2 t + 16 q^4 t^2 + 15 q^6 t^3 + 
11 q^8 t^4 + 5 q^{10} t^5 + q^{12} t^6,\\
    B(t,q)&= 19q + \frac{1}{q^7 t^4} + \frac{4}{q^5 t^3} + \frac{9}{q^3 t^2} + \frac{15}{q t} + 19 q^3 t + 14 q^5 t^2 + 7 q^7 t^3 + 2 q^9 t^4.
\end{align*}
Using the substitutions, we have
\begin{align}
T_{K_1} (q,u) = W_{K_1}(t=-1,q,u) &\not= W_{K_2}(t=-1,q,u) =T_{K_2}(q,u), \nonumber \\
Kh_{K_1}(t,q)=W_{K_1}(t,q,u=1)&=W_{K_2}(t,q,u=1)=Kh_{K_2}(t,q), \nonumber \\
J_{K_1}(q)=Kh_{K_1}(t=-1,q)&=Kh_{K_2}(t=-1,q)=J_{K_2}(q). \nonumber
\end{align}
We conclude that
\begin{itemize}
    \item[i.] the Turaev polynomial is stronger than the Jones knotoid polynomial,
    \item[ii.] the winding homology is stronger than the Khovanov knotoid homology,
    \item[iii.] the Turaev polynomial can distinguish a pair with the same Khovanov knotoid homology.
\end{itemize}
\begin{rem}
For knots, the Turaev polynomial is equal to the Jones polynomial, and the winding homology is equal to the Khovanov homology, since it is possible to choose a shortcut that does not intersect the knotoid. 
\end{rem}
\begin{rem}
Note that the $u$-breadth of $W_K$ is less than or equal to twice the minimum number of intersections between shortcut $\alpha$ and $K$. For this reason, only two distinct powers of $u$ appear in $W_{K_1}$, and $W_{K_2}$. This is also the case for $W_{K_3}$ through $W_{K_6}$. All knotoids from $K_1$ through $K_6$ are obtained by performing a single $\Omega_-$ (or $\overline{\Omega}_-$) move on the arc labeled $\lq\lq 1"$ in the PD notation (as recorded in the KnotTheory database) of the corresponding knot. More precisely, for a knot with $n$ crossings, the overpass or underpass of the knot, between the labels $\lq\lq 1"$ and $\lq\lq2n"$, is deleted.
\end{rem}

\begin{figure}[hbt!]
    \centering
    \begin{subfigure}[b]{0.46\textwidth}
        \centering
        \includegraphics[width=\textwidth]{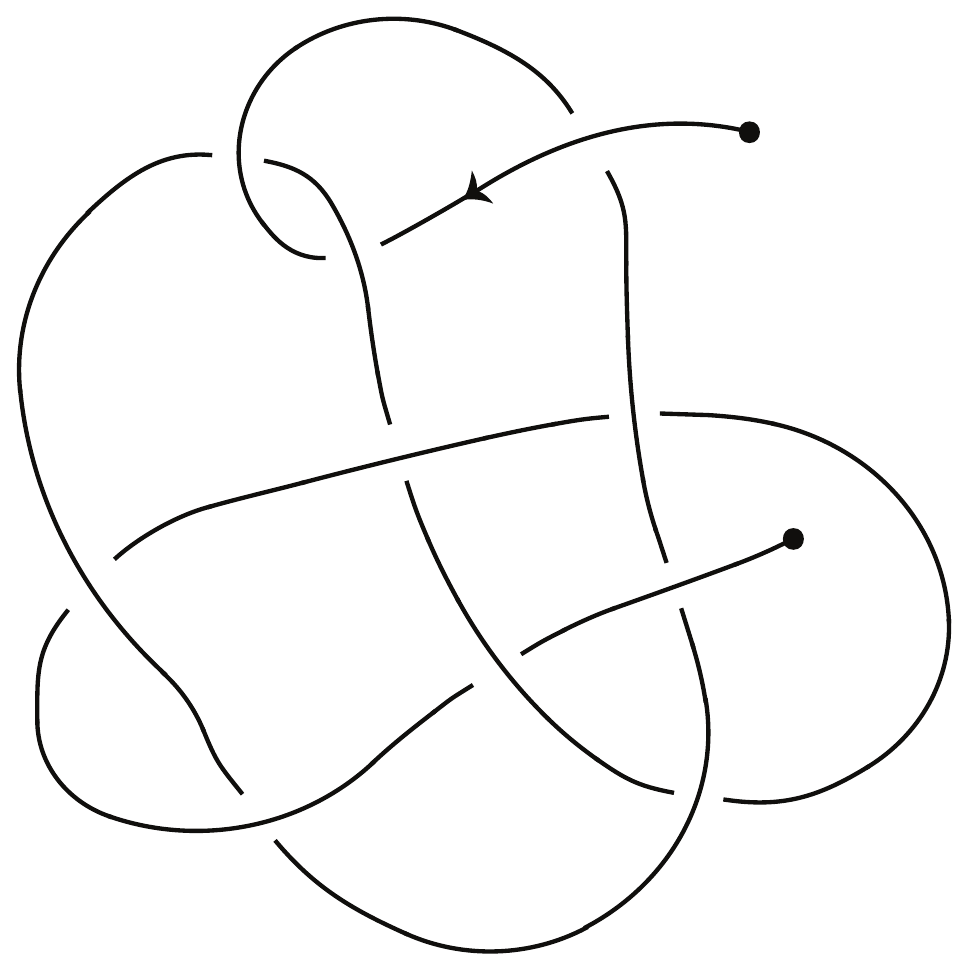}
        \caption{\(K_1\)}
    \end{subfigure}
    \hfill
    \begin{subfigure}[b]{0.46\textwidth}
        \centering
        \includegraphics[width=\textwidth]{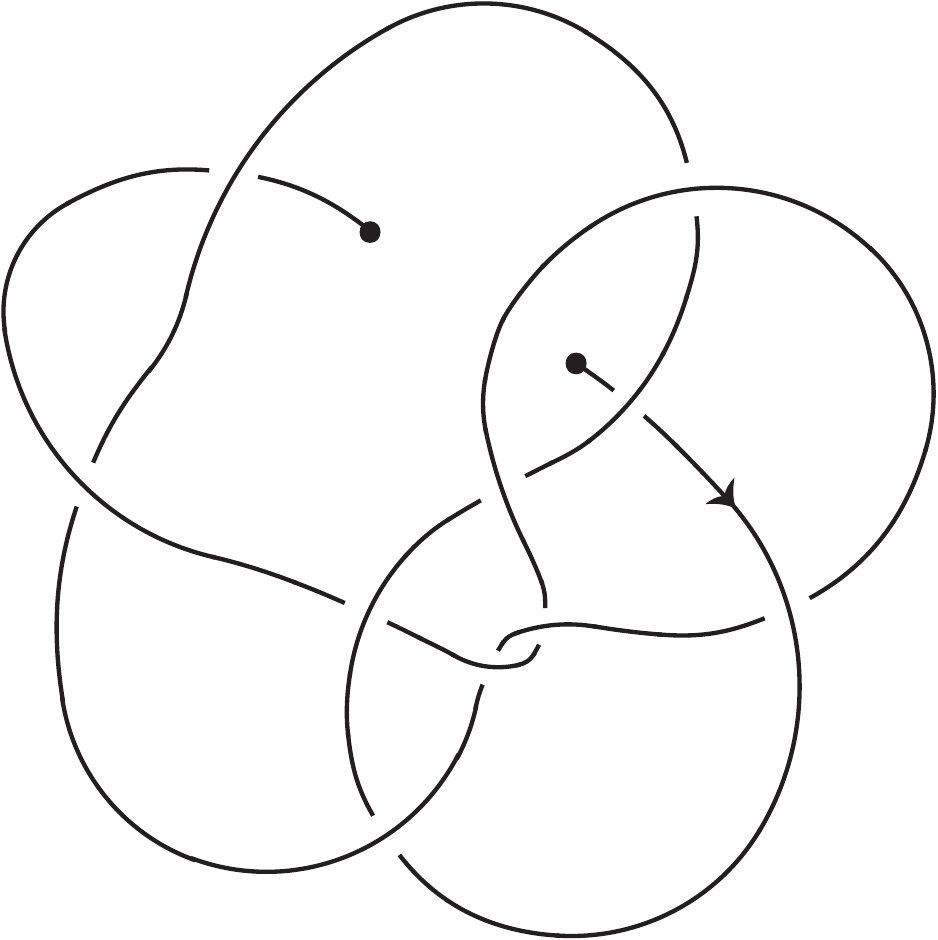}
        \caption{\(K_2\)}
    \end{subfigure}

    \caption{$K_1$ is obtained from the knot $11^a_{138}$ by applying an $\Omega_-$ move, and $(K_1)_-=11^a_{138}$. $K_2$ is obtained from the knot $\overline{11}^a_{285}$ by applying an $\overline{\Omega}_-$ move, and $(K_2)_+=\overline{11}^a_{285}$.}
    \label{fig:k1andk2}
\end{figure}
\begin{figure}[hbt!]
    \centering
    \begin{subfigure}[b]{0.46\textwidth}
        \centering
        \includegraphics[width=\textwidth]{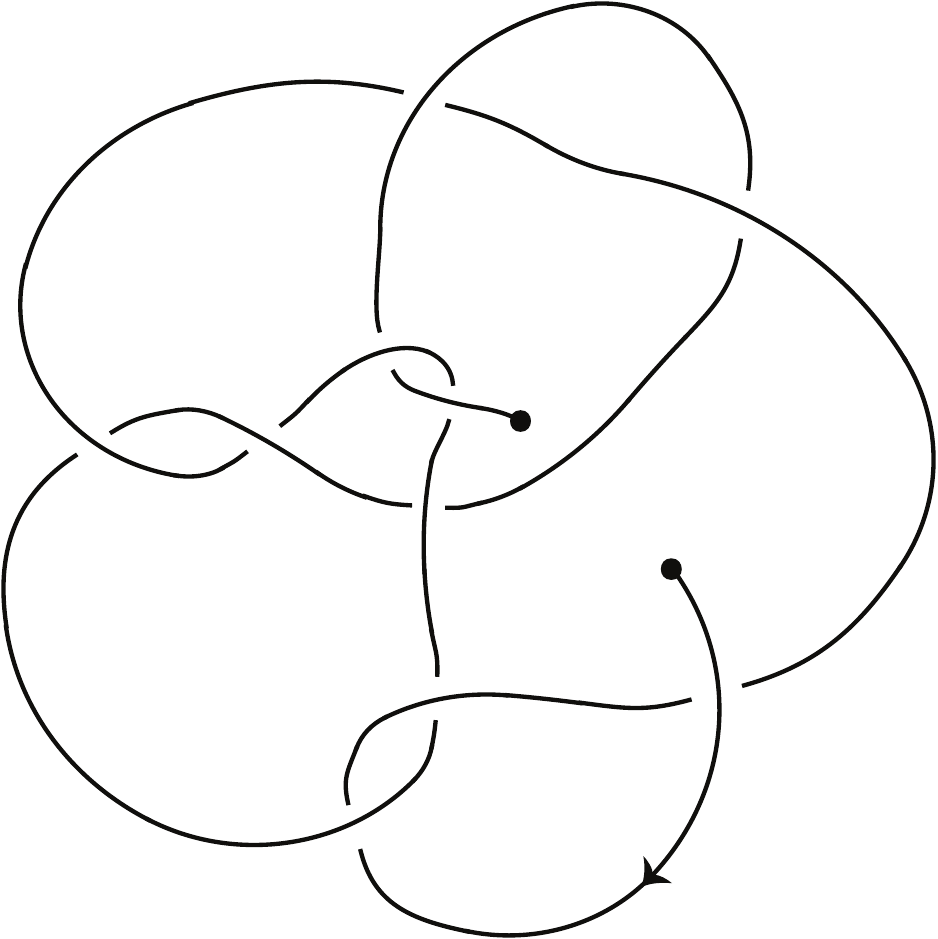}
        \caption{\(K_3\)}
    \end{subfigure}
    \hfill
    \begin{subfigure}[b]{0.46\textwidth}
        \centering
        \includegraphics[width=\textwidth]{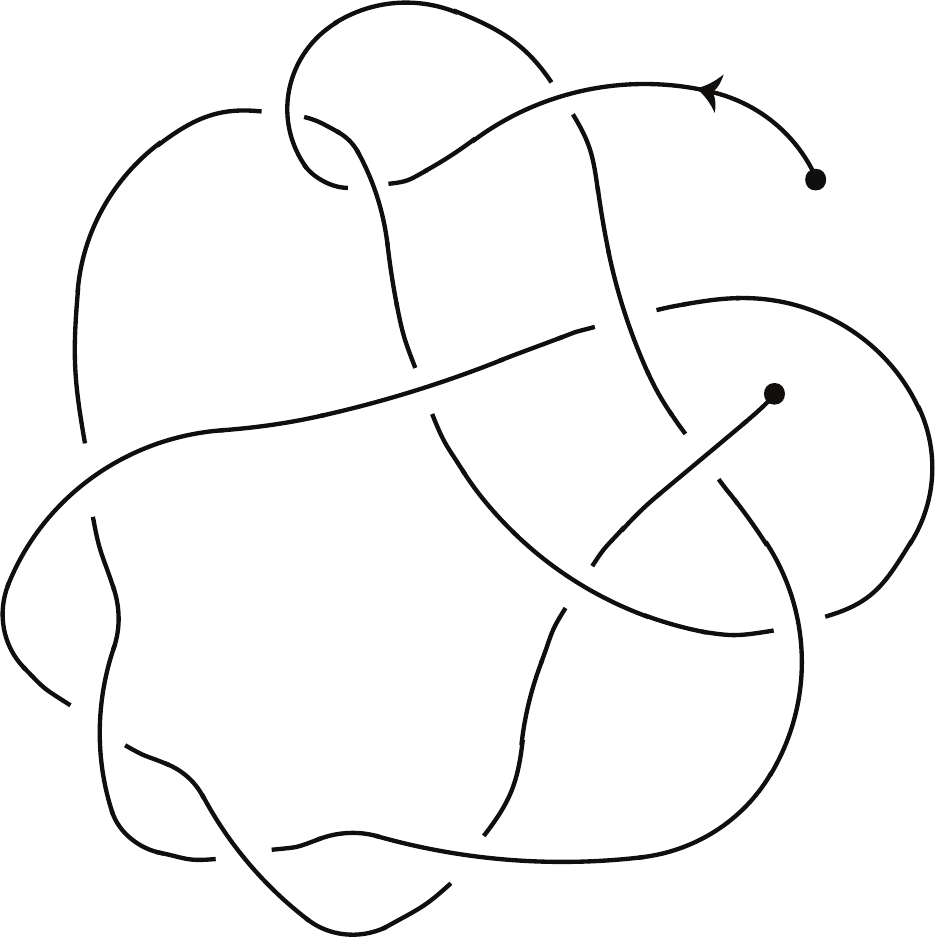}
        \caption{\(K_4\)}
    \end{subfigure}

    \caption{$K_3$ is obtained from the knot $11^a_{217}$ by applying an $\Omega_-$ move, and $(K_3)_-=11^a_{217}$. $K_4$ is obtained from the knot $13^n_{3187}$ by applying an $\Omega_-$ move, and $(K_4)_-=13^n_{3187}$.}
    \label{fig:k3andk4}
\end{figure}
\begin{ex}
The knotoids $K_3$ and $K_4$ (see Figure ~\ref{fig:k3andk4}) have the same Jones polynomial and the same Turaev polynomial, but they are distinguished by the Khovanov homology.
\end{ex}More explicitly,
\begin{align*}
W_{K_3}(t,q,u)&= \frac{13}{q^2} + \frac{1}{q^8 t^3} + \frac{5}{q^6 t^2} + \frac{9}{q^4 t} + 13 t + 9 q^2 t^2 + 
 4 q^4 t^3 + q^6 t^4 +u^2 C(t,q), \\
W_{K_4}(t,q,u)&= \frac{13}{q^2} + \frac{1}{q^{10} t^5} + \frac{1}{q^{10} t^4} + \frac{1}{q^8 t^4} + \frac{2}{q^8 t^3} + \frac{5}{q^6 t^2} + \frac{9}{q^4 t} \\
&+ 13 t + 9 q^2 t^2 + 4 q^4 t^3 +q^6 t^4 +u^2 C(t,q),
\end{align*}
where
\begin{align*}
    C(t,q)=\frac{13}{q} + \frac{1}{q^{13} t^6} + \frac{3}{q^{11} t^5} + \frac{8}{q^9 t^4} + \frac{14}{q^7 t^3} + \frac{18}{q^5 t^2} + \frac{18}{q^3 t} + 7 q t + 2 q^3 t^2.
\end{align*}
Using the substitutions, we obtain
\begin{align}
T_{K_3}(q,u) = W_{K_3}(t=-1,q,u) &= W_{K_4}(t=-1,q,u) =T_{K_4}(q,u), \nonumber \\
Kh_{K_3}(t,q)=W_{K_3}(t,q,u=1)&\not=W_{K_4}(t,q,u=1)=Kh_{K_4}(t,q), \nonumber \\
J_{K_3}(q)=T_{K_3}(q,u=1)&=T_{K_4}(q,u=1)=J_{K_4}(q). \nonumber
\end{align}
We conclude that
\begin{itemize}
    \item[i.] the Khovanov knotoid homology is stronger than the Jones knotoid polynomial,
    \item[ii.] the winding homology is stronger than the Turaev polynomial,
    \item[iii.] the Khovanov knotoid homology can distinguish a pair with the same Turaev polynomial. The other direction was shown in the previous example, and thus, the Turaev polynomial and the Khovanov knotoid homology are not comparable.
\end{itemize}
\begin{rem}
It is well known that the Khovanov knot homology is a stronger invariant than the Jones knot polynomial. The example above shows that this also holds for pure knotoids.
\end{rem}

In $W_{K_3}$, and $W_{K_4}$, the terms containing $u^2$ (or the terms with odd powers of $q$) are the same, and only the terms with even powers of $q$ differ. In general, this is not necessarily the case. The knotoids $K_5$ and $K_6$ (see Figure ~\ref{fig:k5andk6}) also have the same Jones polynomial, and Turaev polynomial. Their winding homology, however, differs on terms that have nonzero powers of $u$ (which have odd powers of $q$) as well as those with even powers of $q$, as follows
\begin{align}
W_{K_5}(t,q,u)&= \frac{1}{q^2} + \frac{1}{q^{16} t^7} + \frac{3}{q^{14} t^6} + \frac{6}{q^{12} t^5} + \frac{9}{q^{10} t^4} + \frac{9}{q^8 t^3} + \frac{7}{q^6 t^2} + \frac{3}{q^4 t} \nonumber \\ 
 &+\frac{1}{u^2}\left( \frac{3}{q^3} + \frac{1}{q^{17} t^8} + \frac{2}{q^{15} t^7} + \frac{3}{q^{13} t^6} + \frac{2}{q^{11} t^5}+ \frac{1}{q^9 t^3} + \frac{3}{q^7 t^2} + \frac{4}{q^5 t} + \frac{t}{q} \right), \nonumber
\end{align}
and
\begin{align}
W_{K_6}(t,q,u)&= \frac{1}{q^4} + \frac{1}{q^2} + \frac{1}{q^{16} t^7} + \frac{3}{q^{14} t^6} + \frac{6}{q^{12} t^5} + \frac{9}{q^{10} t^4} + \frac{9}{q^8 t^3} + \frac{7}{q^6 t^2} + \frac{4}{q^4 t} \nonumber \\ 
& +\frac{1}{u^2}\left( \frac{3}{q^3} + \frac{1}{q^{17} t^8} + \frac{2}{q^{15} t^7} + \frac{3}{q^{13} t^6} + \frac{3}{q^{11} t^5}\right. \nonumber \\ 
&+ \left. \frac{1}{q^{11} t^4} + \frac{1}{q^9 t^4}+ \frac{2}{q^9 t^3} + \frac{3}{q^7 t^2} + \frac{4}{q^5 t} + \frac{t}{q} \right ). \nonumber
\end{align}
Using the substitutions, we obtain
\begin{align}
T_{K_5}(q,u) = W_{K_5}(t=-1,q,u) &= W_{K_6}(t=-1,q,u) =T_{K_6}(q,u), \nonumber \\
Kh_{K_5}(t,q)=W_{K_5}(t,q,u=1)&\not=W_{K_6}(t,q,u=1)=Kh_{K_6}(t,q), \nonumber \\
J_{K_5}(q)=T_{K_5}(q,u=1)&=T_{K_6}(q,u=1)=J_{K_6}(q). \nonumber
\end{align}

\begin{figure}[hbt!]
    \centering
    \begin{subfigure}[b]{0.46\textwidth}
        \centering
        \includegraphics[width=\textwidth]{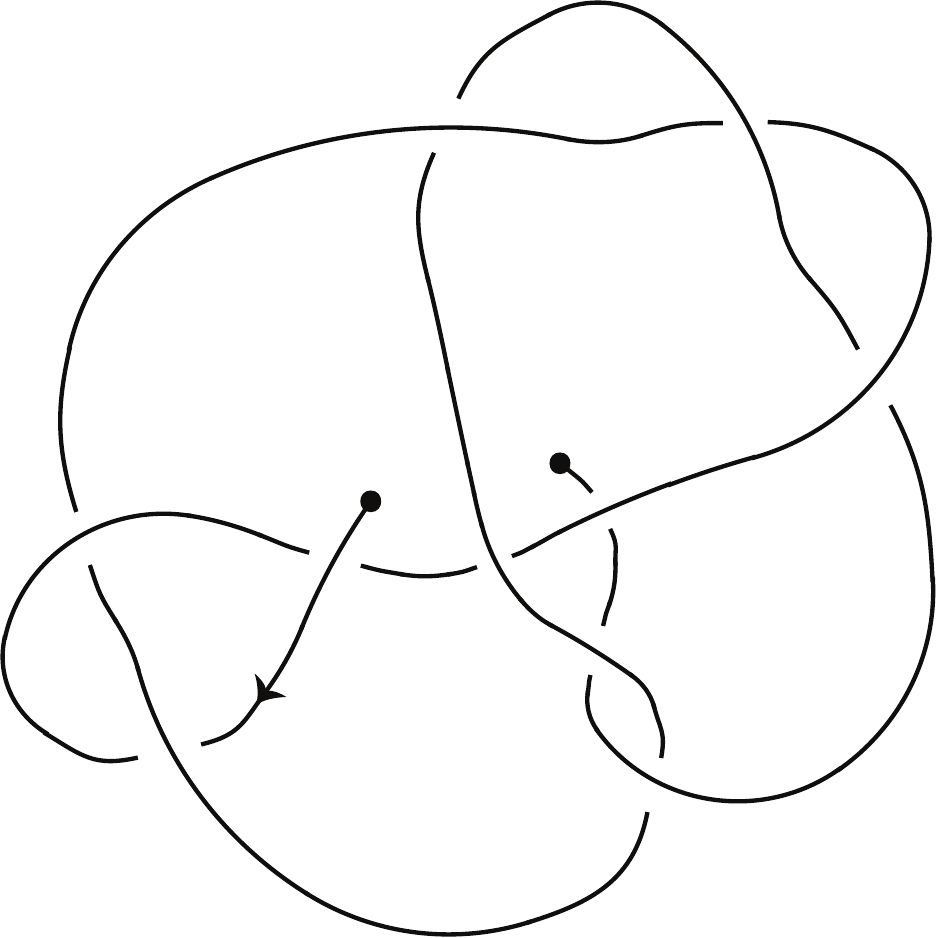}
        \caption{\(K_5\)}
    \end{subfigure}
    \hfill
    \begin{subfigure}[b]{0.46\textwidth}
        \centering
        \includegraphics[width=\textwidth]{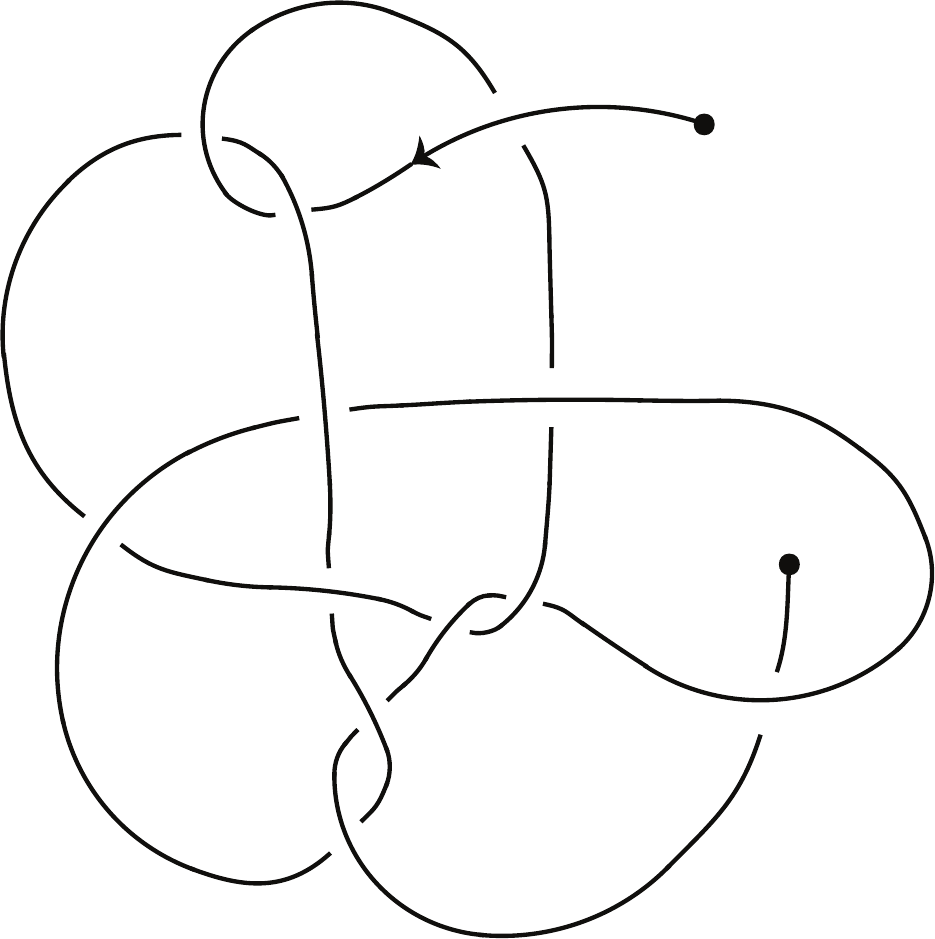}
        \caption{\(K_6\)}
    \end{subfigure}

    \caption{$K_5$ is obtained from the knot $11^n_{113}$ by applying an $\overline{\Omega}_-$ move, and $(K_5)_+=11^n_{113}$. $K_6$ is obtained from the knot $13^n_{2685}$ by applying an $\overline{\Omega}_-$ move, and $(K_6)_+=13^n_{2685}$.}
    \label{fig:k5andk6}
\end{figure}

The examples above do not rule out the possibility that the Turaev polynomial and the Khovanov knotoid homology together as a combined invariant distinguishes any knotoid pair that are distinguished by their winding homology. Next, we answer this question: is there a pair of knotoids that have the same Turaev polynomial, and the same Khovanov knotoid homology, simultaneously, yet they are distinguished by their winding homology? We show that such pairs of knotoids exist indeed. Our strategy to find such an example is to look for knotoids whose Turaev polynomials remain the same under the operation $K \leadsto \textnormal{Sym}(\textnormal{Mir}(K))$, since we know that Khovanov homology does not change under this operation; see Proposition ~\ref{prop:propertiesofw}. The Turaev polynomial of our potential example would need to be symmetric in the powers of $u$, that is, for each term $ C\, t^a q^b u^k$ in the polynomial, the term $ C\, t^a q^b u^{-k}$ must also be included. We start from the simplest case, where the nonzero powers of $u$ contained in each monomial of $T_K(q,u)$ are only $u^2$, and $u^{-2}$. This condition imposes the following restrictions on the knotoid diagrams of interest: (1) any choice of shortcut $\alpha$ must have at least 2 intersections with the knotoid $K$, and (2) at least 2 of these intersections between $K$ and $\alpha$ must have opposite signs. 
\begin{figure}[hbt!]
    \centering
    \begin{subfigure}[b]{0.32\textwidth}
        \centering
        \includegraphics[width=\textwidth]{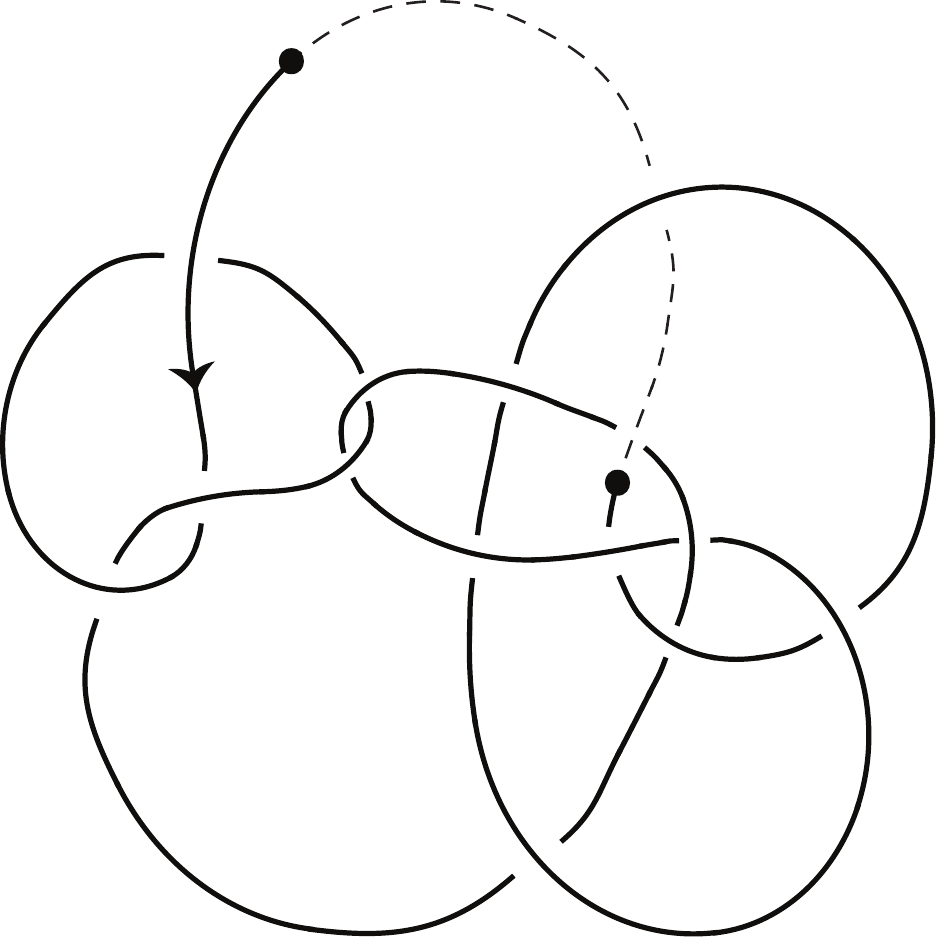}
        \caption{\(K_7\)}
    \end{subfigure}
    \hfill
    \begin{subfigure}[b]{0.32\textwidth}
        \centering
        \includegraphics[width=\textwidth]{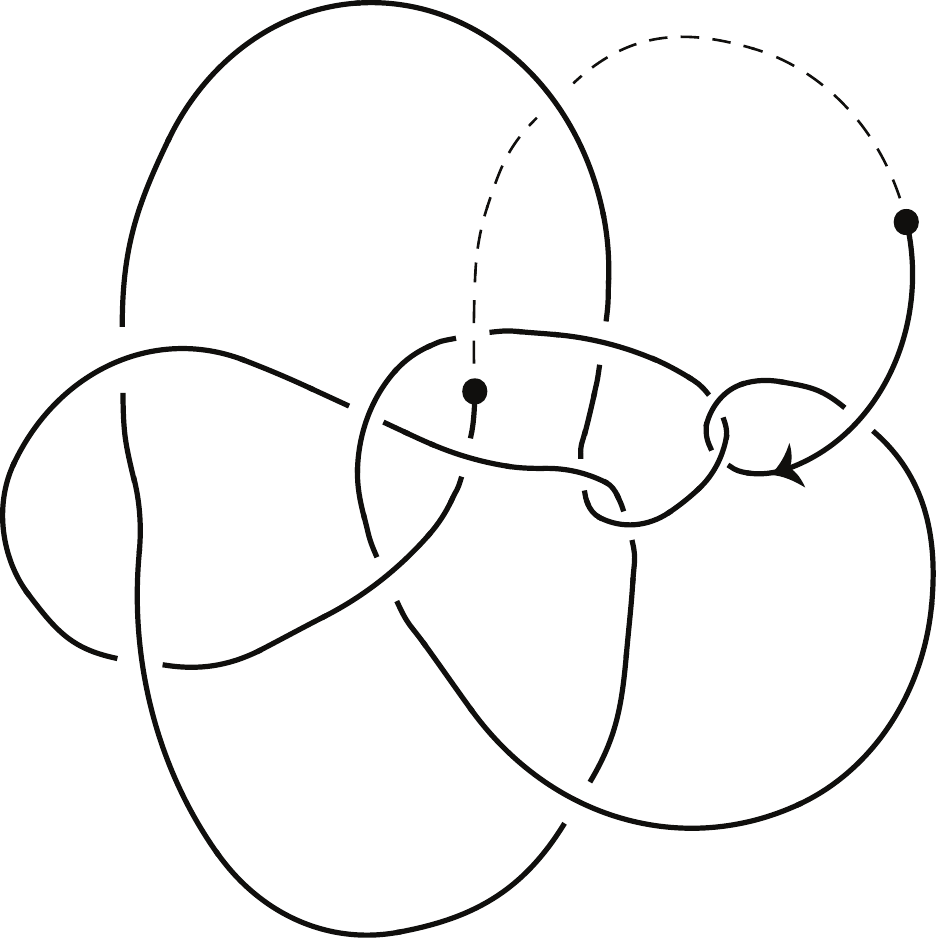}
        \caption{\(K_8\)}
    \end{subfigure}
    \hfill
    \begin{subfigure}[b]{0.32\textwidth}
        \centering
        \includegraphics[width=\textwidth]{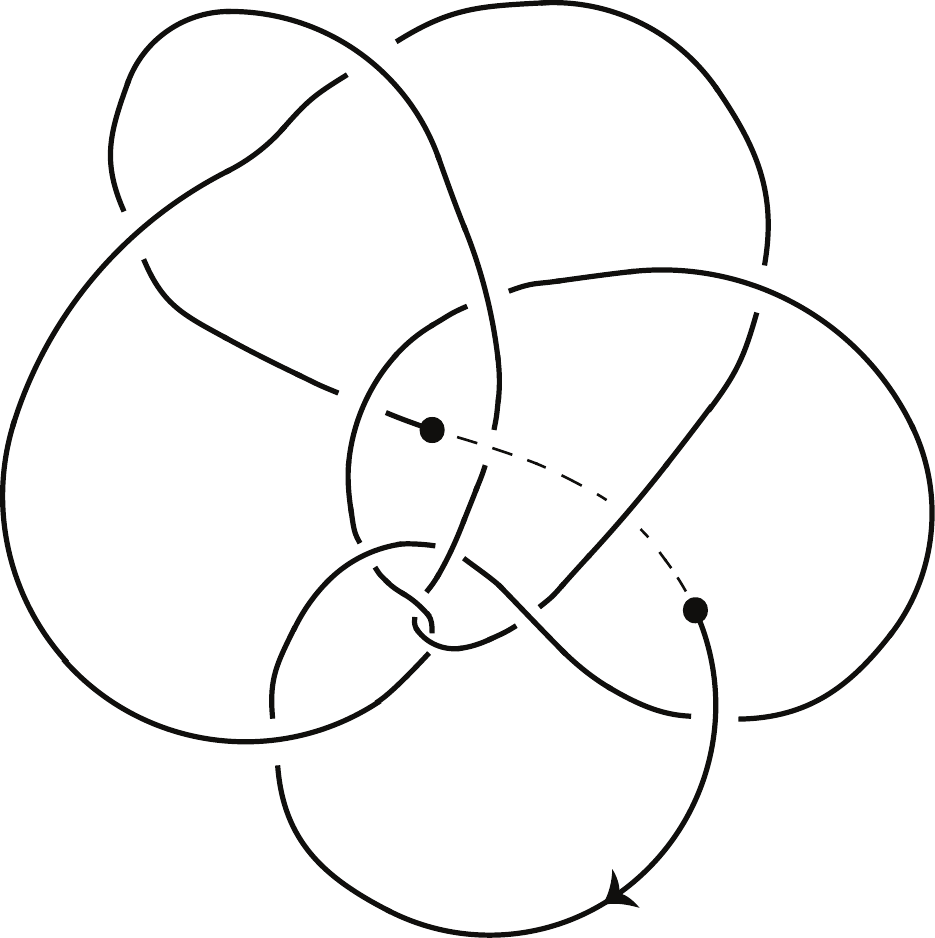}
        \caption{\(K_9\)}
    \end{subfigure}

    \caption{The knotoids $K_7$, $K_8$, $K_9$ are obtained from the knots $14^n_{3532}$, $14^n_{12378}$, $14^n_{22768}$, respectively, by deleting the dashed arcs -- alternatively, applying an $\Omega_-$, and an $\overline{\Omega}_-$ move on the knots. Note that the dashed arcs here are not shortcuts, therefore $K_\pm$'s are not necessarily the original knots for $K=K_7$, $K_8$, $K_9$.}
    \label{fig:k7k8andk9}
\end{figure}
To generate candidate knotoids for the desired example, we considered all knots up to 14 crossings in Hoste-Thistlethwaite table of knots (see ~\cite{HTW}) and applied $\Omega_-$ (or $\overline{\Omega}_-$) move twice on the arc labeled 1 in the PD notation (as recorded in the KnotTheory database) of the knot. In other words, the knot is $\lq\lq$cut" at the arc labeled 1, and the head of the knot(oid) is slid backwards, passing under/over two arcs of the knot. 
\begin{rem}
Note that this operation is not well-defined as the resulting knotoid depends on the choice of the initial arc that is cut, and also, there is no guarantee that the knotoids generated from different knots will be distinct.
\end{rem}
The knotoids that do not satisfy the restrictions above are then eliminated. We calculated the Turaev polynomial of all the candidates, and further eliminated those whose Turaev polynomials are changed under the replacement $u\rightarrow u^{-1}$. Then the winding homology of the remaining knotoids is calculated. At the end of this process, we found three distinct knotoids (see Figure ~\ref{fig:k7k8andk9}) satisfying the desired property. Namely, 
\begin{ex}
For $i=7,8,9$, the knotoid $K_i$ has the same Turaev polynomial and the same Khovanov homology, as $\textnormal{Sym}(\textnormal{Mir}(K_i))$. However, $K_i$ is distinguished from its symmetric mirror by the winding homology.
\end{ex}
More explicitly,
\begin{align*}
W_{K_7}(t,q,u)=& 10 q^2 + \frac{2}{q^2 t^2} + \frac{6}{t} + 9 q^4 t + 5 q^6 t^2 + q^8 t^3 + q^8 t^4 +
  q^{10} t^5  \\
  &+\frac{1}{u^2} \left( \frac{1}{q^3 t^3} + \frac{2}{q t^2} + \frac{q}{t} + q^3 t + 3 q^5 t^2 + 3 q^7 t^3+ q^9 t^4 \right) \\
  &+ u^2 \left( q  + q^3 + \frac{1}{q^3 t^3} + \frac{2}{q t^2} + \frac{2 q}{t} + 2 q^3 t + 3 q^5 t^2 + 3 q^7 t^3 + q^9 t^4 \right),
\end{align*}
so that
\begin{align}
    W_{\textnormal{Sym}(\textnormal{Mir}(K_7))}(t,q,u)&= W_{K_7}(t,q,u^{-1}) \label{equ:K7first} \\
    &\not = W_{K_7}(t,q,u), \nonumber \\ \nonumber \\
    T_{\textnormal{Sym}(\textnormal{Mir}(K_7))}(q,u)
    &= W_{\textnormal{Sym}(\textnormal{Mir}(K_7))}(t=-1,q,u) \label{equ:K7middle} \\ 
    &= W_{K_7}(t=-1,q,u^{-1}) \nonumber \\
    &= W_{K_7}(t=-1,q,u) \nonumber \\
    &=T_{K_7}(q,u), \nonumber 
\end{align}
and
\begin{align}
    Kh_{\textnormal{Sym}(\textnormal{Mir}(K_7))}(t,q) &= W_{\textnormal{Sym}(\textnormal{Mir}(K_7))}(t,q,u=1) \label{equ:K7last} \\
    &=  W_{K_7}(t,q,u=1) \nonumber \\
    &= Kh_{K_7}(t,q). \nonumber
\end{align}
Similarly, we have
\begin{align*}
    W_{K_8}(t,q,u)=& 12 + \frac{1}{q^{10} t^6} + \frac{3}{q^8 t^5} + \frac{6}{q^6 t^4} + \frac{6}{q^4 t^3} + \frac{2}{q^4 t^2} + \frac{3}{q^2 t^2} \\
    &+ \frac{1}{t} + \frac{6}{q^2 t} + 14 q^2 t + 10 q^4 t^2 + 4 q^6 t^3 + q^8 t^4 \\
    &+\frac{1}{u^2} \left( 6 q + \frac{1}{q^7 t^4} + \frac{4}{q^5 t^3} + \frac{8}{q^3 t^2} + \frac{9}{q t}
    + 3 q^3 t + q^3 t^2 + q^5 t^2 + q^5 t^3 \right)\\
    &+ u^2 \left( 6 q  + \frac{1}{q^7 t^4} + \frac{4}{q^5 t^3} + \frac{8}{q^3 t^2} + \frac{9}{q t} + 2 q^3 t \right),
\end{align*}
and
\begin{align*}
W_{K_9}(t,q,u)=& 11 + \frac{1}{q^8 t^4} + \frac{4}{q^6 t^3} + \frac{10}{q^4 t^2} + \frac{13}{q^2 t} + t + 
 6 q^2 t  \\
 &+ 3 q^2 t^2 + 2 q^4 t^2 + 5 q^4 t^3 + 5 q^6 t^4 + 3 q^8 t^5 + q^{10} t^6 \\
 &+ \frac{1}{u^2} \left( \frac{6}{q} + \frac{1}{q^5 t^3} + \frac{1}{q^5 t^2}+ \frac{1}{q^3 t^2}+ \frac{3}{q^3 t}
 + 8 q t + 7 q^3 t^2 + 4 q^5 t^3 + q^7 t^4 \right) \\
 &+u^2 \left( \frac{6}{q} + \frac{2}{q^3 t} + 8 q t + 7 q^3 t^2 + 4 q^5 t^3 + q^7 t^4 \right),
\end{align*}
and the equations ~\eqref{equ:K7first},~\eqref{equ:K7middle},~\eqref{equ:K7last} also hold for $K_8$ and $K_9$.

%------------------------------------------------------------------------

\section{Refined invariants for knotoids in \texorpdfstring{$\R^2$}{}} \label{sec:r2knotoids}
As shown in Section 10 of \cite{Turaev}, the inclusion map $\R^2 \hookrightarrow S^2$ induces a map from the set of knotoids in $\R^2$ into the set of knotoids in $S^2$, which is surjective but not injective. In fact, the number of knotoids in $\R^2$ is much bigger than the number of knotoids in $S^2$ with the same crossing number; see ~\cite{GDS}. Thus, an increase in the number of variables of the polynomials for knotoids in $\R^2$ is expected.

In this section, we introduce the winding potential function of a smooth, closed, oriented curve in $\R^2$. This function is well-defined on $\R^2$, and we use it to refine two different Turaev polynomials for knotoids in $\R^2$.

\subsection{Winding potential function and the algebraic intersection numbers}
Let $K$ be a knotoid in $\R^2$, and $\alpha$ be a shortcut oriented from the leg to the head of $K$. Without loss of generality, we assume that $\gamma=K\cup \alpha^r$ is an oriented, immersed, closed curve with no multiple points besides double points, where $\alpha^r$ is $\alpha$ with reversed orientation. For a point $p\in \R^2-\gamma$, suppose that $(r_p(t),\theta_p(t))$, for $0\leq t\leq 1$, is a parametrization of $\gamma$ in polar coordinates based at $p$. Then the winding number at $p$ is defined by
\begin{equation}
    w_\gamma(p)=\frac{\theta_p(1)-\theta_p(0)}{2\pi}\in \Z.
\end{equation}
We extend this definition to all points on $\R^2$ as follows. For a point $p\in \gamma$ that is not a double point, using a polar coordinate parametrization based at $p$ again, but this time for $0<t<1$, we define
\begin{equation}
    w_\gamma(p)=\frac{\underset{t\rightarrow 1-}{\lim}\theta_p(t)-\underset{t\rightarrow 0+}{\lim}\theta_p(t)}{2\pi}\in\Z+\frac{1}{2}.
\end{equation}
which is a half integer that is the average of the values of $w_\gamma$ on the regions to the left and right of $\gamma$ near $p$. At a double point, we set $w_\gamma$ as the average of the values of $w_\gamma$ at four neighboring regions, or equivalently as the average of the values on two consecutive arcs before and after the double point. The extended map $w_\gamma:\R^2\rightarrow\frac{1}{2}\Z$ is called \emph{the winding potential function} of $\gamma$. The word $\lq\lq$potential" is used, because the extended winding function is a scalar potential in the sense that the difference in the values of the winding function at two points does not depend on the path connecting the points. In particular, the difference in winding potentials at the head, and the leg of the knotoid is given by the algebraic intersection numbers as follows.
\begin{lem}
Let $H$ and $L$ be the head and leg of $K\subset \R^2$, respectively. Then $w_\gamma(H)-w_\gamma(L)=\alpha\cdot K$.
\end{lem}
\begin{proof}
As a point $p$ moves along $\alpha$ from $L$ to $H$, the value of $w_\gamma(p)$ increases by 1, when $\alpha$ crosses $\gamma$ from right to left, and decreases by 1 when $\alpha$ crosses $\gamma$ from left to right. Since the changes at self intersections of $\alpha$ cancel each other, the total change in $w_\gamma$ is equal to $\alpha\cdot K$.
\end{proof}

Similarly, for the oriented, immersed closed curve $\gamma_s=k_s\cup \alpha^r$, we have $w_{\gamma_s}(H)-w_{\gamma_s}(L)=\alpha\cdot k_s$. Viewing the values $w_\gamma(H)$ and $w_\gamma(L)$ as normalization terms, we can write the $u$-grading of a state as the difference in winding potentials at the leg and at the head of the knotoid:
\begin{align}\label{equ:windingons2}
    k_s\cdot \alpha - K\cdot \alpha &= -(w_{\gamma_s}(H)-w_{\gamma_s}(L))+(w_\gamma(H)-w_\gamma(L)) \\
    &= (w_{\gamma_s}(L)-w_\gamma(L))-(w_{\gamma_s}(H)-w_\gamma(H)). \nonumber
\end{align}

\subsection{Refinement of the Turaev polynomial for knotoids in \texorpdfstring{$\R^2$}{}}

Note that the winding potential function $w_\gamma$ is not well-defined on $S^2$ without a choice of region in $S^2-\gamma$ where $\infty$ is placed. Nevertheless, the value of ~\eqref{equ:windingons2} is well-defined and is the same for all such choices on $S^2$. On $\R^2$, one is not restricted to take the difference as in ~\eqref{equ:windingons2} to get a well-defined quantity. We will instead use the pair
\begin{equation}\label{equ:sepwindingpair}
    ( w_{\gamma_s}(L)-w_\gamma(L) , w_{\gamma_s}(H)-w_\gamma(H) )
\end{equation}
to give a refined version, for knotoids in $\R^2$, of the Turaev polynomial $T_K(A,u)$ by
\begin{align}\label{equ:anglebracketforr2}
    \accentset{*}{T}_K (A,\ell,h)=\,&(-A^3)^{-wr(K)}\ell\,^{-w_\gamma(L)}h^{-w_\gamma(H)}\times\\
    &\times \sum_{s\in S(K)}A^{\sigma_s}\ell\,^{w_{\gamma_s}(L)}h^{w_{\gamma_s}(H)}(-A^2-A^{-2})^{|s|-1} \in \Z[A^{\pm 1},\ell\,^{\pm 1},h^{\pm 1}]. \nonumber
\end{align}
The choice of a shortcut does not change the values in the pair ~\eqref{equ:sepwindingpair}. Denoting only the summation part in ~\eqref{equ:anglebracketforr2} by $\langle\!\langle K, \alpha \rangle\!\rangle_*$ (as the summation term depends on the choice of $\alpha$ without the normalizing factors in front), we have the usual skein relation $\langle\!\langle \,\raisebox{-0.05cm}{\includegraphics[scale=0.1]{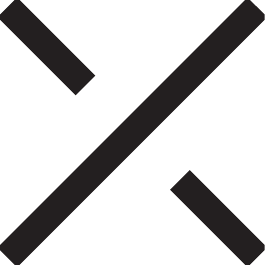}}\,,\alpha \rangle\!\rangle_*=A\langle\!\langle \,\LeftSoft\,\RightSoft\,,\alpha \rangle\!\rangle _*+A^{-1}\langle\!\langle \,\raisebox{-0.03cm}{\rotatebox{90}{\(\LeftSoft\,\RightSoft\)}}\,,\alpha \rangle\!\rangle_*$. Since Reidemeister moves are local deformations of the curves $\gamma_s$, $\gamma$ far from the head/leg, the values of the winding potentials remain unchanged at the head/leg. Then the invariance under Reidemeister moves follow from the skein relation. 

For a knotoid $K$ in $\R^2$, and its image under the inclusion $\R^2\hookrightarrow S^2$ (also denoted by $K$ by abuse of notation), the refined polynomial recovers the Turaev polynomial by the substitution
\begin{equation}
    \accentset{*}{T}_K(A,\ell^{\,a} h^b=u^{a-b})=T_K(A,u).
\end{equation}
\begin{figure}[hbt!]
    \centering
    \begin{subfigure}[b]{.21\textwidth}
    \centering
    \includegraphics[scale=0.3]{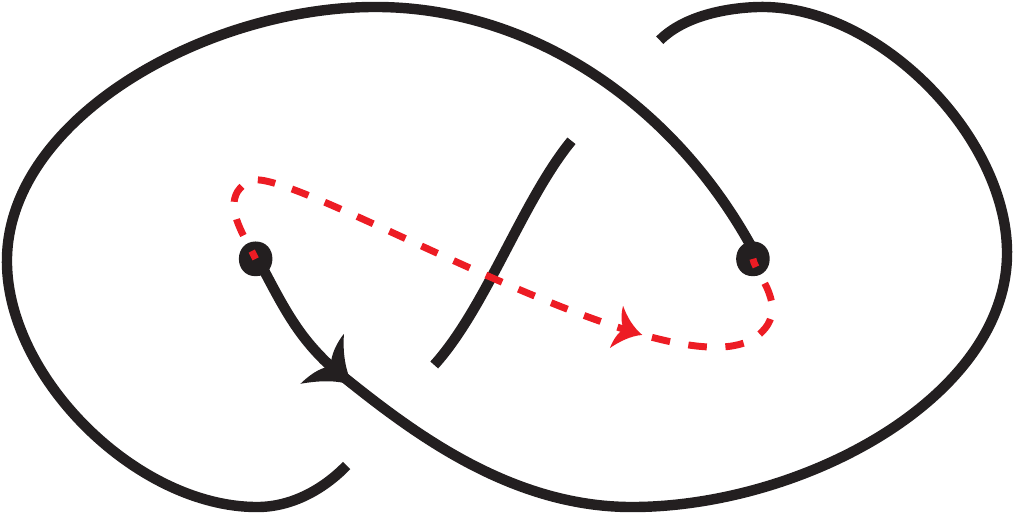}
    \caption{$B_1$}
    \label{fig:bifoil1}
    \end{subfigure}
    \hfill
    \begin{subfigure}[b]{.15\textwidth}
    \centering
    \includegraphics[scale=0.3]{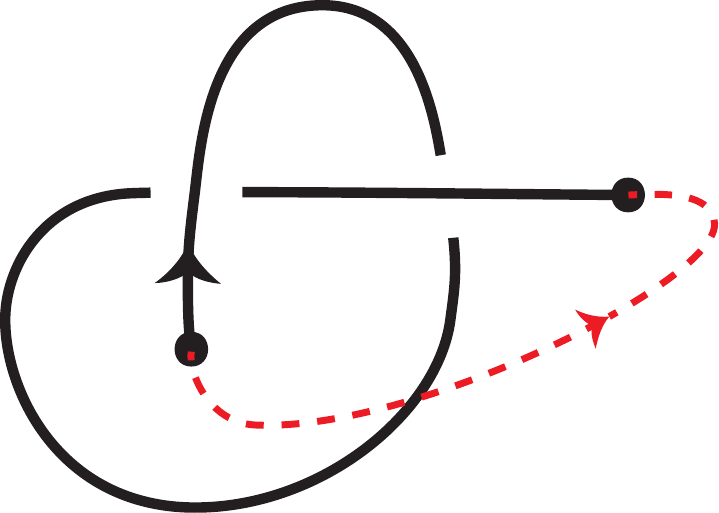}
    \caption{$B_2$}
    \label{fig:bifoil2}
    \end{subfigure}
    \hfill
    \begin{subfigure}[b]{.15\textwidth}
    \centering
    \includegraphics[scale=0.3]{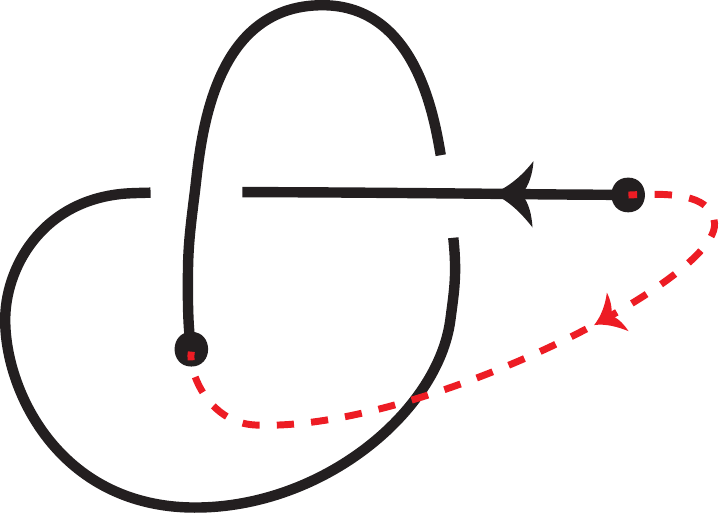}
    \caption{$\textnormal{Rev}(B_2)$}
    \label{fig:bifoil2rev}
    \end{subfigure}
    \hfill
    \begin{subfigure}[b]{.15\textwidth}
    \centering
    \includegraphics[scale=0.3]{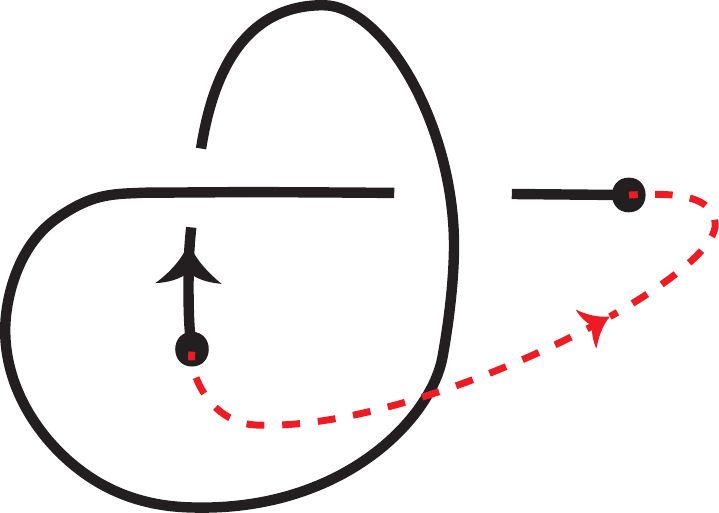}
    \caption{$\textnormal{Mir}(B_2)$}
    \label{fig:bifoil2mir}
    \end{subfigure}
    \hfill
    \begin{subfigure}[b]{.15\textwidth}
    \centering
    \includegraphics[scale=0.3]{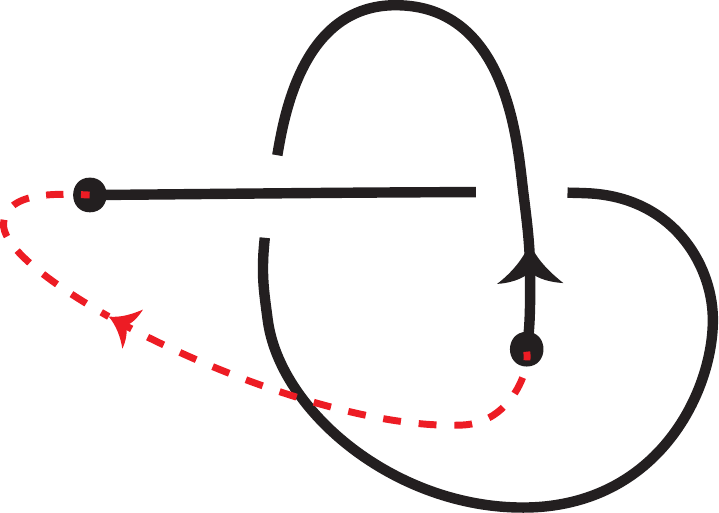}
    \caption{$\textnormal{Sym}(B_2)$}
    \label{fig:bifoil2sym}
    \end{subfigure}
    \caption{Two bifoils $B_1$, $B_2$ with their shortcuts. Reverse, mirror image, and symmetric reflection of $B_2$, respectively. These three elementary involutions of knotoids commute with each other.}
    \label{fig:bifoils}
\end{figure}For example, two knotoids $B_1, B_2\subset \R^2$ in Figure ~\ref{fig:bifoils} are distinguished by their
refined polynomials
\begin{align}
    \accentset{*}{T}_{B_1}&=(-A^{10}-A^6)\,\ell \, h^{-1}+A^6\,h^{-2}+A^6\,\ell^{\,2}+A^4,\\
    \accentset{*}{T}_{B_2}&=(-A^{10}+A^6)\,\ell^{\,2}+A^4,
\end{align}
whereas the substitution $\ell^{\,a}h^b=u^{a-b}$ gives
\begin{equation}
    T_{B_1} = (-A^{10}+A^6)\,u^2+A^4 =T_{B_2}.
\end{equation}
In fact, it is easy to see geometrically that $B_1$ and $B_2$ are equivalent in $S^2$.

\begin{rem}\label{rem:riemannstates}
There is a more intuitive way to think about the values of the winding potential function for a knotoid $K$, and its shortcut $\alpha$. Consider $K$ as a diagram on the complex plane $\CP$ instead of $\R^2$ (or $S^2$). Then the shortcut $\alpha$ can be viewed as a branch cut in $\CP$, along which multiple copies of $\CP-a$ are glued together to form a Riemann surface. For example, the branch cut $(-1,1)\subset \CP$ is used to construct a Riemann surface on which the multi-valued function $\sec^{-1}(z)$ is defined. For our purposes, we will use a Riemann surface $\Sigma$ with infinitely many sheets that are visualized to be stacked vertically. Then the knotoid $K$ and all its states $k_s$ are smooth, oriented curves in $\Sigma$. The number $k_s\cdot \alpha = w_{\gamma_s}(L)-w_{\gamma_s}(H)$ (resp. $K\cdot \alpha = w_\gamma(L)-w_\gamma(H)$) gives the \lq\lq change of level" on the sheets of $\Sigma$ when $k_s$ (resp. $K$) is traced from leg to head. Assuming that $k_s$ and $K$ start on the same sheet of $\Sigma$ at the leg, the difference $k_s\cdot \alpha-K\cdot \alpha$ gives the \lq\lq normalized" change of level when $k_s$ finishes at the head. This quantity is intrinsic to the state $s$ in the sense that it does not depend on the choice of the branch cut in $\CP$, or the choice of diagram for $K$ in $S^2$ to be projected onto $\R^2$ . Now, the values $(w_{\gamma_s}(L)-w_\gamma(L))$, and $(w_{\gamma_s}(H)-w_\gamma(H))$ give the normalized winding numbers of $k_s$ around each pole $L$, and $H$, respectively. The equation ~\eqref{equ:windingons2} then has the natural interpretation that a positive normalized winding of $k_s$ around the leg increases the level in $\Sigma$ by 1, whereas a positive normalized winding around the head decreases the level by 1.
\end{rem}
In Section 10 of ~\cite{Turaev}, Turaev defines a 3-variable polynomial $[\,\,\cdot\,\,]_\circ\in \Z[A^{\pm 1},B,u^{\pm 1}]$ for knotoids in $\R^2$ by
\begin{equation}
    [K]_\circ=(-A^3)^{-wr(K)}u^{K\cdot \alpha}\sum_{s\in S(K)}A^{\sigma_s}u^{k_s\cdot \alpha}(-A^2-A^{-2})^{e_s}B^{f_s},
\end{equation}
where $e_s$ (resp. $f_s$) are the number of closed components not surrounding (resp. surrounding) $k_s$, such that $e_s+f_s=|s|-1$. The refinement above, namely the separation of the normalized winding numbers around each pole, can also be carried out for the polynomial $[K]_\circ$ to get a 4-variable polynomial $[K]_\bullet$ as follows
\begin{align}\label{equ:sqbracketforr2}
    [K]_\bullet(A,B,\ell,h)=&(-A^3)^{-wr(K)} \ell\,^{-w_\gamma(L)}h^{-w_\gamma(H)}\times \\
    &\times \sum_{s\in S(K)}A^{\sigma_s}\ell\,^{w_{\gamma_s}(L)}h^{w_{\gamma_s}(H)}(-A^2-A^{-2})^{e_s}B^{f_s}. \nonumber
\end{align}
Independence of $[K]_\bullet\in \Z[A^{\pm 1},B,\ell^{\,\pm 1},h^{\pm 1}]$ from the choice of the shortcut $\alpha$, and invariance under Reidemeister moves are proved similarly as in the case of $\accentset{*}{T}_K$.

\begin{rem}
It is natural to ask whether it is possible to \lq\lq split" the winding information around the head/leg of knotoids to obtain refined invariants in $S^2$ as we did in this section for knotoids in $\R^2$. The first difficulty one encounters is that the winding numbers are not well-defined on $S^2$, so we work with stereographic projection of knotoids in $S^2$ onto $\R^2$. However, a knotoid in $S^2$ may have non-equivalent projections as knotoids in $\R^2$. In ~\cite{KutluayThesis}, we overcome this difficulty by normalizing the refined Turaev polynomial $\accentset{*}{T}$ of knotoids in $\R^2$ using turning numbers, to obtain a well-defined polynomial invariant on $S^2$. Initially, this invariant and its categorification seem to provide refinements of the Turaev polynomial, and the winding homology, respectively. It turns out, however, that the newly obtained invariants are equivalent to the original invariants, as proved in ~\cite{KutluayThesis}, and thus they only give reformulations of the Turaev polynomial, and the winding homology.
\end{rem}

\end{document}